\documentclass[11pt]{article}

\usepackage[english]{babel}
\usepackage[T1]{fontenc}
\usepackage[utf8]{inputenc}
\usepackage[
  margin=1in,
  includefoot,
  footskip=30pt,
]{geometry}
\usepackage{lmodern}
\usepackage{amsmath}
\usepackage{amssymb}
\usepackage{amsthm}
\usepackage{graphicx}
\usepackage{cite}
\usepackage[dvipsnames]{xcolor}
\usepackage{tikz-cd}
\usepackage{stmaryrd}
\usepackage{mathrsfs}
\usepackage{extarrows}
\usepackage{mathtools}
\usepackage{tensor}
\usepackage{bm}
\usepackage{enumerate}
\usepackage{adjustbox}
\usepackage[mathcal]{eucal}
\usepackage{tabularx}
\usepackage{hyperref}
\hypersetup{
	colorlinks,
	urlcolor={red!55!black},
	citecolor={green!60!black},
	linkcolor={red!55!black}
}
\usepackage{cleveref}
\usepackage[strict]{changepage}

\mathchardef\mhyphen="2D
\def\on{\operatorname}
\definecolor{ao}{rgb}{0.0, 0.5, 0.0}

\makeatletter
\providecommand{\leftsquigarrow}{%
  \mathrel{\mathpalette\reflect@squig\relax}%
}
\newcommand{\reflect@squig}[2]{%
  \reflectbox{$\m@th#1\rightsquigarrow$}%
}
\makeatother

\newtheorem{theorem}{Theorem}[section]
\newtheorem{lemma}[theorem]{Lemma}
\newtheorem{conjecture}[theorem]{Conjecture}
\newtheorem{proposition}[theorem]{Proposition}
\newtheorem{corollary}[theorem]{Corollary}
\newtheorem{introthm}{Theorem}

\theoremstyle{definition}
\newtheorem{construction}[theorem]{Construction}
\newtheorem{definition}[theorem]{Definition}
\newtheorem{notation}[theorem]{Notation}
\newtheorem{remark}[theorem]{Remark}
\newtheorem{example}[theorem]{Example}

\title{Ginzburg algebras of triangulated surfaces and perverse schobers}
\author{Merlin Christ}
\date{\today}

\begin{document}

\maketitle
\abstract{Ginzburg algebras associated to triangulated surfaces provide a means to categorify the cluster algebras of these surfaces. As shown by Ivan Smith, the finite derived category of such a Ginzburg algebra can be embedded into the Fukaya category of the total space of a Lefschetz fibration over the surface. Inspired by this perspective, we provide a description of the unbounded derived category in terms of a perverse schober. The main novelty is a gluing formalism describing the Ginzburg algebra as a colimit of certain local Ginzburg algebras associated to discs. As a first application, we give a new construction of derived equivalences between these Ginzburg algebras associated to flips of an edge of the triangulation. Finally, we note that the perverse schober as well as the resulting gluing construction can also be defined over the sphere spectrum.}

\tableofcontents

\section{Introduction}
Cluster algebras were introduced by Fomin and Zelevinsky \cite{FZ02} as a class of commutative algebras equipped with a combinatorial structure relating different subsets of the algebra called clusters. Since then, there has been a great interest in cluster algebras and their relation to other subjects including Teichm\"uller theory, polyhedral surfaces, representation theory of quivers and aspects of noncommutative algebraic geometry such as Calabi-Yau algebras, Calabi-Yau categories and stability conditions. A survey with many references can be found in \cite{Kel08}, we also refer to the cluster algebra portal \cite{CAP} for further surveys and information regarding cluster algebras.

Relevant for this work is a particular class of cluster algebras associated to oriented marked surfaces equipped with an ideal triangulation introduced in \cite{GSV05,FG06,FG09}, and further studied in \cite{FST08,FT12}. These cluster algebras can be described in two different ways. The first perspective is geometric and provides a description in terms of the decorated Teichm\"uller spaces of the surfaces. The cluster variables arise as lambda lengths, which form the coordinates of the Teichm\"uller space. These lambda lengths satisfy an analogue of the classical Ptolemy relations, which gives rise to the cluster exchange relations. The second perspective makes direct use of the combinatorics of the ideal triangulation. The mutation matrix used to define the cluster algebra arises as the signed adjacency matrix of the ideal triangulation, which counts the number of incidences of the ideal triangles. The resulting algebra does not depend on the choice of ideal triangulation but only on the underlying marked surface. 

This second perspective in particular shows that cluster algebras of marked surfaces can be considered as cluster algebras associated to quivers, which can be categorified via $2$-Calabi-Yau (CY) triangulated categories, called cluster categories, and $3$-CY triangulated categories. To describe the $3$-CY categorification of the cluster algebra associated to a quiver $Q$, one chooses a non-degenerate potential $W$. The $3$-CY categorification is then given by the derived category of the Ginzburg algebra $\mathscr{G}(Q,W)$ associated to the quiver with potential $(Q,W)$. The $2$-CY cluster category can be obtained from the derived category of the Ginzburg algebra via the Verdier quotient $\mathcal{D}(\mathscr{G}(Q,W))^{\on{perf}}/\mathcal{D}(\mathscr{G}(Q,W))^{\on{fin}}$, see \cite{Ami09}. There is also a direct link between the Ginzburg algebras and the combinatorics of the cluster algebras, we refer to \cite{Kel12} for a survey.

To describe the results of this work, we first recall the construction of the quiver $Q_\mathcal{T}^\circ$, and a choice of non-degenerate potential $W_\mathcal{T}$, associated to an ideal triangulation $\mathcal{T}$ of a marked surface ${\bf S}$, see \cite{Lab09,GLS16}. We assume for simplicity that $\mathcal{T}$ has no self-folded triangles. The quiver $Q^\circ_\mathcal{T}$ has as vertices the internal edges of $\mathcal{T}$ and an arrow $a:i\rightarrow j$ for each ideal triangle containing the edges $i,j$ where the edge $j$ follows the edge $i$ in the clockwise order of the edges of the ideal triangle induced by the orientation of the surface. The non-degenerate potential $W_\mathcal{T}=W'_\mathcal{T}+W''_\mathcal{T}\in kQ^\circ_\mathcal{T}$ consists of a part $W'_\mathcal{T}$ which is the sum of the clockwise 3-cycles inscribed in the interior ideal triangles of $\mathcal{T}$ and a part $W''_\mathcal{T}$ which is a sum of counter-clockwise cycles, one for each interior marked point of ${\bf S}$.
 
To above mentioned $2$-CY and $3$-CY categorifications can be described in terms of the combinatorial geometry of $\mathcal{T}$, see \cite{QZ17} and references therein for the $2$-CY cluster category and \cite{Qiu18,QZ19} for the finite part of the derived category of the Ginzburg algebra $\mathscr{G}(Q_\mathcal{T}^\circ,W'_\mathcal{T})$. For us most relevant is Ivan Smith's realization of the finite part of the derived category of $\mathscr{G}(Q_\mathcal{T}^\circ,W_\mathcal{T}')$ as a full subcategory of the Fukaya category of a Calabi-Yau $3$-fold $Y^\circ$ equipped with a Lefschetz fibration $\pi:Y^\circ\rightarrow \Sigma$, see \cite{Smi15}. The surface $\Sigma$ is obtained from ${\bf S}$ by removing all interior marked points, i.e.~$\Sigma= {\bf S}\backslash (M\cap {\bf S}^\circ)$, where ${\bf S}^\circ={\bf S}\backslash \partial {\bf S}$ denotes the interior of ${\bf S}$ and $M$ denotes the set of marked points. Inspired by the geometry of $\pi$, we give in this paper a description of the entire unbounded derived category of the Ginzburg algebra $\mathscr{G}(Q_\mathcal{T}^\circ,W_\mathcal{T}')$ in terms of the global sections of a perverse schober. 

Before we describe our model for $\mathcal{D}(\mathscr{G}(Q_\mathcal{T}^\circ,W_\mathcal{T}'))$, we highlight the relation to a model for the partially wrapped Fukaya categories of graded surfaces or equivalently the derived categories of gentle algebras \cite{HKK17,LP20}. Consider an ideal triangulation of a graded marked surface ${\bf S}$ and the dual ribbon graph $\Gamma$. The Fukaya category of the surface ${\bf S}$ is equivalent to the dg-category of global sections of a constructible cosheaf of dg-categories on the ribbon graph $\Gamma$, see \cite{DK15,HKK17}. The cosheaf description of the Fukaya category categorifies the statement that the middle cohomology $\on{H}_{\Gamma}(\Sigma,\mathbb{Z}[1])$ of the surface $\Sigma$ with support on $\Gamma$ is equivalent to the abelian group of global sections of a constructible cosheaf $\underline{\on{H}}_{\Gamma}(\mathbb{Z}[1])$ on $\Gamma$ whose stalk at a point $x$ is the homology $\on{H}_{\Gamma\cap U}(U,\mathbb{Z}[1])$ of a small neighborhood $x\in U\subset \Sigma$ with support on $\Gamma\cap U$. Our model describes the derived category of the Ginzburg algebra in terms of the global sections of a different constructible cosheaf of dg-categories on $\Gamma$. Denote by $\Gamma^\circ$ the ribbon graph obtained by removing all exterior edges of $\Gamma$. Decategorified, the idea behind our model is to express the middle cohomology of the $3$-fold $Y^\circ$ with support on $\pi^{-1}(\Gamma^\circ)$ in terms of the abelian group of global sections with support on $\Gamma^\circ$ of the perverse pushforward $\pi_*(\mathbb{Z}[3])$ to $\Sigma$, which in turn is equivalent to the global sections with support on $\Gamma^\circ$ of a constructible cosheaf $\underline{\on{H}}_{\Gamma}(\pi_*\mathbb{Z}[3])$ on $\Gamma$.  We will not provide a systematic categorification of the perverse pushforward functor $\pi_*$, but rather provide an explicit description of the categorification of the constructible cosheaf $\underline{\on{H}}_{\Gamma}(\pi_*\mathbb{Z}[3])$. This will be achieved by constructing a perverse schober on the surface that is classified locally, at every critical value of Smith's Lefschetz fibration, by the $\on{Ind}$-complete version of the spherical adjunction
\[ \mathcal{W}(T^*S^2) \longleftrightarrow \mathcal{D}(k)^{\on{perf}}\,.\]
The explicit computability of our model then arises from a concrete algebraic description of this adjunction, as well as the resulting categorification of $\underline{\on{H}}_{\Gamma}(\pi_*\mathbb{Z}[3]))$ in terms of variants of Waldhausen's $\on{S}_\bullet$-construction. A full definition of the notion of a perverse schober on a surface is not yet documented in the literature, we thus introduce a framework for the treatment of perverse schobers on surfaces which are parametrized by ribbon graphs. Our definition of parametrized perverse schober can be seen as a generalization of the approach to topological Fukaya categories of surfaces of \cite{DK18,DK15}, allowing for the treatment of nonconstant coefficients. The main result of this paper is the following.

\begin{introthm}\label{introthm1}
Let $\mathcal{T}$ be an ideal triangulation of an oriented marked surface ${\bf S}$ and consider the dual ribbon graph $\Gamma$. There exists a $\Gamma$-parametrized perverse schober $\mathcal{F}_\mathcal{T}$ whose stable $\infty$-category of global sections with support on $\Gamma^\circ$ satisfies
\[
\mathcal{H}_{\Gamma^\circ}(\Gamma,\mathcal{F}_\mathcal{T})\simeq \mathcal{D}(\mathscr{G}(Q^\circ_\mathcal{T},W'_\mathcal{T}))\,,
\]
i.e.~is equivalent to the unbounded derived $\infty$-category of the Ginzburg algebra $\mathscr{G}(Q_\mathcal{T}^\circ,W'_\mathcal{T})$. 
\end{introthm}

Note that if $\mathcal{T}$ contains no interior marked points, the potential $W'_\mathcal{T}=W_\mathcal{T}$ is non-degenerate. Given an ideal triangulation $\mathcal{T}$ with interior marked points, the potential $W'_\mathcal{T}$ is in general degenerate. In this case, the Ginzburg algebra $\mathscr{G}(Q_\mathcal{T}^\circ,W_\mathcal{T}')$ is not expected to fully capture the cluster combinatorics.

Informally, \Cref{introthm1} can be summarized as the statement that the derived $\infty$-category $\mathcal{D}(\mathscr{G}(Q^\circ_\mathcal{T},W'_\mathcal{T}))$ arises via the gluing of simpler $\infty$-categories. The pieces used in the gluing construction are the derived $\infty$-categories of certain \textit{relative Ginzburg algebras} of $n$-gons. This terminology was suggested by Bernhard Keller in his ICRA 2020 lecture series on relative Calabi-Yau structures. The derived $\infty$-category of a relative Ginzburg algebra also appears as the $\infty$-category of global sections $\mathcal{H}(\Gamma,\mathcal{F}_\mathcal{T})$ of the parametrized perverse schober $\mathcal{F}_\mathcal{T}$ (without any restrictions on the support). The $\infty$-category $\mathcal{H}(\Gamma,\mathcal{F}_\mathcal{T})$ contains $\mathcal{H}_{\Gamma^\circ}(\Gamma,\mathcal{F}_\mathcal{T})\simeq \mathcal{D}(\mathscr{G}(Q_\mathcal{T}^\circ,W_\mathcal{T}')$ as a full subcategory. The passage from all global sections to global sections with support on the interior thus constitutes a loss of information which explains why the non-relative Ginzburg algebras cannot directly be glued. In terms of the underlying cluster algebras, our gluing construction seems to be a special case of the procedure of amalgamation and defrosting of cluster algebras of \cite{FG06a}.

To make the gluing construction of $\mathcal{D}(\mathscr{G}(Q_\mathcal{T}^\circ,W_\mathcal{T}'))$ work, we need to determine the correct way to glue the pieces. Making different choices would lead to different signs of the differentials of the Ginzburg algebra. The total choice of signs is equivalent to a choice of spin structure on the surface $\Sigma={\bf S}\backslash (M\cap {\bf S}^\circ)$, see \Cref{spinsec}. 

The formalism used for the description of the perverse schober $\mathcal{F}_\mathcal{T}$ works not only in the $k$-linear setting but also over the sphere spectrum. Many of our results naturally extend to this more general setting, see \Cref{specsec}. 

In \Cref{sec1.1}, we recall the full definition of the $3$-CY Ginzburg algebra and continue by introducing relative Ginzburg algebras. In \Cref{sec1.2} contains a discussion of parametrized perverse schobers and Smith's results. In \Cref{sec1.3} we describe the gluing construction of the Ginzburg algebra.

\subsection{Relative Ginzburg algebras of triangulated surfaces}\label{sec1.1}

A quiver $Q$ consists of a finite set of vertices, denoted $Q_0$, and a finite set of arrows, denoted $Q_1$, together with source and target maps $s,t:Q_1\rightarrow Q_0$. A quiver is called graded if each arrow carries an integer labeling. Given a graded quiver $Q$, we denote by $kQ$ the graded path algebra over a commutative ring $k$. A potential $W$ for a quiver $Q$ is an element of the cyclic path algebra $kQ^{\on{cyc}}$, meaning the algebra of $k$-linear sums of cyclic paths. 

For the definition of the Ginzburg algebra, due to \cite{Gin06}, we follow \cite{Kel11}. Consider a quiver with potential $(Q,W)$. We denote by $Q'$ the graded quiver with the same set of vertices as $Q$ and graded arrows of the following three kinds.
\begin{itemize}
\item An arrow $a:i\rightarrow j$ in degree $0$ for each $a:i\rightarrow j\in Q_1$.
\item An arrow $a^*:j\rightarrow i$ in degree $1$ for each $a:i\rightarrow j\in Q_1$.
\item An arrow $l_i:i\rightarrow i$ in degree $2$ for each $i\in Q_0$.
\end{itemize}
The cyclic derivative $\partial_a:kQ_{cyc}\rightarrow kQ$ with respect to $a\in Q_1$ is the $k$-linear map taking a cycle $c$ to $\partial_ac=\sum_{c=uav}uv$, where $u,v\in kQ$ are allowed to be lazy paths. We denote the lazy path at a vertex $i\in Q_0$ by $p_i$. We define the Ginzburg algebra $\mathscr{G}(Q,W)$ to be the dg-algebra whose underlying graded algebra is given by the graded path algebra $kQ'$ and whose differential $d$ is determined by the following action on the generators.
\begin{align*}
 a   & \mapsto 0\\
 a^* & \mapsto \partial_{a} W\\
 l_i & \mapsto \sum_{a\in Q_1}p_i[a,a^*]p_i 
\end{align*}
Note that $\mathscr{G}(Q,W)$ is not the completed Ginzburg algebra, as for example considered in \cite{KY11,Smi15}. We will not consider completed Ginzburg algebras in this paper. In terms of the associated derived $\infty$-categories of these dg-algebras this does not mean much of a loss, because the derived $\infty$-category of the completed Ginzburg algebra can be realized as a full subcategory of the derived $\infty$-category of the non-completed Ginzburg algebra. This perspective however neglects the additional topological structure of the completed Ginzburg algebra, see for example the Appendix in \cite{KY11}.

We now introduce a relative version of the Ginzburg algebra $\mathscr{G}(Q_\mathcal{T},W'_\mathcal{T})$ associated to an ideal triangulation $\mathcal{T}$ of an oriented marked surface ${\bf S}$. We define a quiver $Q_\mathcal{T}$ by adapting the definition of the quiver $Q^\circ_\mathcal{T}$ to include the boundary of ${\bf S}$. We let $Q_\mathcal{T}$ be the quiver with a vertex for each edge of $\mathcal{T}$ (including boundary edges) and an arrow $a:i\rightarrow j$ for each ideal triangle containing the edges $i,j$ where the edge $j$ follows the edge $i$ in the clockwise order. If $\mathcal{T}$ contains self-folded triangles, we additionally include an arrow $a:i\rightarrow i$ for each self-folded edge $i$ of $\mathcal{T}$. The quiver $Q_\mathcal{T}$ contains a clockwise $3$-cycle $T(f)$ for each ideal triangle $f$ of $\mathcal{T}$. We define the potential 
\[ \overline{W}'_\mathcal{T}=\sum_f T(f)\in kQ_\mathcal{T}^{\on{cyc}}\,.\]  
We denote by $\tilde{Q}_\mathcal{T}$ the graded quiver with the same set of vertices as $Q_\mathcal{T}$ and graded arrows of the following three kinds.
\begin{itemize}
\item An arrow $a:i\rightarrow j$ in degree $0$ for each $a:i\rightarrow j\in (Q_\mathcal{T})_1$.
\item An arrow $a^*:j\rightarrow i$ in degree $1$ for each $a:i\rightarrow j\in (Q_\mathcal{T})_1$.
\item An arrow $l_{i}:i\rightarrow i$ in degree $2$ for each vertex $i\in (Q_\mathcal{T})_0$ given by an internal edge of $\mathcal{T}$.
\end{itemize}
We define the \textit{relative Ginzburg algebra} $\mathscr{G}_\mathcal{T}$ to be the dg-algebra whose underlying graded algebra is given by the graded path algebra $k\tilde{Q}_\mathcal{T}$ and whose differential is determined by the following action on the generators.
\begin{align*}
 a   & \mapsto 0\\
 a^* & \mapsto \partial_{a} \overline{W}'_\mathcal{T}\\
 l_{i} & \mapsto \sum_{a\in (Q_\mathcal{T})_1}p_i[a,a^*]p_i 
\end{align*}

The relative Ginzburg algebra $\mathscr{G}_\mathcal{T}$ is an example of the more general relative Ginzburg algebras associated to ice quivers with potential, see \cite{Wu21b}. An ice quiver is a quiver equipped with the further datum of a subquiver, whose vertices and arrows are called frozen. The ice quiver underlying $\mathscr{G}_\mathcal{T}$ is given by $Q_\mathcal{T}$, with frozen vertices given by the boundary edges of $\mathcal{T}$ and no frozen arrows. The potential is $\overline{W}_\mathcal{T}'$. 

The quiver $Q_\mathcal{T}^\circ$ is the full subquiver of $Q_\mathcal{T}$ spanned by the vertices corresponding to internal edges. The potential $W_\mathcal{T}'=\sum_{f}T(f)\in (kQ^\circ_\mathcal{T})^{\on{cyc}}$ consists of all $3$-cycles inscribed into internal ideal triangles of $\mathcal{T}$. Note that if the boundary of ${\bf S}$ is empty, then $(Q_\mathcal{T},\overline{W}_\mathcal{T}')=(Q_\mathcal{T}^\circ,W_\mathcal{T}')$ and the relative Ginzburg algebra $\mathscr{G}_\mathcal{T}$ is equivalent to $\mathscr{G}(Q_\mathcal{T}^\circ,W_\mathcal{T}')$. As an example, let ${\bf S}$ be the $3$-gon and $\mathcal{T}$ a triangle. The relative Ginzburg algebra $\mathscr{G}_\mathcal{T}$ is then given by the graded path algebra of the graded quiver 
\begin{equation}\label{quiv1}
\begin{tikzcd}
                                                 & \cdot \arrow[ld, "1"] \arrow[rd, "0", bend left] &                                                  \\
\cdot \arrow[ru, "0", bend left] \arrow[rr, "1"] &                                                    & \cdot \arrow[lu, "1"] \arrow[ll, "0", bend left]
\end{tikzcd}
\end{equation}
with differential $d$ mapping each arrow of degree $1$ to the composite of the two opposite arrows of degree $0$. The Ginzburg algebra $\mathscr{G}(Q_\mathcal{T}^\circ,W_\mathcal{T}')$ of the triangle $\mathcal{T}$ is however zero. 

\Cref{introthm1} extends to relative Ginzburg algebra in the following way. 

\begin{introthm}\label{introthm1.5}
Let $\mathcal{T}$ be an ideal triangulation of an oriented marked surface ${\bf S}$ with dual ribbon graph $\Gamma$. The $\infty$-category of global sections of the parametrized perverse schober $\mathcal{F}_\mathcal{T}$ satisfies
\[
\mathcal{H}(\Gamma,\mathcal{F}_\mathcal{T})\simeq \mathcal{D}(\mathscr{G}_\mathcal{T})\,.
\]
\end{introthm}

In \cite[Section 7.6]{Kel11}, it shown that mutation of quivers with potential induce derived equivalences between the respective Ginzburg algebras. In \cite{Lab09} it is shown that if two ideal triangulations $\mathcal{T},\mathcal{T}'$ are related by a flip of an edge, the associated quivers with potentials $(Q^\circ_\mathcal{T},W_\mathcal{T})$ and $(Q^\circ_{\mathcal{T}'},W_{\mathcal{T}'})$ are related by quiver mutation. In combination these two results show that flips of ideal triangulations induce derived equivalences of the associated Ginzburg algebras. We extend the derived equivalences to the relative Ginzburg algebras. 

\begin{introthm}\label{introthm2}
Let ${\bf S}$ be an oriented marked surface with two ideal triangulations $\mathcal{T},\mathcal{T}'$ related by a flip of an edge $e$ of $\mathcal{T}$. Then there exists an equivalence of $\infty$-categories 
\[ \mu_e: \mathcal{D}(\mathscr{G}_\mathcal{T})\simeq \mathcal{D}(\mathscr{G}_{\mathcal{T}'})\,.\]
\end{introthm}

We will prove \Cref{introthm2} in \Cref{sec6.4} using an intrinsic feature of the theory of parametrized perverse schobers, namely equivalences of global sections induced from contractions of the underlying ribbon graphs.

We thank Bernhard Keller for informing us about an alternative approach to \Cref{introthm2}. A result of Yilin Wu \cite{Wu21b} extends the argument from \cite[Section 7.6]{Kel12} to relative Ginzburg algebras, showing that the mutations of ice quivers with potential of \cite{Pre20} induce derived equivalences between the associated relative Ginzburg algebras. \Cref{introthm2} may then be recovered by additionally extending the results of \cite{Lab09} relating flips of the ideal triangulation and mutations of quivers with potentials to ice quivers.

\subsection{Perverse schobers and Fukaya categories}\label{sec1.2} 

Perverse schobers are a conjectured categorification of the notion of perverse sheaf \cite{KS14}. An approach to the categorification of a perverse sheaf on a disk was suggested in \cite{KS14}. The datum of a perverse sheaf on a disk with a single singularity in the center is equivalent to the datum of a certain quiver diagram; the proposed 'ad-hoc' categorification of the quiver description is a spherical adjunction. In this paper, we extend this ad-hoc categorification to perverse schobers on oriented marked surfaces. We combinatorially describe perverse schober using ribbon graphs. Such a ribbon graph arises as the dual to an ideal triangulation of the marked surface. Given a ribbon graph $\Gamma$ we define a poset $\on{Exit}(\Gamma)$ with 
\begin{itemize}
\item objects the vertices and edges of $\Gamma$,
\item morphisms of the form $v\rightarrow e$ with $v$ a vertex and $e$ an incident edge.
\end{itemize} 
For each $n$-valent vertex $v$ of $\Gamma$ there exists a subposet $\on{Exit}(\Gamma)_{v/}\subset \on{Exit}(\Gamma)$ consisting of the vertex $v$ and the $n$ incident edges. We define a perverse schober $\mathcal{F}$ parametrized by $\Gamma$ to be a functor $\mathcal{F}:\on{Exit}(\Gamma)\rightarrow \on{St}$ into the $\infty$-category of stable $\infty$-categories such that the restriction to $\on{Exit}(\Gamma)_{v/}$ is for every vertex $v$ equivalent to a particular diagram obtained from a spherical adjunction. The exact definition is based on the categorified Dold-Kan correspondence of \cite{Dyc17} and categorifies the 'fractional spin' description of perverse sheaves on a disc of \cite{KS16}. The definition of parametrized perverse schober captures the idea that a perverse schober on a surface is a collection of suitably glued together spherical adjunctions, categorifying the description of perverse sheaves on surfaces given in \cite{KS16}. The $\infty$-category of global sections $\mathcal{H}(\Gamma,\mathcal{F})$ of a parametrized perverse schober $\mathcal{F}$ is defined as the limit of $\mathcal{F}$ in $\on{St}$. Under mild technical assumptions, the global sections of $\mathcal{F}$ are equivalent to a suitable colimit
of the dual to $F$ (left adjoint diagram), which describes a constructible cosheaf, see \Cref{sec4.3}

Given an ideal triangulation $\mathcal{T}$ without self-folded triangles of an oriented marked surface ${\bf S}$, Smith \cite{Smi15} defines a Calabi-Yau $3$-fold $Y$ with an affine conic fibration $\pi:Y\rightarrow {\bf S}$. The relation to Ginzburg algebras is as follows, see \textit{loc.~cit}.
\begin{itemize}
\item The derived category of finite modules over $\mathscr{G}(Q_\mathcal{T}^\circ,W_\mathcal{T}')$ arises as a full subcategory of the derived Fukaya category $\on{Fuk}(Y)$ of $Y$, where $W_\mathcal{T}'$ is the potential of $Q_\mathcal{T}^\circ$ consisting of clockwise $3$-cycles.
\item The derived category of finite modules over $\mathscr{G}(Q^\circ_\mathcal{T},W_\mathcal{T})$ arises as a full subcategory of the derived Fukaya category $\on{Fuk}(Y,b)$ of $Y$ with a twisting background class $b \in H^2(Y,\mathbb{Z}_2)$. Here, $W_\mathcal{T}=W_\mathcal{T}'+W_\mathcal{T}''$ is the potential consisting of clockwise $3$-cycles and counter-clockwise cycles.
\end{itemize}
The geometry of $\pi$ becomes clear when considering its fibers, which are given as follows.
\begin{itemize}
\item The generic fiber of $\pi$ is diffeomorphic to $T^*S^2$.
\item In the interior of each ideal triangle of $\mathcal{T}$, there exists exactly one singular value with singular fiber given by the $2$-dimensional $A_1$-singularity.
\item The fibers of the interior marked points in ${\bf S}$ are given by $\mathbb{C}^2\amalg \mathbb{C}^2$. 
\end{itemize}
We denote by $\Sigma\coloneqq {\bf S}\backslash (M\cap {\bf S}^\circ)$, with ${\bf S}^\circ$ the interior of ${\bf S}$, the surface without the interior marked points and $Y^\circ \coloneqq \pi^{-1}(\Sigma)$. Note that the restriction $\pi|_{Y^\circ}:Y^\circ\rightarrow \Sigma$ of $\pi$ is a Lefschetz fibration.

The twist by the background class $b\in H^2(Y,\mathbb{Z}_2)$ changes signs in the signed count of pseudo-holomorphic curves passing through the fibers of the interior marked points. Without the background class the signed count of such pseudo-holomorphic curves always vanishes so that the derived Fukaya category of $Y^\circ$ is equivalent to the derived Fukaya category of $Y$. The change in the $A_\infty$-structure of the derived Fukaya category of $Y$ induced by the background class $b$ accounts exactly for the difference between the potentials $W'_\mathcal{T}$ and $W_\mathcal{T}$. 

We expect the $\infty$-category of global sections of the parametrized perverse schober $\mathcal{F}_\mathcal{T}$ of \Cref{introthm1} to describe (the $\on{Ind}$-completion of) a partially wrapped Fukaya category of $Y^\circ$. We further expect the global sections with support on $\Gamma^\circ$ (the graph obtained from $\Gamma$ by removing boundary edges) to then correspond to the ($\on{Ind}$-completion of the) wrapped Fukaya category of $Y^\circ$. In the case of the unpunctured $n$-gon, where $Y^\circ=Y$ is the $3$-dimensional $A_{n-3}$-singularity and $Q_\mathcal{T}^\circ$ the $A_{n-3}$-quiver, it is shown in \cite{LU21} that $\mathcal{W}(Y^\circ)\simeq\mathcal{D}(\mathscr{G}(Q_\mathcal{T}^\circ,W_\mathcal{T}'))^{\on{perf}}$, meaning that $\mathcal{H}_{\Gamma^\circ}(\Gamma,\mathcal{F}_\mathcal{T})$ is equivalent to the $\on{Ind}$-completion of the wrapped Fukaya category of $Y^\circ$. 

We describe in \Cref{sec1.3} how the geometry of the Lefschetz fibration manifests itself in the definition of $\mathcal{F}_\mathcal{T}$. We expect that the twisting by the background class $b$ can be described as a deformation of the wrapped Fukaya category. It would be interesting to study the relation between such a deformation and the description in terms of parametrized perverse schobers.

\subsection{The gluing construction of Ginzburg algebras}\label{sec1.3}

We now describe the construction of the perverse schober $\mathcal{F}_\mathcal{T}$ appearing in \Cref{introthm1} and \Cref{introthm1.5}. We assume for simplicity that all ideal triangles of $\mathcal{T}$ are not self-folded. The ribbon graph $\Gamma$ dual to $\mathcal{T}$ parametrizing $\mathcal{F}_\mathcal{T}$ consists of a vertex for each ideal triangle and an edge for each edge of $\mathcal{T}$. Boundary edges of $\mathcal{T}$ correspond to external edges of the ribbon graph. Parametrized perverse schobers can, as can sheaves, be glued. To define $\mathcal{F}_\mathcal{T}$, it thus suffices to define  $\mathcal{F}_\mathcal{T}$ locally at each vertex of $\Gamma$. The local datum at each vertex is a spherical adjunction, which we choose to be
\begin{equation} \label{sphadjeq}
f^*:\mathcal{D}(k)\longleftrightarrow \on{Fun}(S^2,\mathcal{D}(k)):f_*\,,
\end{equation}
where $\on{Fun}(S^2,\mathcal{D}(k))$ is the $\infty$-category of local systems on the $2$-sphere with values in $\mathcal{D}(k)$ and $f^*$ is the pullback functor along $S^2\rightarrow \ast$. This adjunction was shown in \cite{Chr20} to be spherical. 

The $\infty$-category $\on{Fun}(S^2,\mathcal{D}(k))$ is equivalent to the derived $\infty$-category of the polynomial algebra $k[t_1]$ with generator $t_1$ in degree $1$, see \Cref{genlem}. This derived $\infty$-category is by a result of \cite{Abo11} equivalent to the Ind-completion of the wrapped Fukaya of the cotangent bundle $T^*S^2$, which is the generic fiber of the Lefschetz fibration $\pi|_{Y^\circ}$. Under these equivalences, the image $f^*(k)$ corresponds to the Lagrangian zero section of $T^*S^2$. The fibration $\pi|_{Y^\circ}$ has exactly one singular value in each ideal triangle of $\mathcal{T}$, so that, up to homotopy of $\Gamma$, the vertices of $\Gamma$ lie at the singular values of $\pi|_{Y^\circ}$. The singular fibers are given by the $A_1$-singularity. The relation between the geometry of $\pi|_{Y^\circ}$ and the definition of $\mathcal{F}_\mathcal{T}$ can thus be summarized as follows.
\begin{itemize}
\item The wrapped Fukaya category of the generic fiber $T^*S^2$ of $\pi|_{Y^\circ}$ gives rise to the $\infty$-category on the right of \eqref{sphadjeq}. This $\infty$-category describes the generic stalk of $\mathcal{F}_\mathcal{T}$.
\item Each vertex of $\Gamma$ corresponds to a singular value of $\pi|_{Y^\circ}$. The $\infty$-category on the left of \eqref{sphadjeq} describes the categorification of the vector space of vanishing cycles at that singularity of $\pi|_{Y^\circ}$. Since the $A_1$-singularity has a unique vanishing cycle, this $\infty$-category is given by $\mathcal{D}(k)$. 
\item The spherical adjunction $f^*\dashv f_*$ arises from a spherical object, the Lagrangian zero section, in the wrapped Fukaya category of $T^*S^2$ describing the vanishing cycle. 
\end{itemize}
We further note, that the perverse schober only achieves to models the Lefschetz fibration $\pi|_{Y^\circ}$ and not the full fibration $\pi$. The fibers $\mathbb{C}^2\amalg \mathbb{C}^2$ of $\pi$ over the interior marked points of ${\bf S}$ are not encoded in $\mathcal{F}_\mathcal{T}$.

The parametrized perverse schober $\mathcal{F}_\mathcal{T}$ in total corresponds to the datum of a diagram 
\[ \mathcal{F}_\mathcal{T}: \on{Exit}(\Gamma)\rightarrow \on{St}\] in the $\infty$-category $\on{St}$ of stable $\infty$-categories indexed by the poset $\on{Exit}(\Gamma)$, see \Cref{sec1.2}. The computations in \Cref{sec5} show that the parametrized perverse schober $\mathcal{F}_\mathcal{T}$ assigns 
\begin{itemize}
\item to each vertex of $\Gamma_\mathcal{T}$ a stable $\infty$-category equivalent to the derived $\infty$-category of the relative Ginzburg algebra of the $3$-gon, depicted in \eqref{quiv1}. This uses that each vertex of $\Gamma_\mathcal{T}$ is trivalent.
\item to each edge of $\Gamma_\mathcal{T}$ a stable $\infty$-category equivalent to the derived $\infty$-category of the polynomial algebra $k[t_1]$ with generator $t_1$ in degree $1$. Note that $k[t_1]$ is equivalent to the $2$-Calabi-Yau completion of $k$ in the sense of \cite{Kel11}, i.e.~a $2$-dimensional Ginzburg algebra.
\end{itemize}  
The equivalence $\mathcal{H}(\Gamma,\mathcal{F}_\mathcal{T})\simeq \mathcal{D}(\mathscr{G}_\mathcal{T})$ of \Cref{introthm1.5} thus expresses that the derived $\infty$-category of the relative Ginzburg algebra $\mathscr{G}_\mathcal{T}$ is glued from relative Ginzburg algebras of $3$-gons along $2$-dimensional Ginzburg algebras. We further illustrate the gluing construction of $\mathscr{G}_\mathcal{T}$ in two examples in \Cref{sec6.2}.

\subsection*{Notation and conventions}

We follows the notation and conventions of \cite{HTT} and \cite{HA}. In particular, we always use the homological grading.

\subsection*{Acknowledgments}
I wish to thank my supervisor Tobias Dyckerhoff for proposing this topic and for his support and encouragement. I further thank Dylan Allegretti and Bernhard Keller for helpful comments and an anonymous referee for helping improve the readability of the paper. The author acknowledges support by the Deutsche Forschungsgemeinschaft under Germany’s Excellence Strategy – EXC 2121 “Quantum Universe” – 390833306.

\section{Preliminaries}\label{sec2}

This paper is formulated using the language of stable $\infty$-categories. It would in principle be possible to formulate most results in the framework of dg-categories. Our reason to use stable $\infty$-categories is to gain access to the powerful framework developed in \cite{HTT,HA}. As a side effect, we also profit in \Cref{specsec} from the added generality of stable $\infty$-categories over dg-categories. The essential computations in the gluing construction of the Ginzburg algebras are however performed using the category of dg-categories with its quasi-equivalence model structure.

The goal of this section is to review background material on the relation between on the one hand ring spectra, stable $\infty$-categories and their colimits and on the other hand dg-algebras, dg-categories and their homotopy colimits. All material appearing in this section for which we could not find references in the literature is well known to experts. In \Cref{sec2.1,sec2.2} we discuss some generalities on limits and colimits in $\infty$-categories of $\infty$-categories and on $\infty$-categories of modules associated to ring spectra. In \Cref{sec2.3,sec2.4,sec2.5} we relate dg-categories with $\infty$-categories. In \Cref{sec2.6} we discuss semiorthogonal decompositions.

For an extensive treatment of the theory of $\infty$-categories and stable $\infty$-categories we refer to \cite{HTT} and \cite{HA}, respectively.

\subsection{Limits and colimit in \texorpdfstring{$\infty$}{infinity}-categories of \texorpdfstring{$\infty$}{infinity}-categories}\label{sec2.1}

We begin by introducing the following $\infty$-categories of $\infty$-categories.

\begin{definition}
We denote 
\begin{enumerate}
\item by $\on{Cat}_\infty$ the $\infty$-category of $\infty$-categories.
\item by $\on{St}\subset \on{Cat}_\infty$ the subcategory spanned by stable $\infty$-categories and exact functors.
\item by $\on{St}^{\on{idem}}\subset \on{St}$ the full subcategory spanned by idempotent complete stable $\infty$-categories.
\end{enumerate}
An $\infty$-category is called presentable if it is equivalent to the Ind-completion of a small $\infty$-category and admits all colimits\footnote{We always assume all limits and colimits to be small in the sense of \cite{HTT}.}, see \cite[Section 5.5]{HTT}. We further denote
\begin{enumerate}\setcounter{enumi}{3}
\item by $\mathcal{P}r^L\subset \on{Cat}_\infty$ the subcategory spanned by presentable $\infty$-categories and colimit preserving functors.
\item by $\mathcal{P}r^R\subset \on{Cat}_\infty$ the subcategory spanned by presentable $\infty$-categories and accessible and limit preserving functors.
\item by $\mathcal{P}r^L_{\on{St}}\subset \mathcal{P}r^L$ and $\mathcal{P}r^R_{\on{St}}\subset \mathcal{P}r^{R}$ the full subcategories spanned by stable $\infty$-categories.
\end{enumerate}
\end{definition}

We are further interested in $R$-linear $\infty$-categories, where $R$ is an $\mathbb{E}_\infty$-ring spectrum, i.e.~a commutative algebra object in the symmetric monoidal $\infty$-category $\on{Sp}$ of spectra. The $\infty$-category $\mathcal{P}r^L$ also admits the structure of a symmetric monoidal $\infty$-category, see \cite[Section 4.8.1]{HA}. Given an $\mathbb{E}_\infty$-ring spectrum $R$, the $\infty$-category $\on{LMod}_R\in \mathcal{P}r^L$ of left module-spectra over $R$ is an algebra object of $\mathcal{P}r^L$. 

\begin{definition}~  
\begin{enumerate}\setcounter{enumi}{6}
\item Let $R$ be an $\mathbb{E}_\infty$-ring spectrum. The $\infty$-category of $\on{LinCat}_R=\on{LMod}_{\on{LMod}_R}(\mathcal{P}r^L)$ of left modules in $\mathcal{P}r^L$ over $\on{LMod}_{R}$, is called the $\infty$-category of $R$-linear $\infty$-categories.
\end{enumerate}
\end{definition}

\begin{remark}
Though not directly contained in the definition, it can be shown that any $R$-linear $\infty$-category is automatically stable, see \cite[D.1.5]{SAG} for a discussion.
\end{remark}

\begin{remark}
A left-tensoring of an $\infty$-category $\mathcal{M}$ over a monoidal $\infty$-category $\mathcal{C}^\otimes$ is a coCartesian fibration of $\infty$-operads $\mathcal{O}^\otimes \rightarrow \mathcal{LM}^\otimes$ over the left-module $\infty$-operad $\mathcal{LM}^\otimes$, such that there are equivalences of fibers $\mathcal{O}^\otimes_{\langle m\rangle}\simeq \mathcal{M}$ and $\mathcal{O}^\otimes_{\langle a\rangle}\simeq \mathcal{C}^\otimes$. We refer to \cite[Section 4.2.1]{HA} for more details. Objects of $\on{LinCat}_R$ can be identified with stable and presentable $\infty$-categories $\mathcal{C}$ equipped with the datum of a left-tensoring over the symmetric monoidal $\infty$-category $\on{LMod}_R$, such that the tensor product $\mhyphen \otimes_R \mhyphen: \on{LMod}_R\times\mathcal{C} \rightarrow \mathcal{C}$ preserves colimits separately in each variable, see \cite[Appendix D]{SAG}. Let $\mathcal{M}_1,\mathcal{M}_2$ be $R$-linear $\infty$-categories as witnessed by the coCartesian fibrations $\mathcal{O}^\otimes_1,\mathcal{O}^\otimes_2\rightarrow \mathcal{LM}^\otimes$. An $R$-linear functor $\mathcal{M}_1\rightarrow \mathcal{M}_2$ thus corresponds to a morphism of $\infty$-operads $\mathcal{O}^\otimes_1\rightarrow \mathcal{O}^\otimes_2$ over $\mathcal{LM}^\otimes$.
\end{remark}

We now recall in order of appearance results on
\begin{enumerate}[i)]
\item how to compute limits in $\on{Cat}_\infty$,
\item how to compute limits and colimits in $\mathcal{P}r^L$, $\mathcal{P}r^L_{\on{St}}$ and $\mathcal{P}r^R$, $\mathcal{P}r^R_{\on{St}}$,
\item how to compute limits and colimits in $\on{LinCat}_R$ and
\item how to compute limits and colimits in $\on{St}^{\on{idem}}$.
\end{enumerate}

i) There is a general formula for limits in $\on{Cat}_\infty$. Let $D:Z\rightarrow \on{Set}_\Delta$ be a diagram taking values in $\infty$-categories. Consider the coCartesian fibration $p:X\rightarrow Z$ classified by $D$. The limit $\infty$-category $\on{lim}D$ is equivalent to the $\infty$-category of coCartesian sections\footnote{We call a section $s:Z\rightarrow X$ of a coCartesian fibration $p:X\rightarrow Z$ coCartesian if for all edges $e\in Z_1$, the edge $s(e)\in X_1$ is $s$-coCartesian.} of $p$, see \cite[3.3.3.2]{HTT}. If $Z$ is the nerve of a $1$-category, the above model for computing limits in $\on{Cat}_\infty$ can be described more explicitly. We can use the relative nerve construction, see \cite[3.2.5.2]{HTT}, for the coCartesian fibration classified by $D$, which is very explicitly defined. We denote this model for the coCartesian fibration by $p:\Gamma(D)\rightarrow K$ and call it the (covariant) Grothendieck construction. A more detailed introduction to the relative nerve construction can be found in Section 1.2 of \cite{Chr20}.

ii) One of the nice features of presentable $\infty$-categories is that there is an $\infty$-categorical adjoint functor theorem, which states that a functor between presentable $\infty$-categories admits a right adjoint if and only if it preserves all colimits and admits a left adjoint if and only if it is accessible\footnote{A functor between presentable $\infty$-categories being accessible reduces to the condition of preserving filtered colimits.} and preserves all limits. There thus exists an adjoint equivalence of $\infty$-categories \[\on{radj}:\mathcal{P}r^L\simeq \left(\mathcal{P}r^R\right)^{op}:\on{ladj}^{op}\,,\] with the functors $\on{radj},\on{ladj}$ acting as the identity on objects. The functor $\on{radj}$ maps a colimit preserving functor to its right adjoint and the functor $\on{ladj}$ maps an accessible and limit preserving functor to its left adjoint. The adjoint equivalence $\on{radj}\dashv \on{ladj}$ also restricts to an adjoint equivalence between $\mathcal{P}r^L_{\on{St}}$ and $(\mathcal{P}r_{\on{St}}^R)^{op}$. The equivalences $\on{radj},\on{ladj}$ preserve all limits and colimits, so that we can exchange the computations of limits and colimits of diagrams of (stable) presentable $\infty$-categories. For the computation of limits, we can use i) and the fact that that the inclusions $\mathcal{P}^L_{\on{St}}\subset \mathcal{P}r^L\subset \on{Cat}_\infty$ and $\mathcal{P}r^R_{\on{St}}\subset \mathcal{P}r^R\subset \on{Cat}_\infty$ preserve all limits. 

iii) The computation of limits and colimits of $R$-linear $\infty$-categories reduces to the computation of limits and colimits in $\mathcal{P}r^L$, because the forgetful functor $\on{LMod}_{\on{LMod}_R}(\mathcal{P}r^L)\rightarrow \mathcal{P}r^L$ preserves all limits and colimits, see \cite[4.2.3.1,4.2.3.5]{HA}.

iv) The inclusion functor $\on{St}^{\on{idem}}\subset \on{Cat}_\infty$ preserves all limits. The computation of colimits of idempotent stable $\infty$-categories can be related to the computation of colimits of presentable stable $\infty$-categories via the colimit preserving Ind-completion functor $\on{Ind}:\on{St}^{\on{idem}}\rightarrow \mathcal{P}r^L_{\on{St}}$. 
Given an $\infty$-category $\mathcal{C}\in \mathcal{P}r^L_{\on{St}}$, we denote by $\mathcal{C}^c\in \on{St}^{\on{idem}}$ its full subcategory of compact objects. Note that for $\mathcal{C}\in \on{St}^{\on{idem}}$, there exists an equivalence $\on{Ind}(\mathcal{C})^c\simeq \mathcal{C}$.

\subsection{Modules over ring spectra}\label{sec2.2}

Consider the symmetric monoidal $\infty$-category $\on{Sp}$ of spectra. $\on{Sp}$ is a stable and presentable $\infty$-category. An $\mathbb{E}_1$-ring spectrum is an object of $\on{Alg}(\on{Sp})$, the $\infty$-category of (coherently associative) algebra objects in $\on{Sp}$. For every such $\mathbb{E}_1$-ring spectrum $R$, there is a stable and presentable $\infty$-category $\on{RMod}_R$ of right $R$-modules in $\on{Sp}$. If $R$ can be enhanced to a commutative algebra object of $\on{Sp}$, i.e.~an $\mathbb{E}_\infty$-ring spectrum, then $\on{RMod}_R$ inherits the structure of a symmetric monoidal $\infty$-category. In this case, we can form the $\infty$-category $\on{Alg}(\on{RMod}_R)$ of algebra objects in $\on{RMod}_R$. Given $A\in \on{Alg}(\on{RMod}_R)$, we can again form the $\infty$-category $\on{RMod}_A(\on{RMod}_R)$ of right $A$-modules in $\on{RMod}_R$. Alternatively, we can also consider the $\mathbb{E}_1$-ring spectrum $\xi(A)\in \on{Alg}(\on{Sp})$ underlying $A$ obtained as follows. We consider the forgetful functor $\on{RMod}_R\rightarrow \on{Sp}$, mapping a right $R$-module to the underlying spectrum. This functor extends to a functor $\xi:\on{Alg}(\on{RMod}_R)\rightarrow \on{Alg}(\on{Sp})$, which we apply to $A$. We can form the $\infty$-category of right modules $\on{RMod}_{\xi(A)}$ over $\xi(A)$. We will show in \Cref{modcor} that this does not yield a further $\infty$-category, there exists an equivalence of $\infty$-categories 
\[\on{RMod}_{A}(\on{RMod}_R)\simeq \on{RMod}_{\xi(A)}\,.\]

Let $\mathcal{D}$ be a stable $\infty$-category and consider any object $X\in \mathcal{D}$. We can find an $\mathbb{E}_1$-ring spectrum $\on{End}(X)\in \on{Alg}(\on{Sp})$, called the endomorphism algebra, with the following properties, see \cite[7.1.2.2]{HA}. 
\begin{itemize}
\item $\pi_n\on{End}(X) \simeq \pi_0\on{Map}_{\mathcal{D}}(X[n],X)$ for all $n\in \mathbb{Z}$.
\item The induced ring structure of $\pi_*\on{End}(X)$ is determined by the composition of endomorphisms in the homotopy category $\on{Ho}(\mathcal{D})$.
\end{itemize}
The algebra object $\on{End}(X)$ is an endomorphism object of $X$ in the sense of \cite[Section 4.7.1]{HA} and its existence expresses the enrichment of the stable $\infty$-category $\mathcal{D}$ in spectra. 

Assume that the stable $\infty$-category $\mathcal{D}$ is also presentable. An object $X\in \mathcal{D}$ is called a compact generator if 
\begin{itemize}
\item $X$ is compact, i.e.~$\on{Map}_{\mathcal{D}}(X,\mhyphen)$ commutes with filtered colimits and
\item an object $Y\in \mathcal{D}$ is zero if and only if $\on{Map}_{\mathcal{D}}(X,Y[i])\simeq \ast$ for all $i\in \mathbb{Z}$.
\end{itemize}
The importance of this notion is that if $X$ is a compact generator, there exists an equivalence of $\infty$-categories $\mathcal{D}\simeq \on{RMod}_{\on{End}(X)}$, see \cite[7.1.2.1]{HA}.\medskip

We now restrict to $R$-linear $\infty$-categories where $R$ is an $\mathbb{E}_\infty$-ring spectrum. The most important case will be where $R=k$ is a commutative ring. Suppose that $\mathcal{D}$ is an $R$-linear $\infty$-category and $X\in \mathcal{D}$ a compact generator. \Cref{genlem1} shows we can lift $\on{End}(X)$ along the forgetful functor $\xi:\on{Alg}(\on{RMod}_R)\rightarrow \on{Alg}(\on{Sp})$ to an algebra object in $\on{RMod}_R$.

\begin{lemma}\label{genlem1}
Let $R$ be an $\mathbb{E}_\infty$-ring spectrum. Let $\mathcal{C}$ be a stable and presentable $R$-linear $\infty$-category with a compact generator $X$. Then there exists an algebra object $\on{End}_R(X)\in \on{Alg}(\on{RMod}_R)$ and an equivalence of $R$-linear $\infty$-categories
\begin{equation}\label{geneq} 
\mathcal{C}\simeq \on{RMod}_{\on{End}_R(X)}(\on{RMod}_R)\,.
\end{equation}
The algebra object $\on{End}_R(X)$ is mapped under the functor $\xi:\on{Alg}(\on{RMod}_R)\rightarrow \on{Alg}(\on{Sp})$ to the endomorphism algebra \makebox{$\on{End}(X)\in \on{Alg}(\on{Sp})$.} 
\end{lemma}
\begin{proof}
The left tensoring of $\mathcal{C}$ over $R$ determined an $R$-linear functor $\mhyphen \otimes_R X:\on{RMod}_R\rightarrow \mathcal{C}$. By the adjoint functor theorem, the functor admits a right adjoint $G$. We denote $\on{End}_R(X):=G(X)\in \on{RMod}_R$. The existence of lift of $\on{End}_R(X)$ to $\on{Alg}(\on{RMod}_R)$ and the existence of the equivalence \eqref{geneq} follow from \cite[4.8.5.8]{HA}, compare also to the proof of \cite[7.1.2.1]{HA}. The right adjoint of the composite functor 
\[ \on{Sp}\xlongrightarrow{\mhyphen \otimes R}\on{RMod}_R\xlongrightarrow{\mhyphen \otimes_R X} \mathcal{C}\]
maps $X$ to the endomorphism object $\on{End}(X)$. By the universal property of $\on{End}(X)$ and $X\in \mathcal{C}\simeq \on{RMod}_{\xi(\on{End}_R(X))}(\on{Sp})$, there exists a morphism $\xi(\on{End}_R(X))\rightarrow \on{End}(X)$ in $\on{Alg}(\on{Sp})$, which is an equivalence on underlying spectra and thus an equivalence of $\mathbb{E}_1$-ring spectra.
\end{proof}

\begin{remark}\label{rendrem}
In the setting of \Cref{genlem1}, the algebra object $\on{End}_R(X)$ is an endomorphism object of $X$ in the $\infty$-category $\mathcal{C}$ considered as left tensored over $\on{RMod}_R$. We call $\on{End}_R(X)$ the $R$-linear endomorphism algebra of $X$.
\end{remark}

\begin{corollary}\label{modcor}
Let $R$ be an $\mathbb{E}_\infty$-ring spectrum and $A\in \on{Alg}(\on{RMod}_R)$. Then there exists an equivalence of $\infty$-categories
\begin{equation}\label{modeqeq2} 
\on{RMod}_A(\on{RMod}_R)\simeq \on{RMod}_{\xi(A)}\,,
\end{equation}
where $\xi:\on{Alg}(\on{RMod}_R)\rightarrow \on{Alg}(Sp)$ denotes the forgetful functor.
\end{corollary}

\begin{proof}
The $\infty$-category $\on{RMod}_A(\on{RMod}_R)$ is presentable by \cite[4.2.3.7]{HA}, stable by \cite[7.1.1.4]{HA} and left-tensored over $\on{RMod}_R$ by \cite[Section 4.3.2]{HA}. Consider the monadic adjunction $\mhyphen \otimes A :\on{RMod}_R\leftrightarrow \on{RMod}_A(\on{RMod}_R):G$. The adjunction and that $G$ is conservative and accessible imply that $A$ is a compact generator. The $R$-linear endomorphism algebra of $A\in \on{RMod}_A(\on{RMod}_R)$ is given by $A\in \on{Alg}(\on{RMod}_R)$. The statement thus follows from the second part of \Cref{genlem1} and \cite[7.1.2.1]{HA}.
\end{proof}

Let $R$ be an $\mathbb{E}_\infty$-ring spectrum. We end this section with a brief discussion of the relation between colimits of algebra objects in $\on{RMod}_R$ and the colimits of the corresponding $\infty$-categories of right modules in $\on{LinCat}_R$. There is a functor $\theta: \on{Alg}(\on{RMod}_R)\rightarrow \on{LinCat}_R$ that assigns to an algebra object $A\in \on{Alg}(\on{RMod}_R)$ the $\infty$-category $\on{RMod}_A(\on{RMod}_R)$, see \cite[section 4.8.3]{HA}. The functor $\theta$ assigns to an edge $\phi:A\rightarrow B$ in $\on{Alg}(\on{RMod}_R)$ the relative tensor product 
\[\theta(\phi)=\mhyphen \otimes_{A}B:\on{RMod}_A(\on{RMod}_R)\longrightarrow \on{RMod}_B(\on{RMod}_R) \]
using the right $A$-module structure on $B$ provided by $\phi$. For all $\phi:A\rightarrow B$, the functor $\theta(\phi)$ admits a right adjoint, given by the pullback functor $\phi^*:\on{RMod}_B(\on{RMod}_R)\rightarrow \on{RMod}_A(\on{RMod}_R)$ along $\phi$, see \cite[4.6.2.17]{HA}. The functor $\theta$ preserves colimits indexed by contractible simplicial sets (i.e.~simplicial sets whose geometric realization is a contractible space), most notably pushouts.

\subsection{Differential graded categories and their modules}\label{sec2.3}

Let $k$ be a commutative ring. A $k$-linear dg-category is a 1-category enriched in the $1$-category $\on{Ch}(k)$ of chain complexes of $k$-modules. Given a dg-category $C$ and two objects $x,y\in C$, we write $\on{Hom}_C(x,y)$ or $\on{Hom}(x,y)$ for the mapping complex. We consider dg-algebras as dg-categories with a single object.  

\begin{definition}\label{moddef}
Let $A$ and $B$ be $k$-linear dg-algebras.
\begin{itemize}
\item A left $A$-module $M$ is a graded left module over the graded algebra underlying $A$ equipped with a differential $d_M$ such that 
\[ d_M(a.m)=d_A(a).m+(-1)^{\on{deg}(a)}a.d_M(m)\]
for all $a\in A$ and $m\in M$.
\item A right $A$-module $M$ is a graded right module over the graded algebra underlying $A$ equipped with a differential $d_M$ such that \[d_M(m.a)=d_M(m).a+(-1)^{\on{deg}(m)}m.d_A(a)\]
for all $a\in A$ and $m\in M$. We also refer to right $A$-modules simply as $A$-modules.  
\item An $A$-$B$-bimodule $M$ is a graded bimodule over the graded algebras underlying $A$ and $B$ equipped endowed with a differential $d_M$, which exhibits $M$ as a left $A$-module and a right $B$-module. If $A=B$, we call $M$ an $A$-bimodule.
\end{itemize}
\end{definition}

\begin{remark}\label{sgnrl}
Let $M$ be an $A$-$B$-bimodule with differential $d_M$. The shifted $A$-$B$-bimodule $M[1]$ can be described as follows.
\begin{itemize}
\item The differential is $-d_M$.
\item The left action $._{[1]}$ of $a\in A$ on $m\in M[1]$ is given by $a._{[1]}m=(-1)^{\on{deg}(a)}a.m$, where $a.m$ denotes the left action of $a\in A$ on $m\in M$.
\item The right action $._{[1]}$ of $b\in B$ on $m\in M[1]$ is given by $m._{[1]}b=m.b$, where $m.b$ denotes the right action of $b\in B$ on $m\in M$.
\end{itemize} 
\end{remark}

We can identify left $A$-modules with dg-functors $A\rightarrow \on{Ch}(k)$, right $A$-modules with dg-functors $A^{op}\rightarrow \on{Ch}(k)$ and $A$-$B$-bimodules with dg-functors $A\otimes B^{\on{op}}\rightarrow \on{Ch}(k)$. The following definition is thus consistent with \Cref{moddef}.

\begin{definition}
Let $C$ be a dg-category. We call a dg-functor $C^{op}\rightarrow \on{Ch}(k)$ a right $C$-module. We denote by $\on{dgMod}(C)$ the dg-category of right $C$-modules.
\end{definition}

\begin{remark}
Given any dg-category $C$, the dg-category $\on{dgMod}(C)$ is pretriangulated, with distinguished triangles of the form $x\xrightarrow{a} y\rightarrow \on{cone}(a)$. 
\end{remark}

Given a dg-category $C$ and an object $x\in C$, we denote by $\on{End}^{\on{dg}}(x)$ the endomorphism dg-algebra with underlying chain complex given by $\on{Hom}_C(x,x)$ and algebra structure determined by the composition of morphisms in $C$. 

\begin{lemma}\label{dgmodlem}
Let $C$ be a dg-category with finitely many objects $x_1,\dots,x_n$. Then there exists an equivalence of dg-categories $\on{dgMod}(C)\simeq \on{dgMod}(\on{End}^{\on{dg}}(\bigoplus_{i=1}^n x_i)$, where $\on{End}^{\on{dg}}(\bigoplus_{i=1}^nx_i)$ is the endomorphism dg-algebra of $\bigoplus_{i=1}^n x_i$ in $\on{dgMod}(C)$. 
\end{lemma}

\begin{proof}
This follows directly from spelling out the datum of a right module over $C$ and over $\on{End}^{\on{dg}}(\bigoplus_{i=1}^nx_i)$.
\end{proof}

\subsection{A model for the derived \texorpdfstring{$\infty$}{infinity}-category of a dg-algebra}\label{sec2.4}

Let $A$ be a $k$-linear dg-algebra. Starting with the dg-category $\on{dgMod}(A)$, we can form the $1$-category $\on{dgMod}(A)_0$, with the same objects as $\on{dgMod}(A)$ and with mapping sets given by the $0$-cycles. This $1$-category admits the projective model structure, where the weak equivalences are given by quasi-isomorphisms and the fibrations are given by degree-wise surjections. All objects of $\on{dgMod}(A)_0$ are fibrant. A description of the cofibrant objects in $\on{dgMod}(A)_0$ can be found for example in \cite{BMR14}, where they are called $q$-semi-projective objects. A right $A$-module $M$ is cofibrant if and only if
\begin{itemize}
\item the ungraded module $\bigoplus_{i\in \mathbb{Z}}M_i$ is a projective right module over the ungraded algebra $\bigoplus_{i\in \mathbb{Z}}A_i$ and
\item for all acyclic right $A$-modules $N$, the mapping complex $\on{Hom}_A(M,N)$ is acyclic.
\end{itemize}
If $A=k$ is a commutative ring, the cofibrant objects are the complexes of projective $k$-modules. We denote by $\on{dgMod}(A)^\circ\subset \on{dgMod}(A)$ the full dg-subcategory spanned by fibrant-cofibrant objects. 
We call the dg-nerve $\mathcal{D}(A)\coloneqq N_{\on{dg}}(\on{dgMod}(A)^\circ)$ the (unbounded) derived $\infty$-category of $A$.

Before we can further discuss the properties of $\mathcal{D}(A)$, we need to briefly discuss localizations of $\infty$-categories. 

\begin{definition} \label{reflocdef}
A functor $f:\mathcal{C}\rightarrow \mathcal{C}'$ between $\infty$-categories is a reflective localization if $f$ has a fully faithful right adjoint. 
\end{definition}

In \cite{HTT}, localizations in the sense of \Cref{reflocdef} are simply called localizations. We are however interested in a more general class of localizations, which can be characterized by the following universal property.

\begin{definition}\label{locdef}
Let $\mathcal{C}$ be an $\infty$-category and let $W$ be a collection of morphisms in $\mathcal{C}$. We call an $\infty$-category $\mathcal{C}'$ the $\infty$-categorical localization of $\mathcal{C}$ at $W$ if there exists a functor $f:\mathcal{C}\rightarrow \mathcal{C}'$, such that, for every $\infty$-category $\mathcal{D}$, composition with $f$ induces a fully faithful functor 
\[ \chi: \on{Fun}(\mathcal{C}',\mathcal{D})\rightarrow \on{Fun}(\mathcal{C},\mathcal{D})\,,\]
whose essential image consists of those functors $F:\mathcal{C}\rightarrow \mathcal{D}$ for which $F(\alpha)$ is an equivalence in $\mathcal{D}$ for all $\alpha \in W$. In that case, we also write $\mathcal{C}'=\mathcal{C}[W^{-1}]$. 
\end{definition} 

It is shown in \cite[5.2.7.12]{HTT}, that reflective localizations are localizations in the sense of \Cref{locdef}. If the collection of morphisms $W$ is closed under homotopy and composition and contains all equivalences in $\mathcal{C}$, we can regard $\mathcal{C}[W^{-1}]$ as a fibrant replacement of $(\mathcal{C},W)$ in the model category of marked simplicial sets, see also the discussion in the beginning of \cite[Section 4.1.7]{HA}.

Our first goal in this section is to prove the following analogue of \cite[1.3.5.15]{HA}, relating the derived $\infty$-category of $A$ with the $\infty$-categorical localization of $\on{dgMod}(A)_0$ at the collection of quasi-isomorphisms.

\begin{proposition}\label{locprop}
Let $A$ be a dg-algebra and let $W$ denote the collection of quasi-isomorphisms. There exists an equivalence of $\infty$-categories
\[ \mathcal{D}(A)\simeq N(\on{dgMod}(A)_0)[W^{-1}]\,.\]
\end{proposition}

Given a model category $C$, the $\infty$-categorical localization of $N(C)$ at the collection of weak equivalences is called the $\infty$-category underlying $C$. We refer to \cite{Hin16} for general background. \Cref{locprop} thus shows that the derived $\infty$-category of $A$ is the $\infty$-category underlying the model category $\on{dgMod}(A)_0$.

For the proof of \Cref{locprop} we need the following two lemmas.

\begin{lemma}\label{loclem1}
Let $A$ be a dg-algebra. The inclusion functor $N(\on{dgMod}(A)_0)\rightarrow N_{\on{dg}}(\on{dgMod}(A))$ induces an equivalence of $\infty$-categories 
\[N(\on{dgMod}(A)_0)[H^{-1}]\rightarrow N_{\on{dg}}(\on{dgMod}(A))\,,\]
where $H$ is the collection of chain homotopy equivalences.  
\end{lemma}

\begin{proof}
The proof of \cite[1.3.4.5]{HA} applies verbatim.
\end{proof}

\begin{lemma}\label{loclem2}
Let $A$ be a dg-algebra. There exists an equivalence of $\infty$-categories 
\[ N_{\on{dg}}(\on{dgMod}(A)^\circ)\simeq N_{\on{dg}}(\on{dgMod}(A))[W^{-1}]\,.\]
\end{lemma}
\begin{proof}
We adapt the proofs of \cite[1.3.4.6, 1.3.5.12]{HA}. We show that the inclusion functor 
\[ i:\on{N}_{\on{dg}}(\on{dgMod}(A)^\circ)^{\on{op}}\rightarrow \on{N}_{\on{dg}}(\on{dgMod}(A))^{\on{op}}\] 
admits a left adjoint which exhibits $\on{N}_{\on{dg}}(\on{dgMod}(A)^\circ)^{\on{op}}$ as a reflective localization at the collection of quasi-isomorphisms. Note that any functor is a localization if and only if the opposite functor is a localization. We thus conclude that $\on{N}_{\on{dg}}(\on{dgMod}(A)^\circ)$ is equivalent as an $\infty$-category to the localization of $\on{N}_{\on{dg}}(\on{dgMod}(A))$ at the collection of quasi-isomorphisms.  

To verify that $i^{\on{op}}$ is a reflective localization, we need to show that it admits left adjoint $G:\on{N}_{\on{dg}}(\on{dgMod}(A))^{\on{op}}\rightarrow \on{N}_{\on{dg}}(\on{dgMod}(A)^\circ)^{\on{op}}$. To show that $W$ is the collection of quasi-isomorphisms, we need to show by \cite[5.2.7.12]{HTT} that any edge $e:M\rightarrow N$ in $\on{N}_{\on{dg}}(\on{dgMod}(A))^{op}$ is a quasi-isomorphism if and only if $G(e)$ is an equivalence. Consider a trivial fibration $f:Q'\rightarrow Q$ in $\on{dgMod}(A)$ given a cofibrant replacement and any $P\in \on{dgMod}(A)^\circ$. \cite[5.2.7.8]{HTT} shows the existence of $G$, provided that the composition with $f$ induces an isomorphism of spaces
\begin{equation*}
\on{Map}_{\on{N}_{\on{dg}}(\on{dgMod}(A)^\circ)}(P,Q')\rightarrow \on{Map}_{\on{N}_{\on{dg}}(\on{dgMod}(A))}(P,Q)\,.
\end{equation*}
We deduce this from the assertion that composition with $f$ induces a quasi-isomorphism 
\begin{equation}\label{alphaeq} 
\alpha:\on{Hom}_{\on{dgMod}(A)}(P,Q')\rightarrow \on{Hom}_{\on{dgMod}(A)}(P,Q)\,.
\end{equation}
The surjectivity of $\alpha$ follows from the lifting property of the cofibration $0\rightarrow P$ with respect to trivial fibrations. The kernel of $\alpha$ is given by $\on{Hom}_{\on{dgMod}(A)}(P,\on{ker}(f))$. Using that $f$ is a quasi-isomorphism, we deduce that $\on{ker}(f)$ is acyclic. The contractibility of the kernel of $\alpha$ thus follows from property of $P$ being cofibrant. We can thus deduce the existence of $G$. We note that $G$ is pointwise given by choosing a cofibrant replacement. Consider an edge $e:M\rightarrow N$ in $\on{N}_{\on{dg}}(\on{dgMod}(A))^{op}$. If $e$ is a quasi-isomorphism, it follows from Whitehead's theorem for model categories that $G(e)$ is an equivalence. If $G(e)$ is an equivalence, we have the following commutative diagram in $\on{N}_{\on{dg}}(\on{dgMod}(A))$. 
\[
\begin{tikzcd}
G(M) \arrow[r, "G(e)"] \arrow[d] & G(N) \arrow[d] \\
M \arrow[r, "e"]                 & N             
\end{tikzcd}
\]
The vertical edges and the upper horizontal edge are quasi-isomorphisms. It follows that $e$ is also a quasi-isomorphism.
\end{proof}

\begin{proof}[Proof of \Cref{locprop}]
By \Cref{loclem1,loclem2}, there exists an equivalence of $\infty$-categories 
\[ \left( N(\on{dgMod}(A))[H^{-1}]\right)[W^{-1}]\simeq \mathcal{D}(A)\,.\] 
Using that $H\subset W$, the statement follows.
\end{proof}

Let $k$ be a commutative ring. The symmetric monoidal structure of the $1$-category $\on{Ch}(k)$ can be used to also endow the $\infty$-category $\mathcal{D}(k)$ with a symmetric monoidal structure. As shown in \cite[7.1.4.6]{HA} there exists an equivalence of $\infty$-categories
\begin{equation}\label{equiveq1}
\on{N}(\on{Alg}(\on{Ch}^{\otimes}(k)))[W^{-1}] \simeq \on{Alg}(\mathcal{D}(k))\,.
\end{equation}
The left side of \eqref{equiveq1} is the $\infty$-categorical localization of the nerve of the $1$-category of dg-algebras at the collection of quasi-isomorphisms. The right side of \eqref{equiveq1} is the $\infty$-category of algebra objects in $\mathcal{D}(k)$. The equivalence \eqref{equiveq1} expresses that every dg-algebra can be considered as an algebra object in $\mathcal{D}(k)$ and that every algebra object in $\mathcal{D}(k)$ can be obtained this way (meaning it can be rectified). Unless stated otherwise, we will omit the identification \eqref{equiveq1} and consider dg-algebras as algebra objects in the symmetric monoidal $\infty$-category $\mathcal{D}(k)$. 

We can consider $k$ also as an $\mathbb{E}_\infty$-ring spectrum. The $\infty$-category $\on{RMod}_k$ of right modules over $k$ thus inherits a symmetric monoidal structure. The $\infty$-categories $\mathcal{D}(k)$ and $\on{RMod}_k$ are equivalent as symmetric monoidal $\infty$-categories, see \cite[7.1.2.13]{HA}.\medskip

Let $A$ be a $k$-linear dg-algebra and $X$ a cofibrant $A$-module. Consider the Quillen adjunction 
\begin{equation}\label{adj2} 
\mhyphen \otimes_k^{\on{dg}} X:\on{dgMod}(k)\leftrightarrow \on{dgMod}(A):\on{Hom}_A(X,\mhyphen)\,,
\end{equation}
between the tensor functor on the level of chain complexes and the internal Hom functor composed with the forgetful functor $\on{dgMod}(A)\rightarrow \on{dgMod}(k)$. Given a Quillen-adjunction between model categories, there is an associated adjunctions between the underlying $\infty$-categories, see \cite{Maz16}. We denote the adjunction of $\infty$-categories underlying the Quillen adjunction \eqref{adj2} by
\begin{equation}\label{adj1}
\mhyphen \otimes_k^{\on{dg}}X:\mathcal{D}(k)\leftrightarrow \mathcal{D}(A):\on{RHom}_A(X,\mhyphen)\,.
\end{equation}

\begin{lemma}
Let $A$ be a $k$-linear dg-algebra. The $\infty$-category $\mathcal{D}(A)$ admits the structure of a $k$-linear $\infty$-category such that for any $X\in \mathcal{D}(A)$ the functor $\mhyphen \otimes_k^{dg} X$ is $k$-linear.
\end{lemma}

\begin{proof}
The $\infty$-category $\mathcal{D}(A)$ is stable and presentable by \cite[1.3.5.9, 1.3.5.21]{HA}. We now show $\mathcal{D}(A)$ is left tensored over $\mathcal{D}(k)$. Note that $\on{dgMod}(k)_0\simeq \on{Ch}(k)$ is a symmetric monoidal model category with respect to the tensor product, which we denote in the following by $\otimes$, see \cite[7.1.2.11]{HA}. We further denote the Quillen bifunctor $\on{dgMod}(k)\times \on{dgMod}(A)\rightarrow \on{dgMod}(A)$ given by the relative tensor product by $\mhyphen \otimes_k^{\on{dg}}\mhyphen $. Recall that $\mathcal{LM}^\otimes$ denotes the left-module $\infty$-operad, see \cite[4.2.1.7]{HA}. We define a $1$-category $O_A^\otimes$ as follows.
\begin{itemize}
\item An object of $O_A^\otimes$ consists of an object $(\underbrace{a,\dots,a}_{i\on{-many}},\underbrace{m,\dots,m}_{j\on{-many}})\in \mathcal{LM}^{\otimes}$ and objects 
\[(x_1,\dots,x_i)\in (\on{dgMod}(k)^\circ)^{\times i},\,(m_1,\dots,m_j)\in (\on{dgMod}(A)^\circ)^{\times j}\,.\]
\item For $n=1,2$, consider the object $X_n$ of $O_A^\otimes$ given by $l_n=(\underbrace{a,\dots,a}_{i_n\on{-many}},\underbrace{m,\dots,m}_{j_n\on{-many}})\in \mathcal{LM}^{\otimes}$  and 
\[(x^n_1,\dots,x^n_{i_n})\in (\on{dgMod}(k)^\circ)^{\times {i_n}},\,(m_1^n,\dots,m^n_{j_n})\in (\on{dgMod}(A)^\circ)^{\times {j_n}}\,.\]
A morphism $X_1\rightarrow X_2$ consists of a morphism $\alpha:l_1\rightarrow l_2$ in $\mathcal{LM}^\otimes$, which we also consider as a morphism of sets $\tilde{\alpha}:\{1,\dots,i_1+j_1\}\rightarrow \{1,\dots,i_2+j_2\}$, morphisms 
\[ \bigotimes_{e\in \tilde{\alpha}^{-1}(i)} a^1_e\rightarrow a_i^2\]
in $\on{dgMod}(k)^\circ$ for $1\leq i \leq i_2$ and morphisms
\[ \bigg( \bigotimes_{e\in \tilde{\alpha}^{-1}(j)\backslash \on{max}(\tilde{\alpha}^{-1}(j))} a^1_e \bigg) \otimes_k m_{\on{max}(\tilde{\alpha}^{-1}(j))-i_1}^1\rightarrow m^2_{j-i_2}\] in $\on{dgMod}(A)^\circ$ for $i_1+1\leq j \leq i_2+j_2$.
\end{itemize}
The forgetful functor $N(O_A^\otimes)\rightarrow \mathcal{LM}^\otimes$ is a coCartesian fibration of $\infty$-operads, exhibiting $N((\on{dgMod}(A)^\circ)_0)$ as left-tensored over the symmetric monoidal $\infty$-category $N((\on{dgMod}(k)^\circ)_0)$. By the discussion following \cite[4.1.7.3]{HA} and using that $\mhyphen \otimes_k^{\on{dg}}\mhyphen$ preserves weak equivalences in both entries, it follows that the left-tensoring passes to the $\infty$-categorical localizations at the chain homotopy equivalences, meaning that we obtain that $\mathcal{D}(A)$ is left-tensored over $\mathcal{D}(k)$. The action of $\mathcal{D}(k)$ on $\mathcal{D}(A)$ preserves colimits in both variables, as follows from the monoidal product $\mhyphen \otimes_k \mhyphen$ being a Quillen-bifunctor. To see that $\mhyphen \otimes_k^{\on{dg}}X$ is a $k$-linear functor, we need to describe an extension of $
\mhyphen \otimes_k^{\on{dg}} X$ to a map $\alpha:N(O^\otimes_k)\rightarrow N(O^\otimes_A)$ of $\infty$-operads over $\mathcal{LM}^\otimes$. We leave the details of the description of a functor of $1$-categories $\alpha':O^\otimes_k\rightarrow O^\otimes_A$ whose nerve $N(\alpha')$ defines the desired functor $\alpha$ to the reader.
\end{proof}

\begin{proposition}
Let $A$ be a $k$-linear dg-algebra. Using the symmetric monoidal equivalence $\mathcal{D}(k)\simeq \on{RMod}_k$, we can consider $\on{RMod}_A\overset{\eqref{modeqeq2}}{\simeq} \on{RMod}_A(\on{RMod}_k)$ as left-tensored over $\mathcal{D}(k)$. There exists an equivalence 
\begin{equation}\label{modeleq} \mathcal{D}(A)\simeq \on{RMod}_{A}
\end{equation} 
of $\infty$-categories left-tensored over $\mathcal{D}(k)$.
\end{proposition}

\begin{proof}
Consider the adjunction of $\infty$-categories $\mhyphen \otimes_k^{\on{dg}}A:\mathcal{D}(k)\leftrightarrow \mathcal{D}(A):\on{RHom}_A(A,\mhyphen)$ underlying the Quillen adjunction $\mhyphen\otimes_k^{\on{dg}} A:\on{dgMod}(k)\rightarrow \on{dgMod}(A):\on{Hom}_A(A,\mhyphen)$. Using the adjunction it can be directly checked that $A$ is a compact generator of $\mathcal{D}(A)$. It follows from \cite[4.8.5.8]{HA} that there exists an equivalence 
\begin{equation}\label{dmeq1}
\mathcal{D}(A)\simeq \on{RMod}_{\on{End}_k(A)}(\mathcal{D}(k))
\end{equation}
 of $\infty$-categories left-tensored over $\mathcal{D}(k)$, where $\on{End}_k(A)\in \on{Alg}(\mathcal{D}(k))$ is the $k$-linear endomorphism algebra of $A$, see \Cref{rendrem}. We note that the underlying chain complex satisfies $\on{End}_k(A)\simeq \on{RHom}_A(A,A)\simeq A$. By the universal property of $\on{End}_k(A)$, there exists a morphism of dg-algebras $\chi:A\rightarrow \on{End}_k(A)$, the underlying morphism of chain complexes of which is induced by the actions $A\otimes_kA\rightarrow A$ and $A\otimes_k \on{End}_k(A)\rightarrow A$. The latter is induced by the counit of the adjunction $\mhyphen\otimes_k^{\on{dg}}A\dashv \on{RHom}_A(A,\mhyphen)$ and thus equivalent to the former. It follows that $\chi$ induces a quasi-isomorphism $\on{End}_k(A)=\on{RHom}_A(A,A)\simeq A$ on underlying chain complexes and is hence a quasi-isomorphism of dg-algebras. In total we obtain, that there also exists an equivalence of $k$-linear $\infty$-categories $\on{RMod}_{\on{End}_k(A)}(\mathcal{D}(k))\simeq \on{RMod}_{A}(\mathcal{D}(k))\simeq \on{RMod}_A(\on{RMod}_k)$, which combined with \eqref{dmeq1} shows the statement. 
\end{proof}

Let $A,B\in \on{Alg}(\mathcal{D}(k))$ be dg-algebras and $F:\on{RMod}_A\rightarrow \on{RMod}_B$ a $k$-linear functor. Clearly $F(A)\in \on{RMod}_B$ carries the structure of a right $B$-module. Let $m:A\otimes_k A\rightarrow A$ be the multiplication map of $A$. Using the $k$-linearity of $F$, we find an action map
\[A\otimes_k F(A)\simeq F(A\otimes_k A)\xrightarrow{F(m)} F(A)\,,\]
which is part of the datum of a left $A$-module structure on $F(A)$. It turns out that both module structures are compatible, so that we can endow $F(A)$ with the structure of an $A$-$B$-bimodule. In total, we obtain a functor 
\[ \phi:\on{Lin}_k(\on{RMod}_A,\on{RMod}_B)\rightarrow \!\!\tensor[_A]{\on{BMod}}{_{B}}(\mathcal{D}(k))\,.\]
As is shown in \cite[Section 4.8.4]{HA}, the functor $\phi$ is an equivalence of $\infty$-categories. Given a bimodule $M\in\!\! \tensor[_A]{\on{BMod}}{_{B}}(\mathcal{D}(k))$, we denote by $\mhyphen \otimes_A M$ a choice of $k$-linear functor such that $\phi(\mhyphen \otimes_A M)\simeq M$.

\begin{proposition}\label{mndcomprop}
Let $A,B$ be dg-algebras and $M\in\!\!\tensor[_A]{\on{BMod}}{_B}(\mathcal{D}(k))\simeq \mathcal{D}(A^{op}\otimes_kB)$ and consider the functor of $\infty$-categories $\mhyphen \otimes^{\on{dg}}_A M$ underlying the right Quillen functor $\mhyphen \otimes^{\on{dg}}_A M : \on{dgMod}(A)\rightarrow \on{dgMod}(B)$. There exists a commutative diagram in $\on{LinCat}_k$ as follows.
\begin{equation}\label{mndcom}
\begin{tikzcd}
\on{RMod}_A \arrow[d, "\eqref{modeleq}"']\arrow[d, "\simeq"] \arrow[r, "\mhyphen\otimes_AM"] & {\on{RMod}}{_B} \arrow[d, "\eqref{modeleq}"']\arrow[d, "\simeq"] \\
\mathcal{D}(A) \arrow[r, "\mhyphen\otimes_A^{\on{dg}}M"]        & \mathcal{D}(B)                 
\end{tikzcd}
\end{equation}
\end{proposition}

\begin{remark}
As justified by \Cref{mndcomprop}, we will not distinguish in notation between the functors $\mhyphen\otimes_A M$ and $\mhyphen\otimes_A^{\on{dg}}M$ in the remainder of the text. 
\end{remark}

\begin{proof}[Proof of \Cref{mndcomprop}]
The $k$-linear functor 
\[ \chi:\on{RMod}_A \simeq \mathcal{D}(A)\xrightarrow{\mhyphen\otimes_A^{\on{dg}}M}  \mathcal{D}(B)\simeq {\on{RMod}}{_B}\] 
is of the form $\mhyphen\otimes_A N$ for $N\in\!\tensor[_A]{\on{BMod}}{_B}(\mathcal{D}(k))$, see \cite[Section 4.8.4]{HA}. We note that $N$ can be rectified to a strict dg-bimodule and is thus determined by its right $B$-module structure and its left $A$-module structure. The right $B$-module structures of $N$ and $M$ are clearly equivalent. In particular, there exists an equivalence $N\simeq M$ of underlying chain complexes. The left action of $A$ on $N$ is determined by $A\otimes_k N\simeq \chi(A\otimes_kA)\xrightarrow{\chi(m)} \chi(A)\simeq N,$ where $m$ denotes the multiplication of $A$ and is thus equivalent to the given left action of $A$ on the $A$-$B$-bimodule $M$. This shows that $N\simeq M$ as bimodules.
\end{proof}

\begin{proposition}\label{dgrem}
Let $A$ be a $k$-linear dg-algebra and $X\in \on{dgMod}(A)$ a cofibrant $A$-module. The $k$-linear endomorphism algebra $\on{End}_k(X)\in \on{Alg}(\mathcal{D}(k))$ of $X$ is quasi-isomorphic to the endomorphism dg-algebra $\on{End}^{\on{dg}}(X)$ of $X$. 
\end{proposition}

\begin{proof}
\Cref{mndcomprop} shows that the functor 
\[F:\mathcal{D}(k)\simeq \on{RMod}_k \xrightarrow{\mhyphen \otimes_k X} \on{RMod}_A(\on{RMod}_k)\simeq \mathcal{D}(A)\]
is equivalent to $\mhyphen \otimes_k^{\on{dg}}X$. The right adjoint $G$ of $F$ is given by $\on{RHom}_A(X,\mhyphen)$. It follows that $\on{RHom}_A(X,X)\simeq G(X)=\on{End}_k(X)$ in $\mathcal{D}(k)$, see also the definition of $\on{End}_k(X)$ in the proof of \Cref{genlem1}. Using that $\on{RHom}_A(X,X)=\on{Hom}_A(X,X)=\on{End}^{\on{dg}}(X)$ and the explicit $\on{Hom}_{A}(X,X)$-module structure on $X$, it follows from the universal property of the endomorphism object that there exists a morphism of dg-algebras $\alpha:\on{RHom}_A(X,X)\rightarrow \on{End}_k(X)$, which restricts to the quasi-isomorphism on underlying chain complexes and is hence an quasi-isomorphism of dg-algebras.
\end{proof}

\subsection{Morita theory of dg-categories}\label{sec2.5}

Let $k$ be a commutative ring. We denote by $\on{dgCat}_k$ the category of $k$-linear dg-categories. Given a dg-category $C\in \on{dgCat}_k$, the dg-category $\on{dgMod}(C)$ admits a model structure called the projective model structure. We have already encountered this model structure in \Cref{sec2.4} in the case where $C$ is a dg-algebra. We define $C^{\on{perf}}$ as the full dg-subcategory of $\on{dgMod}(C)$ spanned by fibrant-cofibrant objects $x$ which are compact in the homotopy category $H^0(\on{dgMod}(C))$, i.e.~$\on{Hom}(x,\mhyphen)$ preserves coproducts. This assignment forms a functor 
\[ (\mhyphen)^{\on{perf}}:\on{dgCat}_k\rightarrow \on{dgCat}_k\,.\]

As shown by Tabuada \cite{Tab05b}, the category $\on{dgCat}_k$ admits a model structure where 
\begin{itemize}
\item the weak equivalences are the quasi-equivalences, that is dg-functors $F:A\rightarrow B$ such that for all $a,a'\in A$, the morphism between morphism complexes $F(a,a'):\on{Hom}_{A}(a,a')\rightarrow \on{Hom}_{{B}}(F(a),F(a'))$ is a quasi-isomorphism and such that the induced functor on homotopy categories is an equivalence.
\item the fibrations are the dg-functors $F$ such that for all $a,a'\in A$, the morphism between mapping complexes $F(a,a'):\on{Hom}_A(a,a')\rightarrow \on{Hom}_B(F(a),F(a'))$ is degreewise surjective and satisfies that for any isomorphism $b\rightarrow F(a')$ in the homotopy category of $B$ there exists a lift along $F$ to an isomorphism $a\rightarrow a'$ in the homotopy category of $A$.
\end{itemize}

The $\infty$-category underlying this model category is given by the $\infty$-categorical localization $\on{dgCat}_k[W^{-1}]$ of the nerve of $\on{dgCat}_k$ at the collection $W$ of weak equivalences. This model structure can be further localized at the collection $M$ of Morita-equivalences, that is dg-functors $F$ such that $(F)^{\on{perf}}$ is a quasi-equivalence. The resulting model structure is called the Morita model structure. The corresponding localization functor 
\[ L: \on{dgCat}_k[W^{-1}]\longrightarrow \on{dgCat}_k[M^{-1}]\]
exhibits $\on{dgCat}_k[M^{-1}]$ as a reflective localization of $\on{dgCat}_k[W^{-1}]$ and thus preserves colimits. Given $C\in \on{dgCat}_k[W^{-1}]$, its image $L(C)$ is quasi-equivalent to $C^{\on{perf}}$. The Morita model structure models the $\infty$-category of $k$-linear, stable and idempotent complete $\infty$-categories, meaning that there exists an equivalence of $\infty$-categories 
\begin{equation}\label{mdeq1} 
\on{dgCat}_k[M^{-1}]\simeq \on{Mod}_{N_{\on{dg}}(k^{\on{perf}})}(\on{St}^{\on{idem}})\,,
\end{equation}
see \cite{Coh13}. The right side of \eqref{mdeq1} describes the $\infty$-category of modules in the symmetric monoidal category $\on{St}^{\on{idem}}$ over the algebra object $N_{\on{dg}}(k^{\on{perf}})$. The equivalence \eqref{mdeq1} maps a dg-category $C$ to the dg-nerve of the dg-category $C^{\on{perf}}$. $\on{Ind}$-completion provides a further colimit preserving functor $\on{Ind}:\on{Mod}_{\mathcal{D}(k)^{\on{perf}}}(\on{St}^{\on{idem}})\rightarrow \on{LinCat}_k$. In total we obtain the colimit preserving functor 
\begin{equation}\label{psieq} 
\mathcal{D}(\mhyphen): \on{dgCat}_k[W^{-1}]\xlongrightarrow{L} \on{dgCat}_k[M^{-1}]\simeq \on{Mod}_{\mathcal{D}(k)^{\on{perf}}}(\on{St}^{\on{idem}}) \xlongrightarrow{\on{Ind}} \on{LinCat}_k\xlongrightarrow{\on{forget}} \mathcal{P}r^L\,.
\end{equation}
Note that given a dg-algebra $A$, the derived $\infty$-category $\mathcal{D}(A)$ is equivalent to the image of $A$ under \eqref{psieq}, so that the notation $\mathcal{D}(\mhyphen)$ for the functor \eqref{psieq} is justified. Furthermore, we can compute colimits in $\on{dgCat}_k[W^{-1}]$ as homotopy colimits in $\on{dgCat}_k$ with respect to the quasi-equivalence model structure. 

\subsection{Semiorthogonal decompositions}\label{sec2.6}

In this section we discuss semiorthogonal decompositions of stable $\infty$-categories of length $n\geq 2$. Some of the treatment is based on the discussion of semiorthogonal decompositions of length $n=2$ in \cite{DKSS19}.

\begin{definition}
Let $\mathcal{V}$ and $\mathcal{A}$ be stable $\infty$-categories. We call $\mathcal{A}\subset \mathcal{V}$ a stable subcategory if the inclusion functor is fully faithful, exact and its image is closed under equivalences.  
\end{definition}

\begin{notation}
Let $\mathcal{V}$ be a stable $\infty$-category and $\mathcal{A}_1,\dots,\mathcal{A}_n\subset \mathcal{V}$ stable subcategories. We denote by $\langle \mathcal{A}_1,\dots,\mathcal{A}_n\rangle$ the smallest stable subcategory of $\mathcal{V}$ containing $\mathcal{A}_1,\dots,\mathcal{A}_n$.
\end{notation}

\begin{definition}\label{soddef}
Let $\mathcal{V}$ be a stable $\infty$-category and let $\mathcal{A}_1,\dots,\mathcal{A}_n$ be stable subcategories of $\mathcal{V}$. Consider the full subcategory $\mathcal{D}$ of $\on{Fun}(\Delta^{n-1},\mathcal{V})$ spanned by diagrams $D:\Delta^{n-1}\rightarrow \mathcal{V}$ satisfying the following two conditions.
\begin{itemize}
\item $D(i)$ lies in $\langle \mathcal{A}_{n-i},\dots,\mathcal{A}_{n}\rangle$ for $0\leq i\leq n-1$.
\item The cofiber of $D(i)\rightarrow D(i+1)$ in $\mathcal{V}$ lies in $\mathcal{A}_{n-i-1}$ for all $0\leq i \leq n-2$.
\end{itemize}
We call the ordered $n$-tuple $(\mathcal{A}_1,\dots,\mathcal{A}_n)$ a semiorthogonal decomposition of $\mathcal{V}$ of length $n$ if the restriction functor $\mathcal{D}\rightarrow \mathcal{V}$ to the vertex $n-1$ is a trivial fibration.
\end{definition}

\begin{definition}
Let $\mathcal{V}$ be a stable $\infty$-category and $\mathcal{A}\subset \mathcal{V}$ a stable subcategory. We define 
\begin{itemize}
\item the right orthogonal $\mathcal{A}^\perp$ to be the full subcategory of $\mathcal{V}$ spanned by those vertices $x\in \mathcal{V}$ such that for all $a\in \mathcal{A}$ the mapping space $\on{Map}_\mathcal{V}(a,x)$ is contractible. 
\item the left orthogonal $\prescript{\perp}{}{\mathcal{A}}$ to be the full subcategory of $\mathcal{V}$ spanned by those vertices $x\in \mathcal{V}$ such that for all $a\in \mathcal{A}$ the mapping space $\on{Map}_\mathcal{V}(x,a)$ is contractible.
\end{itemize}
\end{definition}

The next lemma shows that semiorthogonal decompositions of length $n$ are simply repeated semiorthogonal decompositions of length $2$.

\begin{lemma}\label{sodlem1}
Let $\mathcal{V}$ be a stable $\infty$-category and $\mathcal{A}_i\subset \mathcal{V}$, for $1\leq i\leq n$, a stable subcategory. $(\mathcal{A}_1,\dots,\mathcal{A}_n)$ is a semiorthogonal decomposition of $\mathcal{V}$ if and only if
\begin{enumerate}[i)]
\item $\langle \mathcal{A}_1,\dots,\mathcal{A}_n\rangle=\mathcal{V}$ and 
\item $(\mathcal{A}_i,\prescript{\perp}{}{\mathcal{A}_i})$ forms a semiorthogonal decomposition of $\langle \mathcal{A}_{i},\dots,\mathcal{A}_n\rangle$ for all $1\leq i\leq n-1$.
\end{enumerate}
\end{lemma}

\begin{proof}
For $1\leq i \leq n-1$ and $0\leq j \leq n-i-1$, denote by $\mathcal{D}^{i}_j\subset \on{Fun}(\Delta^{\{j,\dots,n-i\}},\langle \mathcal{A}_i,\dots,\mathcal{A}_{n}\rangle)$ the full subcategory spanned by diagrams $D^i_j$ satisfying that
\begin{itemize}
\item $D_j^i(l)$ lies in $\langle \mathcal{A}_{n-l},\dots,\mathcal{A}_{n}\rangle$ for $j\leq l\leq n-i$,
\item the cofiber of $D^i_j(l)\rightarrow D^i_j(l+1)$ in $\mathcal{V}$ lies in $\mathcal{A}_{n-l+1}$ for all $j\leq l \leq n-i-1$.
\end{itemize} 
We denote by $r_{i,j}:\mathcal{D}^i_j\rightarrow \langle \mathcal{A}_i,\dots,\mathcal{A}_n\rangle$ the functor given by the restriction to the vertex $n-i$. Note that for $i<k \leq n-j-1$, there is a trivial fibration $\mathcal{D}^i_j\rightarrow \mathcal{D}^i_{n-k}\times_{\langle \mathcal{A}_{k},\dots,\mathcal{A}_n\rangle}\mathcal{D}^k_j$.  

Now suppose that $(\mathcal{A}_1,\dots,\mathcal{A}_n)$ is a semiorthogonal decomposition and let $\mathcal{D}\rightarrow \mathcal{V}$ be the corresponding trivial fibration. Condition $i)$ is immediate. For condition $ii)$, we need to show that $r_{i,n-i-1}$ is a trivial fibration for all $1\leq i \leq n-1$. Using that pullbacks preserve trivial fibrations, it follows that 
\begin{equation}\label{rfun1} 
\mathcal{D}'=\mathcal{D}\times_{\mathcal{V}}\langle \mathcal{A}_i,\dots,\mathcal{A}_n\rangle\rightarrow \langle \mathcal{A}_i,\dots,\mathcal{A}_n\rangle
\end{equation} 
is a trivial fibration. We can describe the elements of $\mathcal{D}'$ as the left Kan extensions along the inclusion $\Delta^{\{0,\dots,n-i\}}\rightarrow \Delta^{n}$ of elements of $\mathcal{D}^{i}_{0}$. It thus follows from \cite[4.3.2.15]{HTT} that the restriction functor $\mathcal{D}'\rightarrow \mathcal{D}^{i}_0$ to $\Delta^{\{0,\dots,i\}}$ is a trivial fibration. Using that the functor \eqref{rfun1} factors through $r_{i,0}:\mathcal{D}^i_0\rightarrow \langle \mathcal{A}_i,\dots,\mathcal{A}_n\rangle$, it follows from the 2/3 property of equivalences that also $r_{i,0}$ is a trivial fibration. 
The following commutative diagram thus shows that $r_{i,n-i-1}$ is a trivial fibration. We have thus shown statement $ii)$.
\begin{equation}\label{rdiag}
\begin{tikzcd}
D^i_{0} \arrow[r, "\text{triv fib}"'] \arrow[rrr, "{r_{i,0}}", bend left=10] & {D^i_{n-i-1}\times_{\langle \mathcal{A}_{i+1},\dots,\mathcal{A}_n\rangle}D^{i+1}_{0}} \arrow[rd, "\lrcorner", phantom, near start] \arrow[r, "\text{triv fib}"'] \arrow[d] & D^i_{n-i-1} \arrow[d] \arrow[r, "{r_{i,n-i-1}}"']      & {\langle \mathcal{A}_{i},\dots,\mathcal{A}_n\rangle} \\
                                                                             & D^{i+1}_{0} \arrow[r, "{r_{i+1,0}}"]                                                                                                                                       & {\langle \mathcal{A}_{i+1},\dots,\mathcal{A}_n\rangle} &                                                     
\end{tikzcd}
\end{equation}

We now show that conditions $i)$ and $ii)$ imply that $(\mathcal{A}_1,\dots,\mathcal{A}_n)$ is a semiorthogonal decomposition of $\mathcal{V}$. If $n=2$, the assertion is obvious. We proceed by induction over $n$. Assume that $(\mathcal{A}_2,\dots,\mathcal{A}_{n})$ is a semiorthogonal decomposition of $\langle \mathcal{A}_2,\dots \mathcal{A}_n\rangle$, meaning that $r_{2,0}$ is a trivial fibration. To show that $(\mathcal{A}_1,\dots,\mathcal{A}_n)$ is a semiorthogonal decomposition of $\mathcal{V}=\langle \mathcal{A}_1,\dots,\mathcal{A}_n\rangle$, we need to show that $r_{1,0}$ is also a trivial fibration. Condition $ii)$ implies that $r_{1,n-2}$ is a trivial fibration. The diagram \eqref{rdiag} for $i=1$ thus shows that $r^1_0$ is also a trivial fibration.
\end{proof}

As it turns out, the functoriality data involved in the definition of semiorthogonal decompositions of length $2$ is redundant.  

\begin{lemma}\label{sod2lem}
Let $\mathcal{V}$ be a stable $\infty$-category and let $\mathcal{A},\mathcal{B}$ be stable subcategories of $\mathcal{V}$. The pair $(\mathcal{A},\mathcal{B})$ forms a semiorthogonal decomposition of length $2$ of $\mathcal{V}$ if and only if 
\begin{enumerate}
\item for all $a\in \mathcal{A}$ and $b\in \mathcal{B}$, the mapping space $\on{Map}_{\mathcal{V}}(b,a)$ is contractible and
\item for every $x\in \mathcal{V}$, there exists a fiber and cofiber sequence $b\rightarrow x\rightarrow a$ in $\mathcal{V}$ with $a\in \mathcal{A}$ and $b\in \mathcal{B}$.
\end{enumerate}
\end{lemma}

\begin{proof}
This follows from \cite[7.2.0.2]{SAG}.
\end{proof}

A simple source of semiorthogonal decompositions are sequences of functors between stable $\infty$-categories.

\begin{lemma}\label{sodlem3}
Let $D:\Delta^{n-1}\rightarrow \on{Cat}_\infty$ be a diagram taking values in stable $\infty$-categories, corresponding to $n$ composable functors 
\[ \mathcal{A}_1\xlongrightarrow{F_1}\mathcal{A}_2\xlongrightarrow{F_2}\dots\xlongrightarrow{F_{n-1}}\mathcal{A}_n\,.\]
\begin{enumerate}
\item The stable $\infty$-category 
\[\{\mathcal{A}_1,\dots,\mathcal{A}_n\}\coloneqq \on{Fun}_{\Delta^{n-1}}(\Delta^{n-1},\Gamma(D))\]
of sections of the Grothendieck construction $p:\Gamma(D)\rightarrow \Delta^{n-1}$, see \Cref{sec2.1}, admits a semiorthogonal decomposition $(\mathcal{A}_1,\dots,\mathcal{A}_n)$ of length $n$. 
\item Let $R$ be an $\mathbb{E}_\infty$-ring spectrum. If each $F_i$ is an $R$-linear functor between $R$-linear $\infty$-categories, then the $\infty$-category $\{\mathcal{A}_1,\dots,\mathcal{A}_n\}$ further inherits the structure of an $R$-linear $\infty$-category such that each inclusion functor $\iota_i:\mathcal{A}_i\rightarrow \{\mathcal{A}_1,\dots,\mathcal{A}_n\}$ is $R$-linear. 
\end{enumerate}
\end{lemma}

\begin{proof}
We begin by showing statement 1. Consider the simplicial set 
\[ Z= \left(\Delta^0\times \Delta^{n-1}\right)\amalg_{\Delta^{\{1\}}\times \Delta^{\{1,\dots,n-1\}}} \left( \Delta^1\times \Delta^{n-2}\right)\amalg \dots \amalg_{\Delta^{\{1,\dots,n-1\}}\times \Delta^{\{1\}}}\left(\Delta^{n-1}\times \Delta^0\right)\,.\]
Let $\mathcal{D}'$ be the full subcategory of $\on{Fun}(Z,\Gamma(D))$ spanned by diagrams given by right Kan extensions along the inclusion $\Delta^{n-1}\times \Delta^0\rightarrow Z$ of a diagram in $\{\mathcal{A}_1,\dots,\mathcal{A}_n\}$. By \cite[4.3.2.15]{HTT}, the restriction functor $\mathcal{D}'\rightarrow \{\mathcal{A}_1,\dots\mathcal{A}_n\}$ to $\Delta^0\times \Delta^{n-1}$ is a trivial fibration. We can describe the elements of $\mathcal{D}'$ up to equivalence as diagrams in $\Gamma(D)$ of the form  
\[ 
\begin{tikzcd}
                    &                     &                               & a_1 \arrow[d]   \\
                    &                     & a_2 \arrow[r, "\on{id}"] \arrow[d] & a_2 \arrow[d]   \\
                    & \dots \arrow[r]\arrow[d]     & \dots \arrow[r]\arrow[d]               & \dots \arrow[d] \\
a_n \arrow[r, "\on{id}"] & \dots \arrow[r, "\on{id}"] & a_n \arrow[r, "\on{id}"]           & a_n            
\end{tikzcd}
\]
satisfying that $a_i\in \mathcal{A}_i$. The restriction functor $\mathcal{D}'\rightarrow \{\mathcal{A}_1,\dots,\mathcal{A}_n\}$ corresponds in the above description pointwise to the restriction to the rightmost column. The $\infty$-category $\mathcal{D}$ of \Cref{soddef} can be identified with the full subcategory of $\on{Fun}(\Delta^{n-1}\times\Delta^{n-1},\Gamma(\alpha))$ spanned by left Kan extensions along $Z\rightarrow \Delta^{n-1}\times \Delta^{n-1}$ of diagrams lying in $\mathcal{D}'$. It follows that the restriction functor $\mathcal{D}\rightarrow \mathcal{D}'$ is a trivial fibration and thus that the restriction functor $\mathcal{D}\rightarrow \{\mathcal{A}_1,\dots,\mathcal{A}_n\}$ is a trivial fibration. This shows statement 1.

We now show statement 2. Consider the diagram of $\infty$-operads over $\mathcal{LM}^\otimes$ 
\[ D^{\otimes}: \mathcal{O}_1^\otimes\xrightarrow{F_1^\otimes} \mathcal{O}_2^\otimes\xrightarrow{F_2^\otimes}\dots\xrightarrow{F_{n-1}^\otimes}\mathcal{O}_n^\otimes\]
exhibiting the functors $F_i$ as $R$-linear. The morphism of $\infty$-operad 
\[\on{Fun}_{\Delta^{n-1}}(\Delta^{n-1},\Gamma(D^\otimes))\times_{\on{Fun}(\Delta^{n-1},\mathcal{LM}^\otimes)}\mathcal{LM}^\otimes\rightarrow \mathcal{LM}^\otimes\]
exhibits $\{\mathcal{A}_1,\dots,\mathcal{A}_n\}$ as left tensored over 
\[ \mathcal{M}\coloneqq \on{Fun}_{\Delta^{n-1}}(\Delta^{n-1},\Gamma(D^\otimes))\times_{\on{Fun}(\Delta^{n-1},\mathcal{LM}^\otimes)}\mathcal{LM}^\otimes \times_{\mathcal{LM}^\otimes}\on{Assoc}^\otimes\,.\]
Let $\tilde{D}^\otimes:\Delta^{n-1}\rightarrow \on{Cat}_\infty$ be the constant diagram with value $\on{LMod}_R^\otimes$.
We find $\mathcal{M}$ to be equivalent as a monoidal $\infty$-category to
\begin{equation}\label{meq}
\on{Fun}_{\Delta^{n-1}}(\Delta^{n-1},\Gamma(\tilde{D}^\otimes))\times_{\on{Fun}(\Delta^{n-1},\on{Assoc}^\otimes)}\on{Assoc}^\otimes\,.
\end{equation}
Pulling back along the monoidal functor $\on{LMod}_R^\otimes\rightarrow \mathcal{M}$, assigning to $x\in \on{LMod}_R^\otimes$ the constant section in \eqref{meq}, we obtain a left-tensoring of $\{\mathcal{A}_1,\dots,\mathcal{A}_n\}$ over $\on{LMod}_R$. To show that the left-tensoring provides the structure of an $R$-linear $\infty$-category, it suffices to show that the monoidal product preserves colimits in the second entry. This follows from the observation that colimits in $\{\mathcal{A}_1,\dots,\mathcal{A}_n\}$ are computed pointwise, i.e.~the $n$ restriction functors $\{\mathcal{A}_1,\dots,\mathcal{A}_n\}\rightarrow \mathcal{A}_i$ preserve colimits.
\end{proof}

We also introduce the following notation used later on.

\begin{notation}\label{cocartnot1}
Let $p:\Gamma\rightarrow \Delta^{n}$ be an inner fibration. Given an edge $e:a\rightarrow a'$ in $\Gamma$ we write $e:a\xrightarrow{!}a'$ if $e$ is a $p$-coCartesian edge and $e:a\xrightarrow{\ast}a'$ if $e$ is a $p$-Cartesian edge.
\end{notation}

Our next goal is to describe an analogue of the construction of the semiorthogonal decomposition in \Cref{sodlem3} in the setting of dg-categories and show that the resulting $\infty$-categories with semiorthogonal decompositions are equivalent. For that we recall the notion of gluing functors of semiorthogonal decompositions of length $2$, see also \cite{DKSS19}. 

\begin{definition}\label{chidef}
Let $\mathcal{V}$ be a stable $\infty$-category with a semiorthogonal decomposition $(\mathcal{A},\mathcal{B})$. We define a simplicial set $\chi(\mathcal{A},\mathcal{B})$ by defining an $n$-simplex of $\chi(\mathcal{A},\mathcal{B})$ to correspond to the following data.
\begin{itemize}
\item An $n$-simplex $j:\Delta^n\rightarrow \Delta^1$ of $\Delta^1$.
\item An $n$-simplex $\sigma:\Delta^n\rightarrow \mathcal{V}$ such that $\sigma(\Delta^{j^{-1}(0)})\subset \mathcal{A}$ and  $\sigma({\Delta^{j^{-1}(1)}})\subset \mathcal{B}$.
\end{itemize}
We define the face and degeneracy maps to act on an $n$-simplex $(j,\sigma)\in \chi(\mathcal{A},\mathcal{B})_n$ componentwise.

We denote by $p:\chi(\mathcal{A},\mathcal{B})\rightarrow \Delta^1$ the apparent forgetful functor.
\end{definition}

\begin{definition}\label{cartdef}
Let $\mathcal{V}$ be a stable $\infty$-category with a semiorthogonal decomposition $(\mathcal{A},\mathcal{B})$. We call
\begin{itemize}
\item $(\mathcal{A},\mathcal{B})$ Cartesian if the functor $p:\chi(\mathcal{A},\mathcal{B})\rightarrow \Delta^1$ is a Cartesian fibration. In that case, we call the functor associated to the Cartesian fibration $p$ the right gluing functor associated to $(\mathcal{A},\mathcal{B})$.
\item $(\mathcal{A},\mathcal{B})$ coCartesian if the functor $p:\chi(\mathcal{A},\mathcal{B})\rightarrow \Delta^1$ is a coCartesian fibration. In that case, we call the functor associated to the coCartesian fibration $p$ the left gluing functor associated to $(\mathcal{A},\mathcal{B})$.
\end{itemize}
\end{definition}

\begin{lemma}[$\!\!$\cite{DKSS19}]
Let $\mathcal{V}$ be a stable $\infty$-category with a semiorthogonal decomposition $(\mathcal{A},\mathcal{B})$. 
\begin{enumerate}
\item If $(\mathcal{A},\mathcal{B})$ is Cartesian, the inclusion functor $\mathcal{A}\rightarrow \mathcal{V}$ admits a right adjoint, the restriction of which to $\mathcal{B}$ is the right gluing functor of $(\mathcal{A},\mathcal{B})$.
\item If $(\mathcal{A},\mathcal{B})$ is coCartesian, the inclusion functor $\mathcal{B}\rightarrow \mathcal{V}$ admits a left adjoint, the restriction of which to $\mathcal{A}$ is the left gluing functor of $(\mathcal{A},\mathcal{B})$.
\end{enumerate}
\end{lemma}

The next proposition can be summarized as showing that Cartesian semiorthogonal decompositions of length $2$ are fully determined by their  left gluing functor and dually that coCartesian semiorthogonal decomposition of length $2$ are fully determined by their right gluing functor.

\begin{proposition}[$\!\!$\cite{DKSS19}]\label{sodprop1}
Let $\mathcal{V}$ be a stable $\infty$-category with a semiorthogonal decomposition $(\mathcal{A},\mathcal{B})$. 
\begin{enumerate}
\item If $(\mathcal{A},\mathcal{B})$ is Cartesian with right gluing functor $G$, there exists an equivalence of $\infty$-categories $\mathcal{V}\simeq \on{Fun}_{\Delta^1}(\Delta^1,\chi(G))$, where $\chi(G)\rightarrow \Delta^1$ is the Cartesian fibration classifying $G$ considered as a functor $\Delta^1\rightarrow \on{Cat}_\infty$.
\item If $(\mathcal{A},\mathcal{B})$ is coCartesian with left gluing functor $F$, there exists an equivalence of $\infty$-categories $\mathcal{V}\simeq \on{Fun}_{\Delta^1}(\Delta^1,\Gamma(F))$.
\end{enumerate}
\end{proposition}

\begin{definition}
Let $\mathcal{V}$ be a stable $\infty$-category and let $(\mathcal{A}_1,\dots,\mathcal{A}_n)$ be a semiorthogonal decomposition of $\mathcal{V}$. We call
\begin{itemize}
\item $(\mathcal{A}_1,\dots,\mathcal{A}_n)$ a Cartesian semiorthogonal decomposition if each semiorthogonal decomposition $(\mathcal{A}_i,\prescript{\perp}{}{\mathcal{A}_i})$ is Cartesian. In that case, we call the right gluing functor of $(\mathcal{A}_i,\prescript{\perp}{}{\mathcal{A}_i})$ the $i$-th right gluing functor of $(\mathcal{A}_1,\dots,\mathcal{A}_n)$. 
\item $(\mathcal{A}_1,\dots,\mathcal{A}_n)$ a coCartesian semiorthogonal decomposition if each semiorthogonal decomposition $(\mathcal{A}_i,\prescript{\perp}{}{\mathcal{A}_i})$ is coCartesian. If $(\mathcal{A}_1,\dots,\mathcal{A}_n)$ is coCartesian,  we call the left gluing functor of $(\mathcal{A}_i,\prescript{\perp}{}{\mathcal{A}_i})$ the $i$-th left gluing functor of $(\mathcal{A}_1,\dots,\mathcal{A}_n)$. 
\end{itemize}
\end{definition}

We now introduce a dg-analogue of \Cref{sodlem3}: semiorthogonal decompositions arising from upper triangular dg-algebras concentrated on the diagonal and upper minor diagonal. 

\begin{definition}\label{tdgalg}
For $1\leq i\leq n$, let $A_i$ be a dg-algebra and for $1\leq i \leq n-1$ let $M_i$ be an $A_{i}$-$A_{i+1}$-bimodule. We denote by 
\[ \bf{A}=
\begin{pmatrix}
A_1 & M_1 &0 & \dots &0 &0\\
0   & A_2 & M_2& \dots &0 &0\\
0   & 0   & A_3& \dots &0 &0\\
\vdots& \vdots & \vdots & \ddots & \vdots & \vdots\\
0   & 0   & 0  & \dots  & A_{n-1} & M_{n-1}\\
0   & 0   & 0  & \dots  & 0 & A_n  
\end{pmatrix}
\] 
be the upper triangular dg-algebra, i.e.~the dg-algebra with underlying chain complex 
\[\bigoplus_{1\leq i \leq n}A_i \oplus \bigoplus_{1\leq i \leq n-1}M_i\]
and multiplication $\cdot$ given by
\begin{align*}
a_i\cdot a_j'& =\delta_{i,j}a_ia_j'\,, & m_i\cdot m_j'&= 0\,,\\
a_i \cdot m_j& = \delta_{i,j}a_i.m_j\,,& m_j\cdot a_i &=\delta_{j+1,i}m_j.a_i\,,
\end{align*}
where $a_i\in A_i,a_j'\in A_j$ and $m_i\in M_i, m_j'\in M_j$ and $\delta_{i,j}$ denotes the Kronecker delta.
\end{definition}

\begin{proposition}\label{sodprop4}
Let $\bf{A}$ be an upper triangular dg-algebra as in \Cref{tdgalg}. Then the stable $\infty$-category $\mathcal{D}(A)$ carries a semiorthogonal decomposition $(\mathcal{D}(A_1),\dots,\mathcal{D}(A_n))$ of length $n$ with $i$-th left gluing functor $\mhyphen \otimes_{A_i} M_i$.
\end{proposition}

\begin{proof}
The upper triangular dg-algebra $\bf A$ is quasi-isomorphic to the upper triangular dg-algebra obtained from cofibrantly replacing each $M_i$. We thus assume without loss of generality that the $M_i$ are cofibrant bimodules. Consider the morphisms of dg-algebras $v^i:A_i\rightarrow \bf A$ and $w^i:{\bf A}\rightarrow A_i$, given on the underlying chain complexes by the inclusion of the direct summand $A_i$ and the projection to the summand $A_i$, respectively. The dg-functor $v^i_!=\mhyphen \otimes^{\on{dg}}_{A_i}A_i\oplus M_i:\on{dgMod}(A_i)\rightarrow \on{dgMod}(A)$ and the pullback $(w^i)^*$ determine right $A$-modules $v^i_!(A_i)$ and $(w^i)^*(A_i)$ with underlying chain complexes $A_i\oplus M_i$, where we set $M_n=0$, and $A_i$, respectively. The functors $\mathcal{D}(v_!^i)$ and $\mathcal{D}((w^i)^*)$ both exhibit $\mathcal{D}(A_i)\subset \mathcal{D}(\bf A)$ as a stable subcategory. For concreteness, we denote the stable subcategories obtained from $\mathcal{D}(v_!^i)$ by $\mathcal{D}(A_i)_v$ and the stable subcategories obtained from $\mathcal{D}((w^i)^*)$ by $\mathcal{D}(A_i)_w$. We wish to show that $(\mathcal{D}(A_1)_v,\dots,\mathcal{D}(A_n)_v)$ is a semiorthogonal decomposition of $\mathcal{D}(\bf A)$. For that it suffices to show statements $i)$ and $ii)$ of \Cref{sodlem1}. To show statement $ii)$, it suffices to show conditions 1 and 2 of \Cref{sod2lem} for the pairs of stable subcategories $\mathcal{D}(A_i)_v,\langle \mathcal{D}(A_{i+1})_v,\dots,\mathcal{D}(A_{n})_v\rangle \subset \langle \mathcal{D}(A_{i})_v,\dots,\mathcal{D}(A_n)_v\rangle$ for all $1\leq i \leq n$. 

We compute for an $A_i$-module $N_i$ and an $A_j$-modules $N_j$ the mapping complex
\begin{equation}\label{hom1}
\on{Hom}_{\on{dgMod}(\bf A)}(v_!^i(N_i),v_!^j(N_j))\simeq \begin{cases}
\on{Hom}_{\on{dgMod}(A_i)}(N_i,N_j) & \text{if }i=j,\\ 
\on{Hom}_{\on{dgMod}(A_j)}(N_i\otimes_{A_i} M_i,N_j) & \text{if }i+1=j,\\
0 & \text{else}. \end{cases}
\end{equation} 
This shows condition 1 of \Cref{sod2lem}.

We observe that the datum of a right dg-module $N$ over $\bf A$ is equivalent to the datum of a sequence 
\[ N_1\xlongrightarrow{f_1}N_2\xlongrightarrow{f_2} \dots \xlongrightarrow{f_{n-1}}N_n\] 
where $N_i$ is a right $A_i\simeq \on{End}^{\on{dg}}((w^i)^*(A_i))$-module and $f_i\in M_i(N_i,N_{i+1})$. Denote by $N_{\geq i}$ the submodule $N_i\xrightarrow{f_i}\dots\xrightarrow{f_{n-1}}N_n$ of $N$. We thus find distinguished triangles $N_{\geq i+1}\rightarrow N_{\geq i} \rightarrow N_{i}$ in $\on{dgMod}(\bf A)$. As shown in \cite[Theorem 4.3.1]{Fao17}, the image under the dg-nerve of a distinguished triangle in a dg-category can be extended to a fiber and cofiber sequence. We can thus express $N\in \mathcal{D}(\bf A)$ as repeated cofibers of modules $N_i\in \mathcal{D}(A_i)_w\subset \mathcal{D}(\bf A)$ with $1\leq i \leq n$. A simple induction, using that there exist distinguished triangles in $\on{dgMod}(\bf A)$ of the form $N_{i}\rightarrow N_{i}\otimes_{A_i} v_!^i(A_i)\rightarrow N_{i}\otimes_{A_i} M_i$ for $1\leq i\leq n-1$ and $N_n\in \mathcal{D}(A_n)_w=\mathcal{D}(A_n)_v$, thus shows that $N\in \langle \mathcal{D}(A_1)_v,\dots,\mathcal{D}(A_n)_v\rangle$. 
It follows that statement $i)$ of \Cref{sodlem1} is fulfilled. 

Consider the subalgebra ${\bf A}_{\geq i}$ of $\bf A$ with underlying chain complex 
\[\bigoplus_{i\leq k\leq n}A_k\oplus \bigoplus_{i\leq k\leq n-1}M_k\,.\] 
The fully faithful dg-functor $\on{dgMod}({\bf A}_{\geq i})\rightarrow \on{dgMod}(\bf A)$ induces a fully faithful functor of $\infty$-categories $\iota:\mathcal{D}(\bf A_{\geq i})\rightarrow \mathcal{D}(\bf A)$. The above arguments show that the essential image of $\iota$ is $\langle \mathcal{A}_i,\dots,\mathcal{A}_n\rangle$ and can easily be adapted to also show condition $2$ of \Cref{sod2lem}. We have thus proven the existence of the desired semiorthogonal decomposition of $\mathcal{D}(\bf A)$.

We now determine the $i$-th left gluing functor of $(\mathcal{D}(A_1)_v,\dots,\mathcal{D}(A_n)_v)$. Consider the fully-faithful left Quillen functor 
\[ \mhyphen \otimes_{A_i}^{\on{dg}}A_i:\on{dgMod}(A_i)_0\rightarrow \on{dgMod}({\bf A}_{\geq i})_0\,.\]
The right adjoint is given by the Quillen functor $\on{Hom}_{\on{dgMod}{{(\bf A}_{\geq i}})}(A_i,\mhyphen)$,
the restriction of which to $\on{dgMod}({\bf A}_{\geq i+1})$ is given by $\on{Hom}_{\on{dgMod}({\bf A}_{\geq i+1})}(M_i,\mhyphen)$, which in turn is left adjoint to $\mhyphen \otimes^{\on{dg}}_{A_i}M_i$. Passing to the underlying adjunctions of $\infty$-categories of the above Quillen adjunctions shows that the $i$-th left gluing functor of $(\mathcal{D}(A_1),\dots,\mathcal{D}(A_n))$ is given by $\mhyphen \otimes_{A_i} M_i$.
\end{proof}

\begin{remark}
An illuminating discussion of the role of the morphism of dg-algebra $v^i$ and $w^i$ appearing in the proof of \Cref{sodprop4} and the resulting stable subcategories of $\mathcal{D}(\bf A)$ can be found in \cite[Section 2.3.2]{Bar20}.
\end{remark}

\begin{proposition}\label{sodprop5}
For $1\leq i\leq n$, let $A_i$ be a dg-algebra and for $1\leq i \leq n-1$, let $M_i$ be an $A_i$-$A_{i+1}$-bimodule. Denote by $\bf A$ the upper triangular dg-algebra of \Cref{tdgalg}. Consider the diagram $\alpha:\Delta^{n-1}\rightarrow \on{LinCat}_k$ corresponding to 
\[ \mathcal{D}(A_1)\xlongrightarrow{\mhyphen\otimes_{A_1}M_1}\mathcal{D}(A_2)\xlongrightarrow{\mhyphen\otimes_{A_2}M_2}\dots\xlongrightarrow{\mhyphen\otimes_{A_{n-1}}M_{n-1}}\mathcal{D}(A_n)\,.\]
Then there exists an equivalence of $\infty$-categories 
\[ \mathcal{D}({\bf A})\simeq \{\mathcal{D}(A_1),\dots,\mathcal{D}(A_n)\}\]
such that for all $1\leq i\leq n$, the following diagram commutes.
\begin{equation}\label{dadiag}
\begin{tikzcd}
                                        & \mathcal{D}(A_i) \arrow[ld, "{\mathcal{D}(v^i_!)[n-i]}"'] \arrow[rd, "{\iota_i}"] &                                               \\
\mathcal{D}(\bf A) \arrow[rr, "\simeq"] &                                                                   & {\{\mathcal{D}(A_1),\dots,\mathcal{D}(A_n)\}}
\end{tikzcd}
\end{equation}
\end{proposition}

\begin{proof}
The left gluing functors of the semiorthogonal decompositions $(\mathcal{D}(A_1),\dots,\mathcal{D}(A_n))$ of $\{\mathcal{D}(A_1),\dots,\mathcal{D}(A_n)\}$ and $\mathcal{D}(\bf A)$ are equivalent. It thus follows from a repeated application of \Cref{sodprop1} that there exists an equivalence of $\infty$-categories $\mathcal{D}({\bf A})\simeq \{\mathcal{D}(A_1),\dots,\mathcal{D}(A_n)\}$. The observation that  \eqref{dadiag} commute, follows from the observation that the equivalences of \Cref{sodprop1} commute with the inclusion functors of the components of the semiorthogonal decomposition, up to delooping.
\end{proof}

\begin{notation}\label{sodnot}
Let $\bf A$ be an upper triangular dg-algebra as in \Cref{tdgalg} and $1\leq i \leq n$. Using the notation from the proof of \Cref{sodprop4}, we denote $p_i{\bf A}=\mathcal{D}(v_!^i)(A_i)\in \mathcal{D}(\bf A)$.
\end{notation}

\section{Parametrized perverse schobers locally}\label{sec3}

Perverse sheaves have their origin in a homology theory of stratified topological spaces called intersection homology. A perverse sheaf is an object in the derived category of constructible sheaves, i.e.~a complex of sheaves which are locally constant on any stratum. The homology of the stratified space is obtained via the sheaf cohomology of the perverse sheaf. Perversity of a complex of constructible sheaves is a condition on its homology with and without compact support. For example, on a complex surface, perversity implies that they the complex is concentrated in degree $0$ away from the singularities and concentrated in degrees $0,1$ at the singularities.

On some nice stratified spaces, there are known descriptions of the abelian category of perverse sheaves in terms of quiver representations, see e.g.~\cite{KS15} for an overview. A way to obtain such a description, is to identify a suitable 'skeleton' of the stratified space and describe the perverse sheaf in terms of certain sheaf cohomology groups with support restrictions related to the skeleton. These homology groups have to fulfill the crucial restriction that they are concentrated in a single degree. The most iconic such description is of the abelian category of perverse sheaves on a disc with a singularity in the center, in terms of the category of diagrams of vector spaces $r_1:V_1\leftrightarrow N_1:s_1$ such that $r_1s_1-\on{id}_{N_1}$ and $s_1r_1-\on{id}_{V_1}$ are equivalences. The vector spaces $N_1$ and $V_1$ are called nearby and vanishing cycles, respectively. 

While it is currently not clear how to categorify constructible sheaves and thus perverse sheaves directly, the remarkable idea of \cite{KS14} is to categorify perverse sheaves using their quiver descriptions, when available. The 'ad-hoc' categorification proposed in \cite{KS14} of the above mentioned quiver description of perverse sheaves on a disc is a spherical adjunction.

Further descriptions of the category of perverse sheaves on a disc with a singularity in the center in terms of quiver representations were given in \cite{KS16}. For each $n\geq 2$, the category of perverse sheaves is equivalent to the abelian category of diagrams of vector spaces $r_i:V_n\leftrightarrow N_i:s_i$ for $1\leq i \leq n$, such that 
\begin{itemize}
\item $r_i\circ s_i= \on{id}_{N_i}$,
\item $r_{i+1}\circ s_i$ (with $i+1$ modulo $n$) is an isomorphism for $1\leq i\leq n$ and
\item $r_i\circ s_j=0$ for $i\neq j,j+1$ mod $n$.
\end{itemize}
The vector spaces $N_i$ are all equivalent, they are the nearby cycles. The vector space $V_n$ describes the sheaf cohomology of the perverse sheaf with support on $n$ outgoing rays starting at the origin, see \Cref{nrays}.

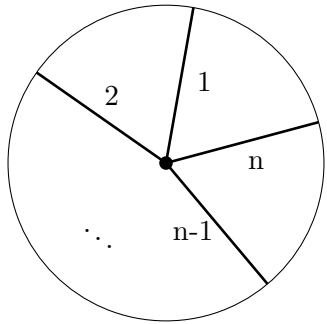
\begin{figure}[h!]
\centering
\begin{tikzpicture}[scale=0.6]
    \node (0) at (0,0){};
    \fill (0) circle (0.15); 
    \draw (0) circle (3.5); 

    \draw [-, line width = 1](0,0) -- (15:3.5);
    \draw [-, line width = 1](0,0) -- (80:3.5); 
    \draw [-, line width = 1](0,0) -- (145:3.5); 
    \draw [-, line width = 1](0,0) -- (-50:3.5);
    
    \node at (2,0){n}; 
    \node at (0.6,-1.5){n-1}; 
    \node at (-1.2,1.5){2}; 
    \node at (0.85,1.8){1}; 
    \node at (-1.5,-1.5){$\ddots$}; 
\end{tikzpicture}
\caption{The complex disc with $n$ outgoing rays with a chosen order.}\label{nrays}
\end{figure}
    
The map $r_i$ is defined as the restriction map to a point on the $i$-th outgoing ray. Note that the restriction maps $r_i$ have a paracyclic symmetry, meaning a cyclic symmetry up to the monodromy of the perverse sheaf, arising from the cyclic symmetry of the $n$ rays given by rotating the disc by $2\pi/n$. In \Cref{sec3.1} we describe an ad-hoc categorification of this quiver description, which will provide a local description of a parametrized perverse schober. The categorification is based on Dyckerhoff's categorified Dold-Kan correspondence, see \cite{Dyc17}. As noted in \textit{loc.~cit.}$\,$, one of the motivations for the categorified Dold-Kan correspondence was the categorification of the local description of perverse sheaves.

\subsection{An ad-hoc categorification}\label{sec3.1}

We begin with briefly recalling the concept of a spherical adjunction. Consider an adjunction of stable $\infty$-categories $F:\mathcal{A}\leftrightarrow \mathcal{B}:G$. We associate the following endofunctors.
\begin{itemize}
\item The twist functor $T_\mathcal{A}$ is defined as the cofiber in the stable $\infty$-category $\on{Fun}(\mathcal{A},\mathcal{A})$ of the unit map $\on{id}_{\mathcal{A}}\rightarrow GF$ of the adjunction $F\dashv G$. 
\item The cotwist functor $T_\mathcal{D}$ is defined as the fiber in the stable $\infty$-category $\on{Fun}(\mathcal{B},\mathcal{B})$ of the counit map $FG\rightarrow \on{id}_{\mathcal{B}}$ of the adjunction $F\dashv G$. 
\end{itemize} 
The adjunction $F\dashv G$ is called spherical if the functors $T_\mathcal{A}$ and $T_\mathcal{B}$ are equivalences. In this case, the functor $F$ is also called spherical. A spherical functor $F$ admits repeated left and right adjoints, each given by the composite of $F$ or $G$ with a power of the twist or cotwist functor. For a treatment of spherical adjunctions in the setting of stable $\infty$-categories, we refer to \cite{DKSS19} and \cite{Chr20}.  

A $2$-simplicial stable $\infty$-category is an $(\infty,2)$-functor $\Delta^{(\on{op},\mhyphen)}\rightarrow \mathcal{S}t$, from the 2-categorical version of the simplex category to the $(\infty,2)$-version $\mathcal{S}t$ of the $\infty$-category $\on{St}$ of stable $\infty$-categories. The categorified Dold-Kan correspondence of \cite{Dyc17} is an adjoint equivalence between the $\infty$-category of bounded below complexes of stable $\infty$-categories and the $\infty$-category of $2$-simplicial stable $\infty$-categories. The right adjoint is called the categorified Dold-Kan nerve $\mathcal{N}$. The categorified Dold-Kan nerve $\mathcal{N}$ generalizes the well known construction from $K$-theory called the Waldhausen $S_\bullet$-construction. More precisely, given a complex of stable $\infty$-categories concentrated in degrees $0,1$, the categorified Dold-Kan nerve recovers Waldhausen's relative $S_\bullet$-construction. We refer to \cite{Dyc17} for further details.

Let $F:\mathcal{A}\leftrightarrow \mathcal{B}$ be a spherical adjunction. We consider the spherical functor $G:\mathcal{B}\rightarrow \mathcal{A}$ as a complex of stable $\infty$-categories concentrated in degrees $0,1$, denoted $G[0]$. We further denote by $\mathcal{B}[1]$ the complex concentrated in degree $1$ with value $\mathcal{B}$. Consider the morphism between bounded below complexes of stable $\infty$-categories $G[0]\rightarrow \mathcal{B}[1]$ depicted as follows.
\[
\begin{tikzcd}
\mathcal{A} \arrow[d] & \mathcal{B} \arrow[l, "G"'] \arrow[d, "\on{id}_\mathcal{B}"'] & 0 \arrow[l] \arrow[d] & \dots \arrow[l] \\
0                     & \mathcal{B} \arrow[l]                                    & 0 \arrow[l]           & \dots \arrow[l]
\end{tikzcd}
\]
Applying the categorified Dold-Kan nerve $\mathcal{N}$, we obtain a morphism $\phi_\ast:\mathcal{N}(G[0])_\ast\rightarrow \mathcal{N}(\mathcal{B}[1])_\ast$ between the simplicial objects in $\on{St}$ underlying the $2$-simplicial objects in $\mathcal{S}t$. Spelling out the definition of the categorified Dold-Kan nerve and the properties of Kan extensions, cf.~\cite[4.3.2.15]{HTT}, we obtain the following.

\begin{lemma}\label{rdklem}
Let $F:\mathcal{A}\leftrightarrow \mathcal{B}:G$ be a spherical adjunction and $\mathcal{N}(G[0])_\ast$ and $\mathcal{N}(\mathcal{B}[1])_\ast$ as above. There exist the following equivalences between $\infty$-categories.
\begin{enumerate}
\item $\mathcal{N}(G[0])_0\simeq \mathcal{A}$.
\item $\mathcal{N}(G[0])_n\simeq \{\mathcal{A},\mathcal{B},\dots,\mathcal{B}\}$ for $n\geq 1$, in the notation of \Cref{sodlem3}, corresponding to the following sequence of $n$ functors. 
\[\mathcal{A}\xlongrightarrow{F}\mathcal{B}\xlongrightarrow{\on{id}}\mathcal{B}\xlongrightarrow{\on{id}}\dots\xlongrightarrow{\on{id}}\mathcal{B}\]
\item $\mathcal{N}(\mathcal{B}[1])_1\simeq \mathcal{B}$.
\end{enumerate}
\end{lemma}

We propose that for $n\geq 0$, the $\infty$-category $\mathcal{N}(G[0])_n$ of $n$-simplicies categorifies the vector space $V_{n+1}$ of sections supported on $n+1$-outgoing rays and the $\infty$-category $\mathcal{B}\simeq \mathcal{N}(\mathcal{B}[1])_1$ of $1$-simplicies categorifies the vector spaces $N_i$ of nearby cycles. Accordingly, we call $\mathcal{B}$ the $\infty$-category of nearby cycles and $\mathcal{A}$ the $\infty$-category of vanishing cycles of $F\dashv G$, or simply of $F$.

\begin{notation}\label{rdknot}
Let $F:\mathcal{A}\leftrightarrow \mathcal{B}:G$ be a spherical adjunction. We denote
\begin{itemize}
\item $\mathcal{V}^1_{F}=\mathcal{A}$.
\item $\mathcal{V}^n_{F}=\{\mathcal{A},\underbrace{\mathcal{B},\dots,\mathcal{B}}_{n-1\text{-many}}\}$ for $n\geq 2$. 
\item $\mathcal{N}_F=\mathcal{B}$.
\end{itemize}
\end{notation}

Assume that $n\geq 3$. We propose that the first restriction map $r_1:V_{n}\rightarrow N_1$ is categorified by the functor 
\[ \varrho_1:\mathcal{V}^{n}_{F}\simeq \mathcal{N}(G[0])_{n-1}\xrightarrow{d_0} \mathcal{N}(G[0])_{n-2}\xrightarrow{d_0}\dots \xrightarrow{d_0}\mathcal{N}(G[0])_1\xrightarrow{\phi_1} \mathcal{N}(\mathcal{B}[1])_1 \simeq \mathcal{N}_F\]
obtained from composing $\phi_1$ with repeated $0$th face maps of the simplicial structure of $\mathcal{N}(G[0])_\ast$. The functor $\varrho_1$ can equivalently be described as the projection functor $\pi_n$ to the $n$-th component of the semiorthogonal decomposition $(\mathcal{V}^1_F,\mathcal{N}_F,\dots,\mathcal{N}_F)$ of length $n$ of $\mathcal{V}^{n}_F$. If $n=1$, we propose that the restriction map $r_1$ is categorified by $F:\mathcal{V}^1_F\rightarrow \mathcal{N}_F$ and if $n=2$ we propose that the restriction map $r_1$ is categorified by $\phi_1=\pi_2$. To categorify the further restriction maps, we need to take into account the paracyclic symmetry. The description of the categorification of $V_n$ in terms of $\mathcal{V}^n_F$ however obscures this paracyclic symmetry. One way to solve this is by lifting the simplicial object $\mathcal{N}(G[0])_\ast$ to a paracyclic object. This approach is realized in \cite{DKSS19}. One can then replace $\mathcal{V}^n_F$ by an equivalent $\infty$-category where the paracyclic symmetry is apparent. For now we adopt a more pedestrian approach and simply require that there be a sequence of adjunctions
\begin{equation}\label{adjseq} \varrho_n \dashv \varsigma_n \dashv \varrho_{n-1}\dashv \dots \dashv \varsigma_2 \dashv \varrho_1 \dashv \varsigma_1\,,\end{equation}
where $\varrho_1$ is as above and propose that $\varsigma_i$ categorifies $s_i$ and $\varrho_i$ categorifies $r_i$. We describe the paracyclic symmetry of $\mathcal{V}^n_F$ to justify our proposed categorification in \Cref{sec3.2}. We call the $\varrho_i$ the categorified restriction maps. A direct computation shows that the functors $\varrho_i$ and $\varsigma_i$ are described as follows.
 
\begin{lemma}\label{rdknot2}
Let $F\dashv G$ as above and $n\geq 1$. Consider the following functors $\varrho_i:\mathcal{V}^n_F\rightarrow \mathcal{N}_F$ and $\varsigma_i:\mathcal{N}_F\rightarrow \mathcal{V}^n_F$ for $1\leq i\leq n$.
\begin{enumerate}
\item If $n=1$, we set $\varrho_1=F$ and $\varsigma_1=G$.
\item If $n\geq 2$, we set 
\[ 
\varrho_i=\begin{cases}
\pi_n& \text{for }i=1,\\         
\on{fib}_{n-i,n-i+1}[i-1]& \text{for }2\leq i\leq n-1,\\
\on{rfib}_{1,2}[n-1]& \text{for }i=n.
\end{cases}\]
The functor $\on{rfib}_{1,2}$ denotes the composition of the projection functor to the first two components of the semiorthogonal decomposition with the relative fiber functor that assigns to a vertex $a\rightarrow b \in \{\mathcal{V}^1_F,\mathcal{N}_F\}$ the vertex $\on{fib}(F(a)\rightarrow b)\in \mathcal{N}_F$. The functor $\on{fib}_{i-1,i}[n-i]$ denotes the composition of the projection functor to the $(i-1)$-th and $i$-th component with the fiber functor.
\item If $n\geq 2$, we set $\varsigma_1$ to be the functor that assigns to $b\in \mathcal{N}_F$ the object \mbox{$G(b)\xrightarrow{\ast}b\xrightarrow{\on{id}}\dots\xrightarrow{\on{id}}b$} in $\mathcal{V}^n_{F}$, see also \Cref{cocartnot1}, and set for $2\leq i\leq n$
\[ \varsigma_i= (\iota_{\mathcal{N}_F})_{n-i+2}[-i+2]\,,\]
where $(\iota_{\mathcal{N}_F})_{j}$ is the inclusion of the $j$-th component of the semiorthogonal decomposition.
\end{enumerate}
These functors form the sequence of adjunctions \eqref{adjseq}.
\end{lemma}

We are now ready to describe the local model for a parametrized perverse schober at a vertex of valency $n$.

\begin{definition}\label{locmod}
Let $F:\mathcal{V}^1_{F}\leftrightarrow \mathcal{N}_F:G$ be an adjunction of stable $\infty$-categories and $n\geq 1$. Consider the poset\footnote{The left cone $\left(\{1,\dots,n\}\right)^\triangleleft$ is defined as the simplicial join $\Delta^0\ast \{1,\dots,n\}$.} $C_n= \left(\{1,\dots,n\}\right)^\triangleleft$. If $n=1$, we denote by $\mathcal{G}_1(F):C_1\simeq \Delta^1\rightarrow \on{St}$ the functor $F$. If $n\geq 2$, we denote by $\mathcal{G}_n(F)$ the functor $C_n\rightarrow \on{St}$ assigning 
\begin{itemize}
\item to the initial vertex $\ast\in C_n$ the stable $\infty$-category $\mathcal{V}_F^{n}$,
\item to each vertex $i\in C_n$ the stable $\infty$-category $\mathcal{N}_F$,
\item to each edge $\ast\rightarrow i$ the functor $\varrho_i$ from \Cref{rdknot2}.
\end{itemize}
\end{definition}

The adjoint functors $\varsigma_i$ will feature in the local description of duals of parametrized perverse schobers, see \Cref{sec4.3}.

\subsection{The paracyclic structure}\label{sec3.2}

We begin by recalling the definition of the paracyclic 1-category $\Lambda_\infty$. 

\begin{definition}\label{linftydef}
For $n\geq 0$, let $[n]$ denote the set $\{0,\dots,n\}$. The objects of $\Lambda_\infty$ are the \mbox{sets $[n]$}. The morphism in $\Lambda_\infty$ are generated by morphisms
\begin{itemize}
\item $\delta^{0},\dots,\delta^{n}:[n-1]\rightarrow [n]$,
\item $\sigma^{0},\dots,\sigma^{n-1}:[n]\rightarrow [n-1]$,
\item $\tau^{n,i}:[n]\rightarrow [n]$ with $i\in \mathbb{Z}$
\end{itemize}
subject to the simplicial relations and the further relations
\begin{align*} 
\tau^{n,i}\circ \tau^{n,j}=\tau^{n,i+j},\quad\quad & \tau^{n,0}=\on{id}_{[n]},\\
\tau^{n,1}\delta^{i}=\delta^{i-1}\tau^{n-1,1}\text{ for }i>0,\quad\quad & \tau^{n,1}\delta^{0}=\delta^n,\\
\tau^{n,1}\sigma^{i}=\tau^{n+1,1}\sigma^{i-1}\text{ for }i>0,\quad\quad & \tau^{n,1}\sigma^0=\sigma^{n}\tau^{n+1,2}\,.
\end{align*}
\end{definition}

The simplex category $\Delta$ is a subcategory of $\Lambda_\infty$. A paracyclic object in an $\infty$-category $\mathcal{C}$ is a functor $\Lambda_\infty^{op}\rightarrow \mathcal{C}$, where we identify the $1$-category $\Lambda_\infty^{op}$ with its nerve. A paracyclic object in $\mathcal{C}$ is thus a simplicial object $X_\ast\in \on{Fun}(\Delta^{op},\mathcal{C})$ with face maps $d_i$ and degeneracy maps $s_i$ together with a sequence of paracyclic isomorphisms $t_n: X_n\rightarrow X_n$ satisfying
\begin{equation} \label{paracyc1} d_i t_n=t_{n-1}d_{i-1}\text{ for }i>0,\quad d_0 t_n= d_n\quad \text{and}\end{equation}
\begin{equation} \label{paracyc2} s_i t_n =t_{n+1}s_{i-1}\text{ for }i>0,\quad s_0 t_n = t_{n+1}^2s_{n}\,.\end{equation}

Let $F\dashv G$ be a spherical adjunction. As shown in \cite{DKSS19}, the simplicial object $\mathcal{N}(G[0])_\ast$ can be lifted to a paracyclic object. We emphasize that the sphericalness of the adjunction $F\dashv G$ is crucial for showing that the paracyclic isomorphism $t_n$ of this paracyclic structure is really an isomorphism. In this section we give an alternative description of the paracyclic isomorphisms $t_n$ in terms of the twist functor $T_{\mathcal{V}^n_F}$ of a spherical adjunction $F'\dashv G'$ described below in \Cref{paralem3}. We call $T_{\mathcal{V}^n_{F}}$ the paracyclic twist functor. We then proceed to show that this isomorphisms realizes the paracyclic symmetry of the functors $\varrho_i$ and $\varsigma_i$. 

\begin{construction}\label{paraconstr}
Let $F:\mathcal{V}^1_F\leftrightarrow \mathcal{N}_F:G$ be a spherical adjunction. Denote the left adjoint of $F$ by $E$. Consider the full subcategory $\mathcal{M}$ of the $\infty$-category of diagrams $\on{Fun}(\Delta^1\times\Delta^1,\Gamma(F))$ of the form
\[
\begin{tikzcd}
a \arrow[r] \arrow[d, "!"'] \arrow[rd] & a' \arrow[d, "\ast"] \\
b' \arrow[r]                          & b                   
\end{tikzcd}
\]
with $a,a'\in \mathcal{V}^1_F$ and $b,b'\in \mathcal{N}_F$. 
The restriction functor $\on{res}:\mathcal{M}\rightarrow \{\mathcal{V}^1_F,\mathcal{N}_F\}$, given by the projection to the edge $a\rightarrow b$ is a trivial fibration. As shown in \cite{DKSS19}, it follows from the sphericalness of the adjunction $F\dashv G$, that the fiber functor in the horizontal direction $\mathcal{M}\rightarrow \{\mathcal{V}^1_{F},\mathcal{N}_F\}$ is also an equivalence. By choosing a section of the trivial fibration $\on{res}$ and composing with the fiber functor we obtain an autoequivalence $\tau:\{\mathcal{V}^1_F,\mathcal{N}_F\}\rightarrow \{\mathcal{V}^1_F,\mathcal{N}_F\}$, called the \textit{relative suspension functor} in \textit{loc.~cit}. 
\end{construction}

\begin{lemma}\label{paralem2}
Let $F:\mathcal{V}^1_F\leftrightarrow \mathcal{N}_F:G$ be a spherical adjunction with cotwist functor $T_{\mathcal{N}_F}$. Denote the left adjoint of $F$ by $E$. The left adjoint of the functor 
\begin{equation}\label{rcof}
\mathcal{V}^2_F=\{ \mathcal{V}^1_F,\mathcal{N}_F\} \xrightarrow{\on{rfib}} \mathcal{N}_F
\end{equation}
is given by the functor that assigns $E(b)\xrightarrow{\ast}T_{\mathcal{N}_F}^{-1}(b)$ to $b\in \mathcal{N}_F$.
\end{lemma}

\begin{proof}
As shown in \cite[Lemma 1.30]{Chr20}, the stable subcategories $(\mathcal{V}^1_F)^\perp,\mathcal{V}^1_F,\mathcal{N}_F,\prescript{\perp}{}{\mathcal{N}_F}\subset \mathcal{V}^2_F$ form semiorthogonal decompositions $((\mathcal{V}^1_F)^\perp,\mathcal{V}^1_F),(\mathcal{V}^1_F,\mathcal{N}_F),(\mathcal{N}_F,\prescript{\perp}{}{\mathcal{N}_F})$ of $\mathcal{V}^2$. We denote by $i_{\mathcal{N}_F},i_{(\mathcal{V}^1_F)^\perp}$ the inclusion functors of $\mathcal{N}_F$ and $\mathcal{N}_F\simeq (\mathcal{V}^1_F)^\perp$ into $\mathcal{V}^2_F$, respectively. The functor $i_{(\mathcal{V}^1_F)^\perp}$ assigns to $b\in\mathcal{N}_F\simeq (\mathcal{V}^1_F)^\perp$ the object $G(b)\xrightarrow{!}b\in \mathcal{V}^2_F$. It is easily checked that there is a sequence of adjunctions 
\begin{equation}\label{adjseq2}
 \on{rfib}[1] \dashv i_{\mathcal{N}_F} \dashv \pi_0 \dashv i_{(\mathcal{V}^1_F)^\perp}\,.
\end{equation}
Composing with the adjunction $\tau^{-1}\dashv \tau$, where $\tau$ is the relative suspension functor from \Cref{paraconstr}, with the sequence of adjunction \eqref{adjseq2} yields the sequence of adjunctions 
\[ \pi_0[1]\dashv i_{(\mathcal{V}^1_F)^\perp}[-1]\dashv T_{\mathcal{N}_F}^{-1}\on{rfib}[1] \dashv i_{\mathcal{N}_F}T_{\mathcal{N}_F}\,.\]
We have thus established the desired adjunction $i_{(\mathcal{V}^1_F)^\perp}T_{\mathcal{N}_F}^{-1}\dashv \on{rfib}$. 
\end{proof}

\begin{lemma}\label{paralem3}
Let $F:\mathcal{V}^1_F\leftrightarrow \mathcal{N}_F:G$ be a spherical adjunction with cotwist functor $T_{\mathcal{N}_F}$. Denote the left adjoint of $F$ by $E$. For $n\geq 2$, consider the functor
\[F':\mathcal{V}^{n}_F \longrightarrow \mathcal{N}_F^{\times n} \]
with components $F'=(\varrho_1,\dots,\varrho_n)$.
\begin{enumerate}
\item The functor $F'$ admits left and right adjoints $E'$, respectively, $G'$, given by
\begin{align*} 
E'= &~(\varsigma_2,\dots,\varsigma_{n},\varsigma_1T_{\mathcal{N}_F}^{-1}[1-n])\,,\\
G'= &~(\varsigma_1,\dots,\varsigma_n)\,.
\end{align*}
\item The adjunction $F'\dashv G'$ is spherical.
\end{enumerate} 
\end{lemma}

\begin{proof}
We begin with showing statement 1. The adjunction $F'\dashv G'$ follows from composing the adjunctions $\varrho_i\dashv \varsigma_i$ with the adjunction $\Delta\dashv \oplus$ between the constant diagram functor $\Delta:\mathcal{N}_F\rightarrow \mathcal{N}_F^{\times n}$ and its right adjoint given by the direct sum functor. Again by composing adjunctions, we obtain that to show that $E'$ is left adjoint to $F'$ it suffices to show that $\varsigma_1T_{\mathcal{N}_F}^{-1}[n]$ is left adjoint to $\varrho_n$. This follows directly from the following observations.
\begin{itemize}
\item The functor $\varrho_n$ factors as 
\[ \mathcal{V}_F^n\xrightarrow{\pi_{1,2}} \mathcal{V}_F^2\xrightarrow{\on{rfib}[n-1]} \mathcal{N}_F\,.\]
\item The left adjoint of $\on{rfib}:\mathcal{V}_F^2\rightarrow \mathcal{N}_F$ was determined in \Cref{paralem2} and is given by the functor that maps $b\in \mathcal{N}_F$ to $E(b)\xrightarrow{\ast}T_{\mathcal{N}_F}^{-1}(b)$.
\item The left adjoint of $\pi_{1,2}$ is given by the functor that maps $E(b)\xrightarrow{\ast}T_{\mathcal{N}_F}^{-1}(b)\in \mathcal{V}^2_F$ to $E(b)\xrightarrow{\ast}T_{\mathcal{N}_F}^{-1}(b)\xrightarrow{\on{id}}\dots \xrightarrow{\on{id}} T_{\mathcal{N}_F}^{-1}(b)\in \mathcal{V}^n_F$. 
\end{itemize}

For statement 2, consider the endofunctor $M=F'G':\mathcal{N}^{\times n}_{F}\rightarrow \mathcal{N}^{\times n}_{F}$ of the adjunction $F'\dashv G'$ with cotwist functor $T_{\mathcal{N}_F^{\times n}}$. We can depict $M$ as the following matrix. 
\[
\begin{pmatrix}
 \on{id}_{\mathcal{N}_F} & \on{id}_{\mathcal{N}_F} & 0 & \dots & 0 & 0\\
 0 & \on{id}_{\mathcal{N}_F} & \on{id}_{\mathcal{N}_F} & \dots & 0 & 0\\
 0 & 0 & \on{id}_{\mathcal{N}_F} &\dots &0 & 0\\
 \vdots & \vdots & \vdots & \ddots & \vdots &\vdots \\
 0 & 0 & 0 & \dots & \on{id}_{\mathcal{N}_F} & \on{id}_{\mathcal{N}_F}\\
 T_{\mathcal{N}_F}[n-1] & 0 & 0&\dots & 0 & \on{id}_{\mathcal{N}_F}
\end{pmatrix}
\]
The counit $cu:M\rightarrow \on{id}_{\mathcal{N}_F^{\times n}}$ is the projection to the diagonal, so that we deduce that the cotwist $T_{\mathcal{N}_F^{\times n}}$ is an equivalence. We further observe that there exists an equivalence $cu\circ T_{\mathcal{N}_F^{\times n}}\simeq T_{\mathcal{N}_F^{\times n}}\circ cu$. The left adjoint $E':\mathcal{N}_F^{\times n}\rightarrow \mathcal{V}^n_F$ clearly satisfies $G'\circ T_{\mathcal{N}_F^{\times n}}^{-1}$. We have shown that all conditions of \cite[Proposition 4.5]{Chr20} are fulfilled and it follows that the adjunction $F'\dashv G'$ spherical. 
\end{proof}

\begin{remark}\label{0bnot}
We highlight the relation of \Cref{paralem3} to other results in the literature. Let $F\dashv G$ be a spherical adjunction. Consider further the (trivially) spherical adjunction $0_{\mathcal{N}}:0\leftrightarrow \mathcal{N}:0_\mathcal{N}'$ and denote by $F'':\mathcal{V}^n_{0_{\mathcal{N}}}\leftrightarrow \mathcal{N}^{\times n}:G''$ the spherical adjunction associated in \Cref{paralem3} to $0_\mathcal{N}\dashv 0_{\mathcal{N}}'$. The adjunction $F''\dashv G''$ appears in the special case $\mathcal{N}=\mathcal{D}(k)^{\on{perf}}$ in \cite[Theorem 5.14]{BD19}, where it is shown that $F''$ carries a left Calabi-Yau structure. The spherical adjunction $F'\dashv G'$ associated to $F\dashv G$ in \Cref{paralem3} can be described as the composition of the spherical adjunctions $F''\dashv G''$ and 
\[(F,0\dots,0):\mathcal{V}^1_F\longleftrightarrow \mathcal{N}_F^{\times n}:(G,0,\dots,0)\]
in the sense of \cite{Bar20}.
\end{remark}

\begin{remark}
Consider the setting of \Cref{paralem3}. Lurie's $\infty$-categorical Barr-Beck theorem implies that the adjunction $F'\dashv G'$ is monadic. Further, if the adjunction $F\dashv G$ is monadic, then the adjunction $F'\dashv G'$ is also comonadic. \Cref{paralem2} thus implies that a monadic spherical adjunction $F\dashv G$ can be recovered from the comonad $M=F'G':\mathcal{N}_F^{\times n}\rightarrow \mathcal{N}_F^{\times n}$ whose underlying endofunctor is determined by the cotwist functor $T_{\mathcal{N}_F}$. This does not imply that the spherical monadic adjunction $F\dashv G$ can be recovered from its twist functor, see also \cite[Section 4.1]{Chr20}. All further data is encoded in the comonad structure of $M$. 
\end{remark}

\begin{proposition}\label{paralem4}
Let $F:\mathcal{V}^1_F\leftrightarrow \mathcal{N}_F:G$ be a spherical adjunction and consider the twist functor $T_{\mathcal{V}^n_F}$ of the spherical adjunction $F'\dashv G'$ described in \Cref{paralem3}. Then there exist equivalences of functors 
\begin{equation}\label{rhoeq}
\varrho_i \circ T_{\mathcal{V}^n_F} =
\begin{cases} 
\varrho_{i+1} & \text{ for }1\leq i\leq n-1\\
T_{\mathcal{N}_F}[n-1]\circ \varrho_1  & \text{ for }i=n
\end{cases}
\end{equation}
and 
\begin{equation}\label{sigmaeq}
T_{\mathcal{V}^n_F}^{-1} \circ \varsigma_i =
\begin{cases} 
\varsigma_{i+1}  & \text{ for }1\leq i\leq n-1\\
\varsigma_1 \circ T_{\mathcal{N}_F}^{-1}[1-n]  & \text{ for }i=n
\end{cases}
\end{equation}
\end{proposition} 

\begin{proof}
By the 2/4 property of spherical adjunctions there exists an equivalence $T_{\mathcal{V}^n_F}^{-1}G'\simeq E'$, showing the identities \eqref{sigmaeq}. The identities \eqref{rhoeq} follow from passing to left adjoints.
\end{proof}

\section{Parametrized perverse schobers globally}\label{sec4}

In \Cref{sec4.1}, we review background material on marked surfaces, their ideal triangulations and associated ribbon graphs. In \Cref{sec4.2,sec4.3}, we introduce the notion of a perverse schober parametrized by a ribbon graph and define the $\infty$-categories of global sections of such a perverse schober. In \Cref{sec4.4}, we discuss how perverse schobers parametrized by different ribbon graphs can be related.

\subsection{Marked surfaces, ideal triangulations and ribbon graphs}\label{sec4.1}

\begin{definition}
By a surface ${\bf S}$, we mean a smooth, connected surface with possibly empty boundary. We denote by $\partial {\bf S}$ and ${\bf S}^\circ$ the boundary and interior of ${\bf S}$, respectively.

A marked surface is a compact surface ${\bf S}$ together with a finite collection of marked points $M\subset {\bf S}$. We further require that each boundary component of ${\bf S}$ contains at least one marked point and if $\partial {\bf S}=\emptyset$, that $M\neq \emptyset$. 
 
Interior marked points are also called punctures. We denote by $\Sigma={\bf S}\backslash (M\cap {\bf S}^\circ)$ the non-compact surface with these punctures removed.
\end{definition}

We remark that the definition of marked surface does not exclude special cases, such as the twice punctured sphere or the once-punctured monogon.

We proceed by defining an ideal triangulation of a marked surface.

\begin{definition}
Let ${\bf S}$ be a marked surface. A curve $\gamma$ in ${\bf S}$ is called simple if
\begin{itemize}
\item the endpoints of $\gamma$ lie in $M$,
\item $\gamma$ does not intersect $M$ and $\partial {\bf S}$, except at the endpoints,
\item $\gamma$ does not self-intersect, except that its endpoints may be coincide,
\item if $\gamma$ is a closed loop, then $\gamma$ is not contractible onto $M$ or $\partial {\bf S}$.
\end{itemize}  
An arc in ${\bf S}$ is an equivalence class of curves under isotopy and reversal of parametrization. Two arcs are called compatible if there are curves in their respective isotopy classes which do not intersect, except possibly at the endpoints.
\end{definition}

\begin{definition}[see e.g.~{\cite[Definition 2.6]{FST08}}]
Let ${\bf S}$ be a marked surface. An ideal triangulation $\mathcal{T}$ of ${\bf S}$ consists of a maximal collection of distinct pairwise compatible arcs in ${\bf S}$.  
\end{definition}

Any collection of distinct and pairwise compatible arcs can be realized by curves in the respective isotopy classes which do not intersect except for the endpoints, see \cite[Proposition 2.5]{FST08}. Given an ideal triangulation $\mathcal{T}$ of a surface ${\bf S}$, we choose such a collection of non-intersecting curves. These curves cut ${\bf S}$ into ideal triangles. An ideal triangle has three (possibly two identical) sides which each connect two (possibly identical) marked points. The sides of an ideal triangle may lie on $\partial{\bf S}$ and be given by non-simple curves in $\partial{\bf S}$ connecting two marked points. An ideal triangle is called self-folded if it has a side which connects a single marked point with itself, see \Cref{selffoldfig}. Two of the sides of a self-folded ideal triangle are identical. 

\begin{figure}[t!]
\begin{center}
\begin{tikzpicture}[scale=0.5]
    \node (0) at (-6,2){};
    \node (1) at (-6,0){};
    \fill (0) circle (0.13); 
    \fill (1) circle (0.13); 
    \draw [very thick](-6,0.4) circle (1.6); 
    \draw [very thick](-6,2) -- (-6,0);
\end{tikzpicture}
\caption{A self-folded ideal triangle.}\label{selffoldfig}
\end{center}
\end{figure}
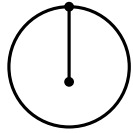

\begin{definition}
Let $\mathcal{T}$ be an ideal triangulation of an oriented marked surface. We choose a collection of non-intersecting simple curves representing $\mathcal{T}$. These simple curves are called the internal edges of $\mathcal{T}$. 

The boundary edges of $\mathcal{T}$ are those sides of the ideal triangles which are non-simple curves lying on the boundary $\partial {\bf S}$. 

The interior ideal triangles of $\mathcal{T}$ are the ideal triangles of $\mathcal{T}$ which are disjoint from the boundary $\partial {\bf S}$.
\end{definition}

In the remainder of this section, we discuss how marked surfaces and their ideal triangulations can be encoded in terms of ribbon graphs. Similar treatments of ribbon graphs can be found in \cite{DK15,DK18}.

\begin{definition}~
\begin{itemize}
\item A graph $\Gamma$ consists of two finite sets $\Gamma_0$ of vertices and $\on{H}_\Gamma$ of halfedges (sometimes simply denoted $\on{H}$) together with an involution $\tau:\on{H}\rightarrow \on{H}$ and a map $\sigma:\on{H}\rightarrow \Gamma_0$.
\item Let $\Gamma$ be a graph. We denote by $\Gamma_1$ the set of orbits of $\tau$. The elements of $\Gamma_1$ are called the edges of $\Gamma$. An edge is called internal if the orbit contains two elements and called external if the orbit contains a single element. An internal edge is called a loop at $v\in \Gamma_0$ if it consists of two halfedges both being mapped under $\sigma$ to $v$. We denote the set of internal edges of $\Gamma$ by $\Gamma_1^\circ$ and the set of external edges by $\Gamma_1^\partial$.
\item A ribbon graph consists of a graph $\Gamma$ together with a choice of a cyclic order on the set $\on{H}(v)$ of halfedges incident to $v$ for each $v\in\Gamma_0$.
\end{itemize}
\end{definition}

\begin{definition}
Let $\Gamma$ be a graph. We denote by $\on{Exit}(\Gamma)$ the category with 
\begin{itemize}
\item the set of elements $\Gamma_0 \amalg \Gamma_1$ and 
\item all non-identity morphisms of the form $v\rightarrow e$ with $v\in \Gamma_0$ a vertex and $e\in \Gamma_1$ an edge incident to $v$. If $e$ is a loop at $v$, then there are two morphisms $v\rightarrow e$.
\end{itemize}
We call $\on{Exit}(\Gamma)$ the exit path category of $\Gamma$.

The geometric realization $|\Gamma|$ of $\Gamma$ is defined as the geometric realization $|\on{Exit}(\Gamma)|$ of $\on{Exit}(\Gamma)$ as a simplicial set. 
\end{definition}

We only consider connected graphs, i.e.~graphs whose geometric realization is connected.

\begin{remark}\label{orrem}
Let $\Gamma$ be a graph and ${\bf S}$ an oriented surface. Any embedding of $|\Gamma|$ into ${\bf S}$ determines a ribbon graph structure on $\Gamma$, where the cyclic order of the halfedges at any vertex is so that the cyclic order in the geometric realization is counter-clockwise with respect to the orientation of ${\bf S}$.
\end{remark}

\begin{notation}\label{ribnot}
We use a graphical notation for ribbon graphs. We denote the vertices by $\cdot$ or sometimes $\times$, and internal edges by a straight line. We denote external edges as follows.
\[ \begin{tikzcd} \cdot& \arrow[l, no head, maps to]{}\end{tikzcd}\]
\end{notation}

\begin{example} 
The following diagram
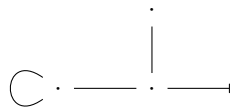
\[ 
\begin{tikzcd}
                                                                              & \cdot \arrow[d, no head] &                                \\
\cdot \arrow[r, no head] \arrow[no head, loop, distance=2em, in=215, out=145] & \cdot                     & {} \arrow[l, no head, maps to]
\end{tikzcd}
\]
denotes a ribbon graph $\Gamma$ with three vertices, four edges in total, one external edge and one loop and the cyclic order of the halfedges at each vertex going in the counter-clockwise direction. The exit path category of $\Gamma$ can be depicted as follows, with $v,v',v''$ denoting the vertices of $\Gamma$ and $e,e',e'',e'''$ denoting the edges of $\Gamma$.
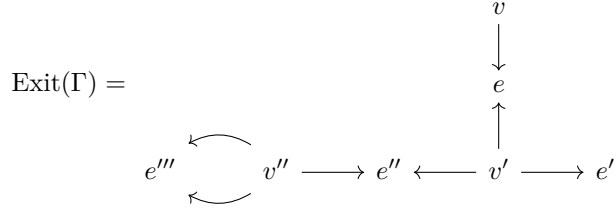
\[
\on{Exit}(\Gamma)=\begin{tikzcd}
     &                                                                                     &     & v \arrow[d]                      &    \\
     &                                                                                     &     & e                                &    \\
e''' & v'' \arrow[r] \arrow[l, bend left, shift left] \arrow[l, bend right, shift right=2] & e'' & v' \arrow[u] \arrow[r] \arrow[l] & e'
\end{tikzcd}
\]
\end{example}

Each ideal triangulation of an oriented marked surface determines a ribbon graph $\Gamma_\mathcal{T}$ as follows.

\begin{definition}\label{dualdef} 
Let ${\bf S}$ be an oriented surface with an ideal triangulation $\mathcal{T}$. We denote by $\Gamma_\mathcal{T}$ the ribbon graph determined by the following.
\begin{itemize}
\item The set of vertices of $\Gamma_\mathcal{T}$ is the set of ideal triangles of $\mathcal{T}$.
\item The set of internal edges of $\Gamma_\mathcal{T}$ is the set of internal edges of $\mathcal{T}$. An internal edge $e$ represented by an edges $\gamma_e$ is incident to the two vertices of $\Gamma_\mathcal{T}$ corresponding to the two ideal triangles incident to $\gamma_e$. Self-folded triangles give rise to loops in $\Gamma_\mathcal{T}$.
\item The set of external edges of $\Gamma_\mathcal{T}$ is the set of boundary edges of $\mathcal{T}$. Such an external edge is incident to the vertex of $\Gamma_\mathcal{T}$ corresponding to the ideal triangle which it is a side of.
\item Given a vertex $v$ of $\Gamma_\mathcal{T}$, the cyclic order of $\on{H}(v)$ is given by the counter-clockwise cyclic order of the edges of the corresponding ideal triangle of $\mathcal{T}$.
\end{itemize}
We call $\Gamma_\mathcal{T}$ the dual ribbon graph of $\mathcal{T}$. 
\end{definition}

\begin{example}\label{d2ex}
We depict an ideal triangulation of the once-punctured 2-gon and its dual ribbon graph.

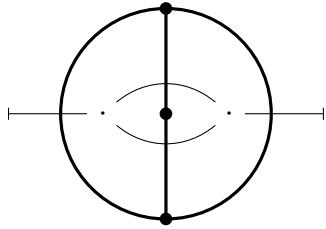
\begin{center}
\begin{tikzpicture}[scale=0.7]
    \node (0) at (-6,2){};
    \node (1) at (-6,0){};
    \node (2) at (-6,-2){};
    \fill (0) circle (0.12); 
    \fill (1) circle (0.12); 
    \fill (2) circle (0.12); 
    \draw [very thick](1) circle (2); 
    \draw [very thick](-6,2) -- (-6,0);
    \draw [very thick](-6,0) -- (-6,-2);
    \node (3) at (-7.2,0){\large $\cdot$};
    \node (4) at (-4.8,0){\large $\cdot$};
    \draw (3) to[bend left=40] (4);
    \draw (3) to[bend right=40] (4);    
    \draw [|-](-3,0) to (-4.5,0);
    \draw [|-](-9,0) to (-7.5,0);
\end{tikzpicture}
\end{center}
\end{example}

Ribbon graphs can be glued along their external edges.

\begin{construction}\label{gluerem}
Let $\Gamma'$ and $\Gamma''$ be ribbon graphs and let $I$ be a finite set and $i':I\rightarrow (\Gamma')_1^\partial$ and $i'':I\rightarrow (\Gamma'')_1^\partial$ injective maps. Then there exists a ribbon graph $\Gamma$ satisfying
\begin{itemize}
\item $\Gamma_0=\Gamma'_0\cup \Gamma''_0$,
\item $\Gamma_1= \Gamma'_1 \amalg_I \Gamma''_1$,
\item that the cyclic order of $\on{H}(v)$, with $v\in \Gamma_0'\subset \Gamma_0$ a vertex, is given by the cyclic order determined by the ribbon graph $\Gamma'$. Analogously, the cyclic order of $\on{H}(v)$, with $v\in \Gamma_0''\subset \Gamma_0$, is determined by $\Gamma''$.
\end{itemize}
We call $\Gamma$ the gluing of $\Gamma'$ and $\Gamma''$ along $I$. Note that there exists an equivalence of posets $\on{Exit}(\Gamma)\simeq \on{Exit}(\Gamma')\amalg_{I}\on{Exit}(\Gamma'')$.
\end{construction}

\begin{example}\label{Tglrem}
Let $\mathcal{T}$ be an ideal triangulation of a surface. To each ideal triangle $F_i$ we associate a ribbon graph $\Gamma_i$ as follows.
\begin{itemize}
\item If $F_i$ is not a self-folded triangle, then $\Gamma_i$ is the following ribbon graph.
\[
\begin{tikzcd}
                               & {} \arrow[d, no head, maps to] &                                \\
{} \arrow[r, no head, maps to] & \cdot                         & {} \arrow[l, no head, maps to]
\end{tikzcd}
\]
\item If $F_i$ is a self-folded triangle, then $\Gamma_i$ is the following ribbon graph.
\[
\begin{tikzcd}
{} \arrow[r, no head, maps to] & \cdot \arrow[no head, loop, distance=2em, in=35, out=325]
\end{tikzcd}
\]
\end{itemize}
Then $\Gamma_\mathcal{T}$ is the gluing of the ribbon graphs $\Gamma_i$ along their external edges determined by the incidence of the ideal triangles.
\end{example}

\subsection{Parametrized perverse schobers}\label{sec4.2}

We essentially define a parametrized perverse schober as a collection of the local models of \Cref{locmod} arising from spherical adjunctions which are suitably glued together along a ribbon graph. This approach is a categorification of the description of perverse sheaves on a surface given in \cite{KS16}. 

\begin{definition}\label{rcdef}
A perverse schober parametrized by a ribbon graph $\Gamma$ is defined to be a functor $\mathcal{F}:\on{Exit}(\Gamma)\rightarrow \on{St}$, subject to the following condition. For each vertex $p\in \Gamma_0$ there exists a spherical functor $F_v:\mathcal{V}^1_{F}\rightarrow \mathcal{N}_{F}$ and a choice of equivalence of posets $C_n\simeq \on{Exit}(\Gamma)_{p/}$, respecting the cyclic ordering of $\{1,\dots,n\}$ and $\on{H}(v)$, such that the restriction of $\mathcal{F}$ to the ribbon corolla $C_n\simeq \on{Exit}(\Gamma)_{p/}$ is equivalent to $\mathcal{G}_n(F_v)$ as objects in $\on{Fun}(C_n,\on{St})$.

We denote by $\mathfrak{P}(\Gamma)$ the full subcategory of the functor category $\on{Fun}(\on{Exit}(\Gamma),\on{St})$ spanned by perverse schobers. 
\end{definition}

\begin{remark}
In accordance with \Cref{sec3.1}, we call $\mathcal{V}^1_{F_v}$ the $\infty$-category of vanishing cycles at $v$ and $\mathcal{N}_{F_v}$ the $\infty$-category of nearby cycles at $v$. 

It follows from the definition of parametrized perverse schober and the assumption that the ribbon graph is connected, that for any two vertices $v,v'$ of $\Gamma$ there exists an equivalence of $\infty$-categories $\mathcal{N}_{F_v}\simeq \mathcal{N}_{F_{v'}}$. The corresponding $\infty$-category, specified up to equivalence, is called the generic stalk of $\mathcal{F}$.
\end{remark}

\begin{notation}\label{notation1}
We will use a graphical notation for perverse schobers parametrized by ribbon graphs similar to the graphical notation for ribbon graphs introduced in \Cref{ribnot}. We denote a parametrized perverse schober by specifying the spherical functor at each vertex of the corresponding ribbon graph and specifying the functor associated to each non-identity morphism in the exit path category.
\end{notation}

\begin{example}
Let $F:\mathcal{V}^1_{F}\rightarrow \mathcal{N}_F$ be a spherical functor and $T:\mathcal{N}_F\rightarrow \mathcal{N}_F$ some autoequivalence. The diagram 
\begin{equation}\label{schobpic}
\begin{tikzcd}
                                                                                                                     & F \arrow[d, "{(T\circ \varrho_1,\varrho_1)}", no head] &                                        \\
0_{\mathcal{N}_F} \arrow[r, "{(\varrho_1,\varrho_2)}", no head] \arrow["{(\varrho_3,\varrho_2)}"', no head, loop, distance=2em, in=215, out=145] & 0_{\mathcal{N}_F}                              & {} \arrow[l, "\varrho_3"', no head, maps to]
\end{tikzcd} 
\end{equation}
corresponds to the parametrized perverse schober given by the following $\on{Exit}(\Gamma)$-indexed diagram in $\on{St}$.
\[
\begin{tikzcd}
            &                                                                                                                      &             & \mathcal{V}^1_{F} \arrow[d, "T\circ F"]                                                              &             \\
            &                                                                                                                      &             & \mathcal{N}_F                                                                                    &             \\
\mathcal{N}_F & {\mathcal{V}^3_{0_{\mathcal{N}_F}}} \arrow[r, "\varrho_1"] \arrow[l, "\varrho_3", bend left] \arrow[l, "\varrho_2"', bend right] & \mathcal{N}_F & {\mathcal{V}^3_{0_{\mathcal{N}_F}}} \arrow[u, "\varrho_1"'] \arrow[l, "\varrho_2"'] \arrow[r, "\varrho_3"] & \mathcal{N}_F
\end{tikzcd}
\] 
\end{example}

The next lemma shows that parametrized perverse schobers can be glued along external edges.

\begin{lemma}\label{sgluelem}
Let $\Gamma'$ and $\Gamma''$ be ribbon graphs, $I$ be a finite set and $i':I\rightarrow (\Gamma')_1^\partial$ and $i'':I\rightarrow (\Gamma'')_1^\partial$ be injective maps. Denote by $\Gamma$ the glued ribbon graph described in \Cref{gluerem}. Consider the functors $\on{ev}':\mathfrak{P}(\Gamma')\rightarrow \on{Fun}(I,\on{St})$ and $\on{ev}'':\mathfrak{P}(\Gamma'')\rightarrow \on{Fun}(I,\on{St})$, given by the restriction functors along the inclusions $I\rightarrow \on{Exit}(\Gamma')$ and $I\rightarrow \on{Exit}(\Gamma'')$, respectively. There exists a pullback diagram in $\on{Cat}_\infty$ as follows.
\begin{equation}\label{gluediag}
\begin{tikzcd}
{\mathfrak{P}(\Gamma)} \arrow[r] \arrow[d] \arrow[rd, "\lrcorner", phantom, near start] & {\mathfrak{P}(\Gamma'')} \arrow[d, "\on{ev}''"] \\
{\mathfrak{P}(\Gamma')} \arrow[r, "\on{ev}'"]                                         & {\on{Fun}(I,\on{St})}                     
\end{tikzcd}
\end{equation}
\end{lemma}

\begin{proof}
Applying the functor $\on{Fun}(\mhyphen,\on{St})$ to the pushout diagram in $\on{Cat}_\infty$
\[
\begin{tikzcd}
I \arrow[r] \arrow[d] \arrow[rd, "\ulcorner", phantom, near end] & {\on{Exit}(\Gamma'')} \arrow[d] \\
{\on{Exit}(\Gamma')} \arrow[r]                   & {\on{Exit}(\Gamma)}                
\end{tikzcd}
\]
yields the following pullback diagram in $\on{Cat}_\infty$.
\[
\begin{tikzcd}
{\on{Fun}(\on{Exit}(\Gamma),\on{St})} \arrow[r] \arrow[d] \arrow[rd, "\lrcorner", phantom, near start] & {\on{Fun}(\on{Exit}(\Gamma''),\on{St})} \arrow[d] \\
{\on{Fun}(\on{Exit}(\Gamma'),\on{St})} \arrow[r]                                         & {\on{Fun}(I,\on{St})}                                    
\end{tikzcd}
\]
The statement that the diagram \eqref{gluediag} is pullback follows from the following observation: an element $\mathcal{F} \in \on{Fun}(\on{Exit}(\Gamma),\on{St})$ lies in $\mathfrak{P}(\Gamma)$ if and only if its restriction to $\on{Fun}(\on{Exit}(\Gamma'),\on{St})$ and $\on{Fun}(\on{Exit}(\Gamma''),\on{St})$ lie in $\mathfrak{P}(\Gamma')$ and $\mathfrak{P}(\Gamma'')$, respectively.
\end{proof}

\subsection{Global sections and duality}\label{sec4.3}

\begin{definition}\label{glscdef}
Let $\Gamma$ be a ribbon graph and let $\mathcal{F}:\on{Exit}(\Gamma)\rightarrow \on{St}$ be a $\Gamma$-parametrized perverse schober. 
\begin{itemize}
\item We call the stable $\infty$-category $\mathcal{H}(\Gamma,\mathcal{F}):=\on{lim}\mathcal{F}$ the $\infty$-category of global sections of $\mathcal{F}$. Global sections form a functor 
\[\mathcal{H}(\Gamma,\mhyphen):\mathfrak{P}(\Gamma)\rightarrow \on{St}\,.\]
\item For $e\in \Gamma_1$, we denote by $\on{ev}_e:\mathcal{H}(\Gamma,\mathcal{F})\rightarrow \mathcal{F}(e)$ the evaluation functor contained in the limit diagram defining $\mathcal{H}(\Gamma,\mathcal{F})$. Given a sub-ribbon graph $\Gamma'\subset \Gamma$, we denote by $\mathcal{H}_{\Gamma'}(\Gamma,\mathcal{F})$ the full subcategory of $\mathcal{H}(\Gamma,\mathcal{F})$ of global sections $X$ such that $\on{ev}_e(X)=0$ for all edges $e\in \Gamma_1\backslash \Gamma_1'$. We call $\mathcal{H}_{\Gamma'}(\Gamma,\mathcal{F})$ the $\infty$-category of sections of $\mathcal{F}_\mathcal{T}$ with support on $\Gamma'$.
\end{itemize}
\end{definition}

\begin{definition}\label{dsdef}
Let $\Gamma$ be a ribbon graph. We denote $\on{Entry}(\Gamma):=\on{Exit}(\Gamma)^{op}$. Given a $\Gamma$-parametrized perverse schober $\mathcal{F}$, we call the right adjoint, respectively, left adjoint diagrams 
\[ \mathbb{D}^R\mathcal{F},~\mathbb{D}^L\mathcal{F}:\on{Entry}(\Gamma)\longrightarrow \on{St}\] 
the right dual respectively left dual of $\mathcal{F}$.
\end{definition}

\begin{remark}\label{gdrem}
Consider the setup of \Cref{dsdef}. \Cref{paralem3} implies that there exists an equivalence $\mathbb{D}^R\mathcal{F}\simeq \mathbb{D}^L \mathcal{F}$ in $\on{Fun}(\on{Entry}(\Gamma),\on{St})$, which restricts on each vertex $v$ with corresponding spherical adjunction $F_v\dashv G_v$ to the twist functor of the spherical adjunction $F_v'\dashv G_v'$.
\end{remark}

\begin{remark}\label{limcolim}
Let $\mathcal{F}$ be a $\Gamma$-parametrized perverse schober. All functors in the image of $\mathcal{F}$ are part of a spherical adjunction and thus preserve all existing limits and colimits. If $\mathcal{F}:\on{Exit}(\Gamma)\rightarrow \on{St}$ takes as values presentable $\infty$-categories, it thus factors through the two forgetful functors $\mathcal{P}r^L_{\on{St}},\mathcal{P}r^R_{\on{St}}\rightarrow \on{St}$. 

Assume that $\mathcal{F}$ takes as values presentable $\infty$-categories. The left and right duals $\mathbb{D}^L\mathcal{F},\mathbb{D}^R\mathcal{F}:\on{Entry}(\Gamma)\rightarrow \on{St}$ both factor through the forgetful functor $\mathcal{P}r^L_{\on{St}}\rightarrow \on{St}$ and there exist equivalences of $\infty$-categories 
\begin{equation}\label{deq1}
\mathcal{H}(\Gamma,\mathcal{F}) \simeq \underset{{\mathcal{P}r^L}}{\on{colim}}\,\mathbb{D}^R\mathcal{F}\,,
\end{equation}
\begin{equation*}
\mathcal{H}(\Gamma,\mathcal{F}) \simeq \underset{{\mathcal{P}r^L}}{\on{colim}}\,\mathbb{D}^L\mathcal{F}\,.
\end{equation*}
We can thus, under the assumption of presentability, equivalently express parametrized perverse schobers and their global sections via their duals. These two perspectives may be seen as a categorification of the two possible perspectives on perverse sheaves, either in terms of sheaves or in terms of cosheaves interchanged by Verdier duality, see \cite{KS19}. 
\end{remark}

\subsection{Contractions of ribbon graphs}\label{sec4.4}

\begin{definition}
Let $\Gamma$ be a ribbon graph and $\mathcal{F}$ a $\Gamma$-parametrized perverse schober. Let $v\in \Gamma_0$ be a vertex of $\Gamma$ and consider the $\infty$-category $\mathcal{V}^1_{F_v}$ of vanishing cycles of $\mathcal{F}$ at $v$. We call $v$ a singularity of $\mathcal{F}$ if there is no equivalence of $\infty$-categories $\mathcal{V}^1_{F_v}\simeq 0$.

Given a subset $V\subset \Gamma_0$, we denote by $\mathfrak{P}(\Gamma,V)$ the full subcategory of $\mathfrak{P}$ spanned by perverse schobers whose singularities lie $V$.
\end{definition}

The goal of this section is to show that that parametrized perverse schobers can be transported along contractions of ribbon graphs which do not contract any edges joining two singularities, such that the $\infty$-categories of global sections are preserved up to equivalence.

\begin{definition}~
\begin{itemize}
\item Let $\Gamma$ be a ribbon graph and $e\in \Gamma_1$ an edge connecting two distinct vertices $v_1,v_2$. Let $\{e_1,e_2\}$ be the orbit representing the edge $e$. We define a ribbon graph $\Gamma'$ with
\begin{itemize}
\item $\Gamma_0'=\Gamma_0/(v_1\sim v_2)$ is the set obtained from $\Gamma_0$ obtained by identifying $v_1$ and $v_2$,
\item $\on{H}_{\Gamma'}=\on{H}_{\Gamma}\backslash \{e_1,e_2\}$,
\item $\tau:\on{H}_{\Gamma'}\rightarrow \on{H}_{\Gamma'}$ is the restriction of $\tau:\on{H}_{\Gamma}\rightarrow \on{H}_{\Gamma}$.
\item $\sigma:\on{H}_{\Gamma'}\rightarrow \Gamma_0'$ is the composite map $:\on{H}_{\Gamma'}\subset \on{H}_{\Gamma}\xrightarrow{\sigma} \Gamma_0\rightarrow \Gamma_0'$.
\item the cyclic order on $\on{H}_{\Gamma'}(v)$ with $v\in \Gamma_0'\backslash [v_1]$ is identical to the cyclic order on $\on{H}_{\Gamma}(v)$. Choose any two linear orders of the elements of $\on{H}_{\Gamma}(v_1)\backslash \{e_1\}$ and $\on{H}_{\Gamma}(v_2)\backslash \{e_2\}$ compatible with the given cyclic ordering. Consider the total order on 
\[ \on{H}_{\Gamma'}([v_1])=\big(\on{H}_{\Gamma}(v_1)\backslash \{e_1\}\big)\cup \big(\on{H}_{\Gamma}(v_2)\backslash \{e_2\}\big)\] 
which restricts to the given total orders on $\on{H}_{\Gamma}(v_1)\backslash \{e_1\},\on{H}_{\Gamma}(v_2)\backslash \{e_2\}$ and such that all elements of $\on{H}_{\Gamma}(v_2)\backslash \{e_2\}$ follow the elements in $\on{H}_{\Gamma}(v_1)\backslash \{e_1\})$. We let the cyclic order on $\on{H}_{\Gamma'}([v_1])$ to be the cyclic order induced by the above total order in the sense of \Cref{ordrem}.
\end{itemize}

We call $\Gamma'$ the edge contraction of $\Gamma$ at $e$.
\item Let $\Gamma$ and $\Gamma'$ be ribbon graphs. We say that there exists a contraction from $\Gamma$ to $\Gamma'$ if $\Gamma'$ is obtained as a (finitely many times) repeated edge contraction of $\Gamma$. We write $c:\Gamma\rightarrow \Gamma'$. 
\end{itemize}
\end{definition}

\begin{remark}\label{ordrem}
A total order on a finite set $\on{H}$ with cardinality $n$ can be defined as a bijection $\phi:\{1,\dots,n\}\simeq H$. Such a total order induces a cyclic order, where $\phi(i+1)$ follows $\phi(i)$ if $i\neq n$ and $\phi(1)$ follows $\phi(n)$. 
\end{remark}

\begin{lemma}\label{locglem}
Let $F:\mathcal{V}^1_{F}\leftrightarrow \mathcal{N}_F:G$ be a spherical adjunction. Let $m,n\geq 1$ and consider the stable $\infty$-categories $\mathcal{V}^m_{0_{\mathcal{N}_F}}$ and $\mathcal{V}^n_F$ with categorified restriction maps $\varrho^1_i:\mathcal{V}^m_{0_{\mathcal{N}_F}}\rightarrow \mathcal{N}_F$, with $i=1,\dots,m$, respectively $\varrho^2_j:\mathcal{V}^n_{F}\rightarrow \mathcal{N}_F$, with $j=1,\dots,n$. 
\begin{enumerate}
\item There exists a pullback diagram in $\on{Cat}_\infty$ as follows. 
\begin{equation}\label{Vpbsq1}
\begin{tikzcd}
\mathcal{V}_F^{n+m-2} \arrow[r, "\alpha"] \arrow[d, "\beta"] \arrow[rd, "\lrcorner", phantom, near start] & \mathcal{V}_F^{n} \arrow[d, "\varrho^2_{1}"] \\
\mathcal{V}_{0_{\mathcal{N}_F}}^{m} \arrow[r, "\varrho^1_m"]                                            & \mathcal{N}_F                         
\end{tikzcd}
\end{equation}
\item Denote by $\varrho_1,\dots,\varrho_{n+m-2}:\mathcal{V}_F^{n+m-2}\rightarrow \mathcal{N}_F$ the categorified restriction maps. There exist equivalences of functors $\varrho_j\simeq \varrho^1_j\circ \beta$ and $\varrho_{i+m-2}\simeq \varrho^2_i\circ \alpha$ for $j=1,\dots,m-1$ and $i=2,\dots n$.
\end{enumerate}
\end{lemma}
\begin{proof}
Let $D_1:\Delta^{m-2}\rightarrow \on{St}$ be the constant diagram with value $\mathcal{N}_F$ and $D_2:\Delta^{n-1}\rightarrow \on{St}$, $D:\Delta^{n+m-3}\rightarrow \on{St}$ be the diagrams obtained from the sequences of composable functors
\[ \mathcal{V}^1_F\xlongrightarrow{G}\mathcal{N}_F\xlongrightarrow{\on{id}}\dots\xlongrightarrow{\on{id}}\mathcal{N}_F\,.\] 
The diagram $D$ restricts to the diagrams $D_1$ and $D_2$ on $\Delta^{\{0,\dots,n-1\}}$ and $\Delta^{\{n-1,\dots,n+m-3\}}$, respectively. The inclusion functor $\Delta^{\{0,\dots,n-1\}}\amalg_{\Delta^{\{n-1\}}}\Delta^{\{n-1,\dots,n+m-3\}}\rightarrow \Delta^{n+m-3}$ is inner anodyne. It follows that the restriction functor 
\[ \on{res}:\on{Fun}(\Delta^{n+m-3},\Gamma(D))\rightarrow \on{Fun}(\Delta^{n-1},\Gamma(D_1))\times_{\mathcal{N}_F} \on{Fun}(\Delta^{m-2},\Gamma(D_2))\]
is a trivial fibration, from which we obtain a further trivial fibration
\[ \on{Fun}_{\Delta^{n+m-3}}(\Delta^{n+m-3},\Gamma(D))\rightarrow \on{Fun}_{\Delta^{n-1}}(\Delta^{n-1},\Gamma(D_1))\times_{\mathcal{N}_F} \on{Fun}_{\Delta^{m-2}}(\Delta^{m-2},\Gamma(D_2))\,.\]
Using the equivalences of $\infty$-categories
\begin{align*} 
\mathcal{V}_F^n & \simeq \on{Fun}_{\Delta^{n-1}}(\Delta^{n-1},\Gamma(D_2))\,,\\
\mathcal{V}_{0_{{\mathcal{N}_F}}}^m & \simeq \on{Fun}_{\Delta^{m-2}}(\Delta^{m-2},\Gamma(D_1))\,,\\
\mathcal{V}^{n+m-3}_F & \simeq \on{Fun}_{\Delta^{n+m-3}}(\Delta^{n+m-3},\Gamma(D))\,,
\end{align*} 
it follows that there exists a pullback diagram of the form \eqref{Vpbsq1}. The functors $\alpha[2-m]$ and $\beta$ in this pullback diagram are given by the restriction functors to the first $m-1$ and last $n$ components, respectively. The description of the categorified restriction maps can thus be checked directly.
\end{proof}

\begin{construction}\label{locgrem}
Consider the setup of \Cref{locglem} and the following diagram,
\begin{equation}\label{Vpbsq2}
\begin{tikzcd}
                                                 & \mathcal{V}^n_F \arrow[d, "\varrho^2_j"] \\
\mathcal{V}^m_{0_{\mathcal{N}_F}} \arrow[r, "\varrho^1_i"] & \mathcal{N}_F                       
\end{tikzcd}
\end{equation}
where $1\leq i\leq m$ and $1\leq j\leq n$ are arbitrary. We can use the paracyclic twist functors $\big(T_{\mathcal{V}^n_F}\big)^{1-j}$ and $\big( T_{\mathcal{V}^m_{0_{\mathcal{N}_F}}}\big)^{m-i}$, see \Cref{sec3.2} and \Cref{paralem4} in particular, to find a natural equivalence between the diagram \eqref{Vpbsq2} and the following diagram. 

\begin{equation}\label{Vpbsq3}
\begin{tikzcd}
                                                 & \mathcal{V}^n_F \arrow[d, "\varrho^2_1"] \\
\mathcal{V}^m_{0_{\mathcal{N}_F}} \arrow[r, "\varrho^1_m"] & \mathcal{N}_F                       
\end{tikzcd}
\end{equation}

The limits of the diagrams \eqref{Vpbsq2} and \eqref{Vpbsq3} are therefore both equivalent to $\mathcal{V}^{n+m-2}_F$. \Cref{paralem4} also shows that under this equivalence the resulting categorified restriction maps $\varrho_i:\mathcal{V}^{n+m-2}_F\rightarrow \mathcal{N}_F$ are cyclically permuted and may each further change by post-composition with an autoequivalence of the form $(T_{\mathcal{N}_F}[n-1])^l$ for some $l\in \mathbb{Z}$. 
\end{construction}

\begin{proposition}\label{conprop}
Let $c:\Gamma\rightarrow \Gamma'$ be a contraction of ribbon graphs and let $V\subset \Gamma_0$ be a subset such that no two vertices in $V$ are contracted to a single vertex by $c$. There is a functor of $\infty$-categories $c_*:\mathfrak{P}(\Gamma,V)\rightarrow \mathfrak{P}(\Gamma')$ making the following diagram commute. 
\begin{equation}\label{glsecdiag}
\begin{tikzcd}
{\mathfrak{P}(\Gamma,V)} \arrow[rr, "c_*"] \arrow[rd, "{\mathcal{H}(\Gamma,\mhyphen)}"'] &         & {\mathfrak{P}(\Gamma')} \arrow[ld, "{\mathcal{H}(\Gamma',\mhyphen)}"] \\
                                                                & \on{St} &                                                  
\end{tikzcd}
\end{equation}
\end{proposition}

\begin{proof}
It suffices to show the statement in the case that $c$ is the contraction of an edge $e\in \Gamma_1$ connecting two vertices $v_1,v_2$ such that $v_1\notin V$. The edge contraction $c$ induces a functor $\on{Exit}(c):\on{Exit}(\Gamma)\rightarrow \on{Exit}(\Gamma')$ determined by mapping
\begin{itemize}
\item $x\in \Gamma_0\backslash \{v_1,v_2\}\subset \on{Exit}(\Gamma)$ to $x\in \Gamma_0\backslash \{v_1,v_2\}\subset \on{Exit}(\Gamma')$,
\item $v_1,v_2\in \Gamma_0 \subset \on{Exit}(\Gamma)$ to $[v_1]$,
\item $f\in \Gamma_1\backslash \{e\}\subset \on{Exit}(\Gamma)$ to $f\in \Gamma_1\backslash \{e\}\subset \on{Exit}(\Gamma')$ and 
\item $e\in \Gamma_1\subset \on{Exit}(\Gamma)$ to $[v_1]\in \Gamma_0'\subset \on{Exit}(\Gamma')$.
\end{itemize}
We define $\mathcal{E}$ to be the poset determined by the following properties.
\begin{itemize}
\item There exist fully faithful functors $\on{Exit}(\Gamma'),\on{Exit}(\Gamma)\rightarrow \mathcal{E}$.
\item The induced functor $\on{Exit}(\Gamma')\amalg\on{Exit}(\Gamma)\rightarrow \mathcal{E}$ is bijective on objects
\item For $x'\in \on{Exit}(\Gamma')$ and $x\in \on{Exit}(\Gamma)$, there exists a unique morphism from $x'$ to $x$ in $\mathcal{E}$ if and only if there exists a morphism $x'\rightarrow \on{Exit}(c)(x)$. There are no morphisms from $x$ to $x'$. 
\end{itemize} 
Note that the poset $\mathcal{E}$ can be equivalently described as the total space of a Cartesian fibration classifying the functor $\on{Exit}(c):\Delta^1\rightarrow \on{Cat}_\infty$.

We define $c_*:\on{Fun}(\on{Exit}(\Gamma),\on{St})\rightarrow \on{Fun}(\on{Exit}(\Gamma),\on{St})$ as the composition of the right Kan extension functor along the inclusion $\on{Exit}(\Gamma)\rightarrow \mathcal{E}$ with the restriction functor to $\on{Exit}(\Gamma')$. It follows from \Cref{locglem} and \Cref{locgrem} that $c_*$ maps $\mathfrak{P}(\Gamma,V)$ to $\mathfrak{P}(\Gamma')$. The commutativity of the diagram \eqref{glsecdiag} follows from right Kan extensions commuting with right Kan extensions.
\end{proof}

\section{Algebraic descriptions of \texorpdfstring{$\mathcal{V}^n_{f^*}$}{Vn}}\label{sec5}

This section provides auxiliary computations to be used in \Cref{sec6}. In \Cref{sec5.1} we study the $\infty$-category $\on{Fun}(S^n,\mathcal{D}(k))$ of local systems on the $n$-sphere with values in the derived $\infty$-category of a commutative ring $k$ and the spherical adjunction $f^*:\mathcal{D}(k)\leftrightarrow \on{Fun}(S^n,\mathcal{D}(k)):f_*$. We show that there is an equivalence of $\infty$-categories $\on{Fun}(S^n,\mathcal{D}(k))\simeq \mathcal{D}(k[t_{n-1}])$, where $k[t_{n-1}]$ denotes the polynomial algebra with generator in degree $|t_{n-1}|=n-1$.
In \Cref{sec5.2} we describe the perverse schober on the disc obtained from the spherical adjunction $f^*\dashv f_*$.

\subsection{Local systems on spheres}\label{sec5.1}
In \cite{Chr20}, we showed the following.

\begin{proposition}\label{locprop1}
For $n\geq 0$ let $S^n$ denote the singular set of the topological $n$-sphere and consider the map $f:S^n\rightarrow \ast$. Let further $\mathcal{D}$ be a stable $\infty$-category and 
\begin{equation}\label{fstar} 
f^*:\mathcal{D}\rightarrow \on{Fun}(S^n,\mathcal{D}) 
\end{equation}
be the pullback functor with right adjoint $f_*$, given by the limit functor. The adjunction $f^*\dashv f_*$ is spherical with twist functor $T_\mathcal{D}\simeq [-n]$.
\end{proposition}

We call the $\infty$-category $\on{Fun}(S^n,\mathcal{D})$ the $\infty$-category of local systems on $S^n$ with values in $\mathcal{D}$. It is well known that if $\mathcal{D}=\mathcal{D}(k)$ for some commutative ring $k$, the $\infty$-category of local systems on $S^1$ with values in $\mathcal{D}(k)$ is equivalent to the $\infty$-category $\mathcal{D}(k[t,t^{-1}])$, where $k[t,t^{-1}]$ is the ring of Laurent polynomials. In this section we show the existence of an equivalence of $k$-linear $\infty$-categories $\on{Fun}(S^n,\mathcal{D}(k))\simeq \mathcal{D}(k[t_{n-1}])$ for $n>2$, where $k[t_{n-1}]$ the polynomial algebra with generator in degree $|t_{n-1}|=n-1$. In \Cref{specsec}, we will show that this description also generalizes to $\mathcal{D}=\on{RMod}_R$ for $R$ an $\mathbb{E}_\infty$-ring spectrum. We will end this section by an explicit description of the cotwist functor of the adjunction $f^*\dashv f_*$.\medskip

We begin with the following observation.
\begin{remark}\label{rlinrem}
Let $Z$ be a simplicial set. The $\infty$-category $\on{Fun}(Z,\mathcal{D}(k))$ admits a symmetric monoidal structure, such that the pullback functor $h^*$ along $h:Z\rightarrow \ast$ is a symmetric monoidal functor, see for example \cite[Section 3.3]{Chr20}. We can thus consider $\on{Fun}(Z,\mathcal{D}(k))$ as a left module in $\mathcal{P}r^L_{\on{St}}$ over itself and the functor $h^*$ as a morphism of algebra objects in $\mathcal{P}r^L$. Pulling back along $h^*$ provides $\on{Fun}(Z,\mathcal{D}(k))$ with the structure of a left module over $\mathcal{D}(k)$ and thus with the structure of a left-tensoring over $\mathcal{D}(k)$. This shows that $\on{Fun}(Z,\mathcal{D}(k))$ is a $k$-linear $\infty$-category such that the functor $h^*$ is $k$-linear.
\end{remark}

We let $L$ denote the simplicial set consisting of a single vertex and a single non-degenerate $1$-simplex. We use \Cref{rlinrem} to lift $\on{Fun}(L,\mathcal{D}(k))$ to a $k$-linear $\infty$-category. Denote by $k[t_0]$ the polynomial algebra with $|t_0|=0$.

\begin{lemma}\label{l1lem}
Consider the pullback functor $g^*:\mathcal{D}(k)\rightarrow \on{Fun}(L,\mathcal{D}(k))$ along $g:L\rightarrow \ast$. There exists an equivalence of $k$-linear $\infty$-categories $\on{Fun}(L,\mathcal{D}(k))\simeq \mathcal{D}(k[t_0])$ such that the following diagram commutes.
\begin{equation}\label{tri1}
\begin{tikzcd}
                                            & \mathcal{D}(k) \arrow[ld, "g^*"'] \arrow[rd, "\phi^*"] &                      \\
{\on{Fun}(L,\mathcal{D}(k))} \arrow[rr, "\simeq"] &                                                     & {\mathcal{D}(k[t_0])}
\end{tikzcd}
\end{equation}
Here $\phi^*$ denotes the pullback functor along the morphism of dg-algebras $k[t_0]\xrightarrow {t_0\mapsto 1} k$.
\end{lemma}

\begin{proof}
We observe that $\on{Fun}(L,\mathcal{D}(k))$ admits a compact generator $X$, given by the diagram $k[t_0]\xrightarrow{\cdot  t_0}k[t_0]$ in $\mathcal{D}(k)$. The homology of the $k$-linear endomorphism algebra $\on{End}_k(X)$ is concentrated in degree $0$ and a direct computation shows that it is equivalent to the $k$-vector space $k[t_{0}]$. It follows that $\on{End}_k(X)$ is formal. The composition of morphisms induces the polynomial algebra structure on $k[t_{0}]$. In total, this shows that $\on{End}_k(X)$ is quasi-isomorphic as a dg-algebra to $k[t_{0}]$. The existence of the equivalence of $k$-linear $\infty$-categories $\epsilon:\on{Fun}(L,\mathcal{D}(k))\simeq \mathcal{D}(k[t_0])$ thus follows from \Cref{genlem1}. The $k$-linear functors $g^*,\phi^*$ are fully determined by $g^*(k)=(k\xrightarrow{id}k)$, respectively, $\phi^*(k)$, see \cite[Section 4.8.4]{HA}. The apparent equivalence $\epsilon g^*(k)\simeq \phi^*(k)$ thus implies the commutativity of the diagram \eqref{tri1}. 
\end{proof}

For $n\geq 1$, we denote by $i:\ast\rightarrow S^n$ the inclusion functor of any vertex and by $i^*:\on{Fun}(S^n,\mathcal{D}(k))\rightarrow \mathcal{D}(k)$ the associated pullback functor. We use \Cref{rlinrem} to lift $i^*$ and the functor $f^*$ from \eqref{fstar} to $k$-linear functors. We further denote by $g_!,f_!$ and $i_!$ the left adjoints of $g^*,f^*$, respectively, $i^*$.

\begin{lemma}\label{pbpulem}~
\begin{enumerate}
\item There exists a pushout diagram in $\on{LinCat}_k$ as follows.
\begin{equation}\label{pupbeq1}
\begin{tikzcd}
\on{Fun}(L,\mathcal{D}(k))        \arrow[r, "g_!"] \arrow[d, "g_!"'] \arrow[rd, "\ulcorner", phantom, near end] & \mathcal{D}(k) \arrow[d, "i_!"] \\
\mathcal{D}(k) \arrow[r, "i_!"']                                               & \on{Fun}(S^2,\mathcal{D}(k))              
\end{tikzcd}
\end{equation}
\item Let $n\geq 2$. There exists a pushout diagram in $\on{LinCat}_k$ as follows.
\begin{equation}\label{pupbeq2}
\begin{tikzcd}
\on{Fun}(S^{n-1},\mathcal{D}(k)) \arrow[r, "f_!"] \arrow[d, "f_!"'] \arrow[rd, "\ulcorner", phantom, near end] & \mathcal{D}(k)\arrow[d, "i_!"] \\
\mathcal{D}(k) \arrow[r, "i_!"']                                               & \on{Fun}(S^n,\mathcal{D}(k))              
\end{tikzcd}
\end{equation}
\end{enumerate}
\end{lemma}

\begin{proof}
We begin by showing statement 2. Consider the following pushout diagram of spaces.
\[
\begin{tikzcd}
S^{n-1} \arrow[r, "f"] \arrow[d, "f"'] \arrow[rd, "\ulcorner", phantom, near end] & \ast \arrow[d, "i"] \\
\ast \arrow[r, "i"']                                           & S^n                
\end{tikzcd}
\]
The above diagram is also pushout in $\on{Cat}_\infty$. Applying the limit preserving functor $\on{Fun}(\mhyphen,\mathcal{D}(k)):\on{Cat}_\infty^{op}\rightarrow \on{Cat}_\infty$ maps this pushout diagram to the following pullback diagram in $\mathcal{P}r^R$.
\[
\begin{tikzcd}
\on{Fun}(S^{n},\mathcal{D}(k)) \arrow[r, "i^*"] \arrow[d, "i^*"'] \arrow[rd, "\lrcorner", phantom, near start] & \mathcal{D}(k)\arrow[d, "f^*"] \\
\mathcal{D}(k) \arrow[r, "f^*"']                                               & \on{Fun}(S^{n-1},\mathcal{D}(k))              
\end{tikzcd}
\]
The left adjoint diagram is the diagram \eqref{pupbeq2} and thus pushout in $\on{LinCat}_k$.

We now show statement 1. The geometric realization of $L$ is equivalent to the topological $1$-sphere. There thus exists a morphism of simplicial sets $L\rightarrow S^1$ such that the limit functor $g_*=\on{lim}:\on{Fun}(L,\mathcal{D}(k))\rightarrow \mathcal{D}(k)$ restricts via the pullback functor $i^*:\on{Fun}(S^1,\mathcal{D}(k))\rightarrow \on{Fun}(L,\mathcal{D}(k))$ to the limit functor $f_*$. The left adjoint $g^*:\mathcal{D}(k)\rightarrow \on{Fun}(L,\mathcal{D}(k))$ thus factors through $\on{Fun}(S^1,\mathcal{D}(k))$. It thus follows from the explicit model for limits in $\on{Cat}_\infty$ that the right adjoint diagram of diagram \eqref{pupbeq1} is pullback in $\mathcal{P}r^R$. It follows that the diagram \eqref{pupbeq1} is pushout in $\on{LinCat}_k$.
\end{proof}

\begin{proposition}\label{genlem}
Let $n\geq 2$. There exists an equivalence of $k$-linear $\infty$-categories 
\begin{equation}\label{loceq}
\on{Fun}(S^n,\mathcal{D}(k))\simeq \mathcal{D}(k[t_{n-1}])\,,
\end{equation} 
such that the following diagram in $\on{LinCat}_k$ commutes. 
\begin{equation}\label{eqdiag1}
\begin{tikzcd}[row sep=small]
{\on{Fun}(S^n,\mathcal{D}(k))} \arrow[rr, "\simeq"] \arrow[rd, "i^*"] &                                                     & {\mathcal{D}(k[t_{n-1}])} \arrow[ld, "G"] \\
                                                                           & \mathcal{D}(k)\arrow[ld, "f^*"'] \arrow[rd, "\phi^*"] &                                          \\
{\on{Fun}(S^n,\mathcal{D}(k))} \arrow[rr, "\simeq"]                   &                                                     & {\mathcal{D}(k[t_{n-1}])}                
\end{tikzcd}
\end{equation}
Here $G$ denoted the monadic functor and $\phi^*$ the pullback functor along the morphism of dg-algebras $\phi:k[t_{n-1}]\xrightarrow{t_{n-1}\mapsto 0} k$.
\end{proposition}

\begin{proof}
We observe that the composition of the autoequivalence of dg-algebras $k[t_0]\xrightarrow{t_0\mapsto t_0-1}k[t_0]$ with the morphism of dg-algebras $k[t_0]\xrightarrow{t_0\mapsto 1}k[t_0]$ is given by $k[t_0]\xrightarrow{t_0\mapsto 0}k$. It therefore follows from \Cref{l1lem}, that for $n=2$ the pushout square \eqref{pupbeq2} is equivalent to the image under $\mathcal{D}(\mhyphen)$ of the following homotopy pushout diagram of dg-categories with a single object.
\[
\begin{tikzcd}
{k[t_0]} \arrow[r, "t_0\mapsto 0"] \arrow[d, "t_0\mapsto t_0"'] \arrow[rd, "\ulcorner", phantom, near end] & k \arrow[d] \\
{(k[t_0,t_1],d)} \arrow[r]                                                   & {k[t_1]}                  
\end{tikzcd}
\]
Above $k[t_0,t_1]$ denotes the dg-algebra freely generated by $t_0$ and $t_1$ in degrees $0$ and $1$ and differential $d$ determined by $d(t_1)=t_0$, $d(t_0)=0$. It follows that there exists an equivalence of $k$-linear $\infty$-categories $\on{Fun}(S^2,\mathcal{D}(k))\simeq \mathcal{D}(k[t_1])$, making the the upper half of \eqref{eqdiag1} commute. The $k$-linear functors $\phi^*:\mathcal{D}(k)\rightarrow \mathcal{D}(k[t_{1}])$ and $f^*:\mathcal{D}(k)\rightarrow \on{Fun}(S^2,\mathcal{D}(k))$ are determined by $\phi^*(k)$ and $f^*(k)$, respectively. The homology of the chain complex underlying the right $k[t_1]$-module $\phi^*(k)$ is concentrated in degree $0$, given by $k$. Using that the upper half of \eqref{eqdiag1} commutes and that $i^*f^*(k)=k$, it follows that $f^*(k)$ is also mapped under the equivalence $\on{Fun}(S^2,\mathcal{D}(k))\simeq \mathcal{D}(k[t_{1}])$ to a right $k[t_1]$-module with homology $k$. There exists but a unique right $k[t_1]$-module, up to quasi-isomorphism, with homology $k$. This right $k[t_1]$-module $k$ is equivalent to the module determined by the morphism of dg-algebras $k[t_{1}]\xrightarrow{t_1\mapsto 0}k$. It follows that the lower half of \eqref{eqdiag1} commutes for $n=2$. For $n>3$, one can argue analogously and by induction. The pushout square \eqref{pupbeq2} is equivalent to the image under $\mathcal{D}(\mhyphen)$ of the following homotopy pushout diagram of dg-categories.
\[
\begin{tikzcd}
{k[t_{n-2}]} \arrow[r, "t_{n-2}\mapsto 0"] \arrow[d, "t_{n-2}\mapsto t_{n-2}"'] \arrow[rd, "\ulcorner", phantom, near end] & k \arrow[d] \\
{(k[t_{n-2},t_{n-1}],d)} \arrow[r]                                                   & {k[t_{n-1}]}                  
\end{tikzcd}
\]
Here again $k[t_{n-2},t_{n-1}]$ denotes the freely generated dg-algebra with two generators in degrees $n-2$ and $n-1$ and differential determined by $d(t_{n-1})=t_{n-2}$, $d(t_{n-2})=0$. We find the desired equivalence of $k$-linear $\infty$-categories $\on{Fun}(S^n,\mathcal{D}(k))\simeq \mathcal{D}(k[t_{n-1}])$ making the upper half of \eqref{eqdiag1} commute. Showing that the lower half of \eqref{eqdiag1} commutes is analogous to the case $n=2$.
\end{proof}

\begin{remark}\label{sphrem}
The composite functor $\ast \xrightarrow{i} S^n\xrightarrow{f} \ast $ is an equivalence of spaces. The functor 
\[ \mathcal{D}(k)\xrightarrow{i_!}\on{Fun}(S^n,\mathcal{D}(k))\xrightarrow{f_!}\mathcal{D}(k)\] 
is thus equivalent to $\on{id}_{\mathcal{D}(k)}$. By the sphericalness of $f^*\dashv f_*$, there exists an equivalence $f_*\simeq f_![-n]$ and therefore $f_*i_!(k)\simeq k[-n]$. Note that \Cref{genlem} shows that $i_!(k)$ is  equivalent to the compact generator $k[t_{n-1}]\in \mathcal{D}(k[t_{n-1}])\simeq \on{Fun}(S^n,\mathcal{D}(k))$. 
\end{remark}

We now describe the cotwist functor $T_{\on{Fun}(S^n,\mathcal{D}(k))}$ of the spherical adjunction $f^*\dashv f_*$ for $n\geq 2$. Consider the morphism of dg-algebras $\varphi:k[t_{n-1}]\rightarrow k[t_{n-1}]$ determined by $\varphi(t_{n-1})= (-1)^{n-1}t_{n-1}$. Let $\varphi^*:\mathcal{D}(k[t_{n-1}])\rightarrow \mathcal{D}(k[t_{n-1}])$ be the pullback functor. 

\begin{proposition}\label{ctwprop}
Let $n\geq 2$ and let $k$ be a commutative ring. There exists a commutative diagram in $\on{LinCat}_k$ as follows.
\[
\begin{tikzcd}[column sep=60]
{\on{Fun}(S^n,\mathcal{D}(k))} \arrow[r, "{T_{\on{Fun}(S^n,\mathcal{D}(k))}}"] \arrow[d, "\simeq"]\arrow[d, "\eqref{loceq}"'] & {\on{Fun}(S^n,\mathcal{D}(k))} \arrow[d, "\simeq"]\arrow[d, "\eqref{loceq}"'] \\
{\mathcal{D}(k[t_{n-1}])} \arrow[r, "{\varphi^*[-n]}"]                    & {\mathcal{D}(k[t_{n-1}])}                      
\end{tikzcd}
\]
\end{proposition}

\begin{remark}\label{signrem}
Note that only if $n$ is odd there exists an equivalence of functors $T_{\on{Fun}(S^n,\mathcal{D}(k))}\simeq [-n]$. The functor $T_{\on{Fun}(S^n,\mathcal{D}(k))}[n]$ is otherwise the involution reversing the sign of $t_{n-1}$. 
\end{remark}

\begin{proof}[Proof of \Cref{ctwprop}]
The dg-category of $k[t_{n-1}]$-bimodules is equivalent to the dg-category $\on{dgMod}(k[t_{n-1}]\otimes_k k[t_{n-1}]^{\on{op}})$. The former dg-category thus inherits a model structure from the projective model structure of the latter, whose underlying $\infty$-category is equivalent to the $\infty$-category of $k$-linear endofunctors of $\mathcal{D}(k[t_{n-1}])$. Let $\odot$ denote the multiplication in $k[t_{n-1}]$. We denote by $\widehat{{k[t_{n-1}]}}$ the $k[t_{n-1}]$-bimodule $\widehat{{k[t_{n-1}]}}$ with
\begin{itemize}
\item underlying chain complex $k[t_{n-1}]$, 
\item left action on $a\in \widehat{k[t_{n-1}]}$ determined by $t_{n-1}^i.a=(-1)^{i(n-1)}t_{n-1}^i\odot a$ and 
\item right action on $a\in \widehat{k[t_{n-1}]}$ determined by $a.t_{n-1}^i=a\odot t_{n-1}^i\,$.
\end{itemize}
Note that $\varphi^*\simeq \mhyphen \otimes_{k[t_{n-1}]} \widehat{k[t_{n-1}]}$. We can thus prove the Proposition by showing that the composite of the twist functor $T_{\on{Fun}(S^n,\mathcal{D}(k))}'$ of the spherical adjunction $f_!\dashv f^*$ with  the equivalence \eqref{loceq} is equivalent to $\mhyphen \otimes \widehat{k[t_{n-1}]}[n]$. Using the commutativity of the lower part of diagram \eqref{eqdiag1} in \Cref{genlem}, it suffices to show that the twist functor $T$ of the spherical adjunction $\phi_!\dashv \phi^*$ is equivalent to $\mhyphen \otimes \widehat{k[t_{n-1}]}[n]$. 

Using \Cref{sphrem}, it follows that $\phi^*\phi_!(k[t_{n-1}])\simeq k\in \mathcal{D}(k[t_{n-1}])$ with $k$ the $k[t_{n-1}]$-bimodule determined by the morphism of dg-algebras $\phi$. The $k$-linear functor $\phi^*\phi_!$ is thus equivalent to the functor $\mhyphen \otimes_{k[t_{n-1}]}k$, for a $k[t_{n-1}]$-bimodule $k$. There is but a unique such bimodule, which carries the action $t_{n-1}.1=0=1.t_{n-1}\in k$. A cofibrant replacement of the $k[t_{n-1}]$-bimodule $k$ is given the cone of the morphism of bimodules 
\[ \alpha:\widehat{k[t_{n-1}]}[n-1]\rightarrow k[t_{n-1}]\,.\]
\[ t_{n-1}^i\mapsto t_{n-1}^i\]
To see that $\alpha$ indeed exists, note that by the definition of $\widehat{k[t_{n-1}]}$ and the sign rule for the shift of left modules, see \Cref{sgnrl}, the left action of $k[t_{n-1}]$ on $\widehat{k[t_{n-1}]}[n-1]$ is determined by $t_{n-1}.1=(-1)^{(n-1)+(n-1)}t_{n-1}=t_{n-1}$.  We deduce that the twist functor $T$ is equivalent to the functor given by tensoring with the homotopy pushout in the following diagram of cofibrant $k[t_{n-1}]$-bimodules.
\[
\begin{tikzcd}
{k[t_{n-1}]} \arrow[r] \arrow[d] \arrow[rd, "\ulcorner", phantom, near end] & \on{cone}(\alpha) \arrow[d] \\
0 \arrow[r]                                                       & \widehat{k[t_{n-1}]}[n]                      
\end{tikzcd}
\]
We have shown $T\simeq \mhyphen \otimes_{k[t_{n-1}]}\widehat{k[t_{n-1}]}[n]$, finishing the proof.
\end{proof}

\subsection{\texorpdfstring{$\mathcal{V}^2_{f^*}$}{V2} and \texorpdfstring{$\mathcal{V}^3_{f^*}$}{V3}}\label{sec5.2}

For $n\geq 2$, we consider the spherical adjunction $f^*:\mathcal{D}(k)\leftrightarrow \on{Fun}(S^n,\mathcal{D}(k)):f_*$ of \Cref{locprop1}. The goal of this section is to prove \Cref{algmod1,algmod2} below, describing the parametrized perverse schober on a $2$-gon and $3$-gon obtained from the spherical adjunction $f^*\dashv f_*$.

\begin{proposition}\label{algmod1}
Let $k$ be a commutative ring and $n\geq 2$. Let $D_2$ be the freely generated dg-category with two objects $x,z$ depicted as follows, 
\[
\begin{tikzcd}
z \arrow[rr, "c", bend left] &  & x \arrow[ll, "c^*"]
\end{tikzcd}
\]
with morphisms in degrees $|c|=0$ and $|c^*|=n-1$ and vanishing differentials. Let further $D_3$ be the freely generated dg-category with three objects $x,y,z$ depicted as follows
\[
\begin{tikzcd}
                                               & y \arrow[ld, "a^*"] \arrow[rd, "b", bend left] &                                                \\
x \arrow[ru, "a", bend left] \arrow[rr, "c^*"] &                                                & z \arrow[lu, "b^*"] \arrow[ll, "c", bend left]
\end{tikzcd}
\]
with morphisms in degrees $|b|=|c|=0,~|a|=n-2,~|b^*|=|c^*|=n-1,~|a^*|=1$ and differential determined by 
\[ d(a)=d(b)=d(c)=0\,,\]
\[ d(a^*)=cb,~d(b^*)=ac,~d(c^*)=ba\,.\] 
There exist equivalences of $\infty$-categories
\begin{align} 
\mathcal{N}_{f^*} & \simeq \mathcal{D}(k[t_{n-1}]) \,,\label{model1}\\
\mathcal{V}^2_{f^*} & \simeq \mathcal{D}(D_2)\,,\label{model3}\\
\mathcal{V}^3_{f^*} &\simeq  \mathcal{D}(D_3)\,.\label{model4}
\end{align}
\end{proposition}

\begin{remark}
\Cref{algmod1} shows that dimension $n=2$ is distinguished. Only in dimension $n=2$ do we find the morphisms $a,b,c$ in $D_3$ to all be in degree $0$. 
\end{remark}

\begin{notation}\label{ifn}~
For $\tau=+,-$, denote by $i_1^\tau,i_2^\tau,i_3^\tau:k[t_{n-1}]\rightarrow D_3$ the dg-functors determined by mapping $\ast$ to $x,z$, respectively, $y$, and $t_{n-1}$ to $\tau(cc^*-a^*a)$, $\tau(bb^*-c^*c)$ and $\tau((-1)^{n}aa^*-b^*b)$, respectively. We further set $(-)^n=+$ if $n$ is even and $(-)^n=-$ if $n$ is odd.
\end{notation}

\begin{proposition}\label{algmod2}~
\begin{enumerate}
\item Under the equivalences \eqref{model1} and \eqref{model4}, the functors $\varsigma_2,\varsigma_3:\mathcal{N}_{f^*}\rightarrow \mathcal{V}^3_{f^*}$ are equivalent to the image under $\mathcal{D}(\mhyphen)$ of the dg-functors $i_2^{(-)^{n-1}}$ and $i_3^{(-)^{n}}$, respectively. If $n=2$, then $\varsigma_1$ is equivalent $\mathcal{D}(i_1^{+})$.
\item Suppose that $n=2$. Under the equivalences \eqref{model1} and \eqref{model3}, the functors $\varsigma_1,\varsigma_2:\mathcal{N}_{f^*}\rightarrow \mathcal{V}^2_{f^*}$ are equivalent to the image under $\mathcal{D}(\mhyphen)$ of the dg-functors $k[t_{1}]\rightarrow D_2$ determined by mapping $t_{1}$ to $cc^*$ and $c^*c$, respectively.
\end{enumerate}
\end{proposition}

\begin{lemma}\label{auxlem1}
Let $n\geq 2$.
\begin{enumerate}
\item There exists an equivalence of $\infty$-categories 
\begin{equation}\label{ke1} 
\mathcal{V}^2_{f^*}\simeq \mathcal{D}(\bf A_2)\,,
\end{equation} 
where 
\[ \bf{A_2}=\begin{pmatrix} k & k[-n] \\ 0 & k[t_{n-1}]\end{pmatrix}\] 
is the upper triangular dg-algebra. 
\item There exists an equivalence of $\infty$-categories 
\begin{equation}\label{ke2}
\mathcal{V}^3_{f^*}\simeq \mathcal{D}(\bf A_3)\,,
\end{equation}
where 
\[ \bf{A_3}=\begin{pmatrix} k & k[-n] & 0 \\ 0 & k[t_{n-1}] & k[t_{n-1}] \\ 0 & 0 & k[t_{n-1}]\end{pmatrix}\]
is the upper triangular dg-algebra.
\end{enumerate}
\end{lemma}

\begin{proof}
This follows from \Cref{sodprop5} and the characterization of $f^*$ in \Cref{genlem} as well as \Cref{sphrem}.
\end{proof}

\begin{proof}[Proof of \Cref{algmod1}]
We have constructed the equivalence \eqref{model1} in \Cref{genlem}.

Before we continue with the construction of the equivalences \eqref{model3},\eqref{model4} we collect the following sign rules used implicitly in the following. Let $C$ be any dg-category, $x,y\in \on{dgMod}(C)$ two dg-modules and consider a morphism $\alpha:x\rightarrow y$. The cone $\on{cone}(\alpha)=x[1]\oplus y$ has the differential $d_n(x,y)=(-d(x),d(y)-\alpha(x))$. The morphism complex $\on{Hom}(x,y)$ has differential $d(f)=d_y\circ f - (-1)^{\on{deg}(f)}f\circ d_x$.

We proceed with the construction of the equivalence \eqref{model4}.

Consider the compact generators 
\[ v=p_1{\bf A_3}[-n],~y=p_2{\bf A_3},~z=p_3{\bf A_3}\]
of $\mathcal{D}(\bf A_3)$. Let $w=\on{cone}(e)$ with $e$ being given by $1\in k[-n]\simeq \on{Hom}_{\on{dgMod}(\bf A_3)}(v,y)$. Note that $w,y,z$ are also compact generators of $\mathcal{D}({\bf A_3})$. A direct inspection shows that the full dg-subcategory $\langle w,y,z\rangle\subset\on{dgMod}(\bf A_3)$ spanned by $w,y,z$ is quasi-equivalent to the dg-category $C$ with objects $w,y,z$ and morphisms generated by the morphisms depicted in the following  
\[
\begin{tikzcd}
                                                     & y \arrow[ld, "\epsilon"', bend left] &                                                      \\
w \arrow[ru, "\eta", bend left] \arrow[rr, "\gamma"] &                                      & z \arrow["t_z"', loop, distance=2em, in=35, out=325]
\end{tikzcd}
\]
in the degrees $|\epsilon|=|\gamma|=0,\,|t_z|=|\eta|=n-1,\,|\delta|=n$ and with vanishing differentials, subject to the relation $t_z\circ \gamma=\gamma\circ\epsilon\circ\eta$. We list the images of the generating morphisms in $C$ under the quasi-equivalence $C\rightarrow \langle w,y,z\rangle$.
\begin{itemize}
\item The morphism $\epsilon$ maps to $(\on{id}_y,0)\in \on{Hom}(y,y)\oplus \on{Hom}(v[1],y)\simeq \on{Hom}(w,y)$ (the splitting here and in the following neglects the differentials).
\item The morphism $\eta$ maps to $t_{n-1}\in k[t_{n-1}]\simeq \on{Hom}(y,y)\subset \on{Hom}(w,y)$.
\item The morphism $\gamma$ maps to $1\in k[t_{n-1}]\simeq \on{Hom}(y,z)\subset \on{Hom}(w,z)$.
\item The morphism $t_z$ maps to $t_{n-1}\in k[t_{n-1}]\simeq \on{Hom}(z,z)$.
\end{itemize}

We denote by $x$ the cone of $-\gamma:w\rightarrow z$ in $\on{dgMod}(C)$ and by $\langle x,y,z\rangle$ the full dg-subcategory of $\on{dgMod}(C)$ spanned by $x,y,z$. There exists a quasi-equivalence of dg-categories $C'\rightarrow \langle x,y,z\rangle$, where $C'$ is the dg-category with objects $x,y,z$ and generating morphisms given by
\[
\begin{tikzcd}
                                               & y \arrow[rd, "b", bend left] \arrow[ld, "a^*"] &                              \\
x \arrow[ru, "a", bend left] \arrow[rr, "c^*"] &                                     & z \arrow[ll, "c", bend left]
\end{tikzcd}
\]
in the degrees $|a|=n-2$, $|b|=|c|=0$, $|a^*|=1$ and $|c^*|=n-1$ subject to the relation $ac=0$. The differentials are determined on the generators by $d(a^*)=cb$ and $d(c^*)=ba$.
We collect the images under the dg-functor $C'\rightarrow \langle x,y,z\rangle$ of the generating morphisms in $C'$. 
\begin{itemize}
\item The morphism $b$ maps to $\gamma\circ \epsilon$.
\item The morphism $a$ maps to $(\eta,0)\in \on{Hom}_{C}(w[1],y)\oplus \on{Hom}_{C}(z,y)\simeq \on{Hom}_{\langle x,y,z\rangle}(x,y)$ (the splitting again neglects the differentials).
\item The morphism $c$ maps to $(0,\on{id}_z)\in \on{Hom}_{C}(z,w[1])\oplus \on{Hom}_C(z,z)\simeq \on{Hom}_{\langle x,y,z\rangle}(z,x)$.
\item The morphisms $a^*$ maps to $(\epsilon,0)\in \on{Hom}_{C}(y,w[1])\oplus \on{Hom}_C(y,z)\simeq \on{Hom}_{\langle x,y,z\rangle}(y,x)$.
\item the morphisms $c^*$ maps to $(0,(-1)^{n}t_z)\in \on{Hom}_C(w[1],z)\oplus \on{Hom}_{C}(z,z)\simeq \on{Hom}_{\langle x,y,z\rangle}(x,z)$.
\end{itemize}

A direct inspection reveals the homology of the mapping complexes in $C'$.
\[ H_\ast\on{Hom}_{C'}(x,x)\simeq H_\ast\on{Hom}_{C'}(y,y)\simeq H_\ast\on{Hom}_{C'}(z,z)\simeq k[t_{n-1}]\]
\[ H_\ast\on{Hom}_{C'}(x,y)\simeq k[t_{n-1}][n-1]\]
\[ H_\ast\on{Hom}_{C'}(y,z)\simeq H_\ast\on{Hom}_{C'}(z,x)\simeq k[t_{n-1}]\]
\[ H_\ast\on{Hom}_{C'}(x,z)\simeq H_\ast\on{Hom}_{C'}(y,x)\simeq H_\ast\on{Hom}_{C'}(z,y)\simeq 0\]
We define a dg-functor $\mu:D_3\rightarrow C'$ by mapping $x,y,z$ to $x,y,z$, $a,b,c$ to $a,b,c$ and $a^*,b^*,c^*$ to $a^*,0,c^*$. \Cref{ginzhomlem} implies that $\mu$ is a quasi-equivalence. In total we obtain equivalences of $\infty$-categories 
\[ \mathcal{D}(D_3)\simeq \mathcal{D}(C')\simeq \mathcal{D}(\langle x,y,z\rangle) \simeq \mathcal{D}(C)\simeq \mathcal{D}({\bf A_3})\overset{\eqref{ke2}}{\simeq} \mathcal{V}^3_{f^*}\,, \]
which yields the desired equivalence \eqref{model4}.

For the construction of the equivalence \eqref{model3}, we consider the pushout diagram in $\mathcal{P}r^L$
\[
\begin{tikzcd}
\mathcal{N}_{f^*} \arrow[d] \arrow[rd, "\ulcorner", phantom, near end] \arrow[r, "\varsigma_3"] & \mathcal{V}^3_{f^*} \arrow[d, "{\pi_{1,3}}"] \\
0 \arrow[r]                                                                                     & \mathcal{V}^2_{f^*}                         
\end{tikzcd}
\]
where $\pi_{1,3}$ denotes the projection to the first and third component of the semiorthogonal decomposition. The upper part of the diagram is equivalent to the image under $\mathcal{D}(\mhyphen)$ of the following diagram of dg-categories,
\[
\begin{tikzcd}
{k[t_1]} \arrow[d] \arrow[r, "i_3^{(-)^n}"] & D_3 \\
0                                           &    
\end{tikzcd}
\]
whose homotopy colimit is easily seen to be quasi-equivalent to $D_2$. This observation provides us with the equivalence of $\infty$-categories $\mathcal{D}(D_2)\simeq \mathcal{V}^2_{f^*}$.
\end{proof}

\begin{lemma}\label{ginzhomlem}
There exist isomorphisms of graded $k$-modules
\[ H_\ast\on{Hom}_{D_3}(x,x)\simeq H_\ast\on{Hom}_{D_3}(y,y)\simeq H_\ast\on{Hom}_{D_3}(z,z)\simeq k[t_{n-1}]\]
\[ H_\ast\on{Hom}_{D_3}(x,y)[1-n]\simeq H_\ast\on{Hom}_{D_3}(y,z)\simeq H_\ast\on{Hom}_{D_3}(z,x)\simeq k[t_{n-1}]\]
and 
\[ H_\ast\on{Hom}_{D_3}(x,z)\simeq H_\ast\on{Hom}_{D_3}(y,x)\simeq H_\ast\on{Hom}_{D_3}(z,y)\simeq 0\,.\]
\end{lemma}

\begin{proof}
The morphisms $a,b,c$ and $a^*,b^*,c^*$ freely generate $D_3$. Any morphism in $D_3$ is thus given by a $k$-linear sum of composites of these generating morphisms. Given a morphism in $D_3$, we call the maximal number of generating morphisms appearing in any one of its summands the length of the morphism. We show by induction over the length of a morphism $u$ in $D_3$ the following statements: 
\begin{enumerate}[i)]
\item Any cycle $u:x\rightarrow x$ (in the morphism complex of $D_3$) is homologous to $\lambda(a^*a-cc^*)^i$ with $\lambda\in k$ and $i\geq 0$.
\item Any cycle $u:x\rightarrow y$ is homologous to $a\circ \lambda (a^*a-cc^*)^i$.
\item Any cycle $u:z\rightarrow y$ is null homologous.
\item Further, the three statements above also hold for $x,y,z$, $a,b,c$ and $a^*,b^*,c^*$ cyclically permuted (and replacing $b^*b-aa^*$ by $b^*b-(-1)^{n-2}aa^*$ because of the grading).
\end{enumerate}

The base case is clear. For the induction step, we note that iv) can be shown in the same way as i),ii) and iii) are shown in the following. We begin with i). Consider a cycle $u:x\rightarrow x$ with summands up to a given length. Using that $D_3$ is freely generated by $a,b,c,a^*,b^*,c^*$, we can decompose $u$ into $u=a^*u_1+cu_2$ for two chains $u_1,u_2$. The condition $d(u)=0$ implies that $d(u_1)=0$. By the induction assumption, we find a chain $v$ satisfying $u_1=d(v)+a\lambda(a^*a-cc^*)^i$. It follows that $u$ is homologous to the cycle 
\[ u'=u+d(a^*v)=a^*a\lambda(a^*a-cc^*)^i+c(u_2+bv)\,.\]
From $d(u')=0$ it follows that $d(u_2+bv)=-ba\lambda(a^*a-cc^*)^i$. Note that $\lambda(a^*a-cc^*)^i$ is not a boundary unless $\lambda=0$, in which case $u$ is nullhomologous. We thus assume that $\lambda\neq 0$ and find that there exists a chain $v'$ with  
\[ u_2+bv= -c^*\lambda(a^*a-cc^*)^i+v'\,.\]
By the induction assumption it follows that $v'$ is a boundary, so that we find that $u$ is homologous to $\lambda (aa^*-cc^*)^{i+1}$, completing the induction step for i).

We continue with the induction step for ii). Consider a cycle $u:x\rightarrow y$. Write $u=a\circ u_1+b^*\circ u_2$ for some morphisms $u_1$ and $u_2$ with $d(u_2)=0$. By the induction assumption we find a chain $v$ with $u_2=d(v)$. We hence find 
\[ u+(-1)^{n}d(b^*v)=a\circ u_1+(-1)^{n}acv\]
with $d(u_1+(-1)^ncv)=0$. By the induction assumption, we find $u_1+(-1)^ncv=d(v')+\lambda(a^*a-cc^*)^i$, so that 
\[ u+(-1)^{n}d(b^*v)-(-1)^nd(av')=\lambda a\circ (a^*a-cc^*)^i\,,\]
finishing the induction step for ii).

For iii), we consider a cycle $u:z\rightarrow y$. We decompose $u$
into $u=a\circ u_1+ b^*\circ u_2$ with $d(u_2)=0$. Using the induction assumption we thus find $v$ with $u_2=\lambda (bb^*-c^*c)^i+d(v)$. It follows that $u$ is homologous to 
\[a\circ (u_1+c\circ v)+\lambda b^*(bb^*-c^*c)^i\,.\] 
The condition $d(u)=0$ implies that $(-1)^{n-1}d(u_1+c\circ v)=\lambda c(bb^*-c^*c)^l$. Since $c(bb^*-c^*c)^l$ is a nonzero homology class unless $\lambda=0$, we find that $\lambda=0$ and $d(u_1+c\circ v)=0$. Applying the induction assumption once more, we find that $u$ is homologous to $ac\circ \lambda'(bb^*-c^*c)^j$, which is the boundary of $b^*\lambda'(bb^*-c^*c)^j$.  We conclude that $u$ is nullhomologous, completing the induction step for iii) and the proof. 
\end{proof}

\begin{proof}[Proof of \Cref{algmod2}]
We begin with the proof of statement 1. The functors $\varsigma_2,\varsigma_3[1]:\mathcal{N}_{f^*}\rightarrow \mathcal{V}^3_{f^*}$ are the inclusions of the third, respectively, second component of the semiorthogonal decomposition of $\mathcal{V}^3_{f^*}$. Using \Cref{sodprop5} and the notation from the proof of \Cref{algmod1}, we obtain that the functor $\varsigma_2$ is modeled by the dg-functor $\sigma_2:k[t_{n-1}]\rightarrow C'$ determined by mapping the unique object $\ast$ of $k[t_{n-1}]$ to $z$ and $t_{n-1}$ to $(-1)^nc^*c$ (the sign arises from the sign of the image of $c^*$ under the dg-functor $C'\rightarrow \langle x,y,z\rangle$). Let further $\sigma_2':k[t_{n-1}]\rightarrow D_3$ be the dg-functor determined by mapping $\ast$ to $z$ and $t_{n-1}$ to $(-1)^{n}(c^*c-bb^*)=(-1)^{n-1}(bb^*-c^*c)$. The commutative diagram of dg-categories
\[ 
\begin{tikzcd}
                        & {k[t_{n-1}]} \arrow[ld, "\sigma_2"'] \arrow[rd, "\sigma_2'"] &                \\
{C'} &                                                            & D_3 \arrow[ll]
\end{tikzcd}
\]
shows that $\varsigma_2$ is also modeled by $\sigma_2'$. An analogous argument shows that $\varsigma_3$ is modeled by the dg-functor $\sigma_3:k[t_{n-1}]\rightarrow D_3$ determined by mapping $\ast$ to $y'$ and $t_{n-1}$ to $aa^*-(-1)^{n-2}b^*b=(-1)^n((-1)^{n}aa^*-b^*b)$. In the case that $n=2$, we can use the sequence of adjunctions 
\[ \varsigma_3\dashv \varrho_2\dashv \varsigma_2\dashv \varrho_1\dashv \varsigma_1 \]
and the rotational symmetry of $D_3$ (which only exists for $n=2$),  to deduce that $\varsigma_1$ is modeled by the dg-functor $k[t_{1}]\rightarrow D_3$ determined by mapping $\ast$ to $x'$ and $t_{1}$ to $(cc^*-a^*a)$.

We now show statement 2. To clarify notation, we denote the right adjoints of the categorified restrictions maps by $\varrho_1,\varrho_2:\mathcal{N}_{f^*}\rightarrow \mathcal{V}^2_{f^*}$ by $\tilde{\varsigma}_1$ and $\tilde{\varsigma}_2$, respectively, instead of $\varsigma_1$ and $\varsigma_2$. We note that there are commutative diagrams in $\mathcal{P}r^L$ for $i=1,2$ as follows
\[
\begin{tikzcd}
\mathcal{N}_{f^*} \arrow[d] \arrow[rd, "\ulcorner", phantom, near end] \arrow[r, "\varsigma_3"] & \mathcal{V}^3_{f^*} \arrow[d, "{\pi_{1,3}}"] & \mathcal{N}_{f^*} \arrow[l, "\varsigma_i"'] \arrow[ld, "\tilde{\varsigma}_i"] \\
0 \arrow[r]                                                                           & \mathcal{V}^2_{f^*}                          &                                                                              
\end{tikzcd}
\]
where $\pi_{1,3}$ denotes the projection to the first and third component of the semiorthogonal decomposition. The functor $\pi_{1,3}$ is modeled by the dg-functor $D_3\rightarrow D_2$ given by mapping $x,z$ to $x,z$ and $c,c^*$ to $c,c^*$ and all other morphisms to zero. The desired models for $\tilde{\varsigma}_1$ and $\tilde{\varsigma}_2$ thus follow from the models for $\varsigma_1$ and $\varsigma_2$. 
\end{proof}

\section{Gluing Ginzburg algebras via perverse schobers}\label{sec6}

\subsection{The main result}\label{sec6.1}

We fix a commutative ring $k$. Given an ideal triangulation $\mathcal{T}$ of an oriented marked surface, we defined the relative Ginzburg algebra $\mathscr{G}_\mathcal{T}$ in \Cref{sec1.1}. Consider also the dual ribbon graph $\Gamma_\mathcal{T}$ of $\mathcal{T}$ of \Cref{dualdef}. In this section we construct a $\Gamma_\mathcal{T}$-parametrized perverse schober $\mathcal{F}_\mathcal{T}$ with the property that the spherical adjunction at each vertex is given by $f^*:\mathcal{D}(k)\leftrightarrow \on{Fun}(S^{2},\mathcal{D}(k)):f_*$ and whose $\infty$-category $\mathcal{H}(\Gamma_\mathcal{T},\mathcal{F}_\mathcal{T})$ of global sections is equivalent to the derived $\infty$-category of the relative Ginzburg algebra $\mathscr{G}_\mathcal{T}$. As we show in the proof of \Cref{thm2}, we can achieve this by making for each vertex of $\Gamma_\mathcal{T}$ a choice of total order of the halfedges at the vertex compatible with given the cyclic order and then defining $\mathcal{F}_\mathcal{T}$ via a gluing of the local model of a parametrized perverse schober from \Cref{locmod}. To match the signs in the differential of the Ginzburg algebra, some care is required in the choices of the gluing diagrams. This construction depends on the choices of total orders. However, we discuss in \Cref{spinsec} an interpretation of all involved choices in terms of spin structures on the surface without the interior marked points $\Sigma={\bf S}\backslash (M\cap {\bf S}^\circ)$ and why the resulting parametrized perverse schober is up to equivalence independent of the choices made.

\begin{theorem}\label{thm2}
Let $\mathcal{T}$ be an ideal triangulation of an oriented marked surface. There exists a $\Gamma_\mathcal{T}$-parametrized perverse schober $\mathcal{F}_\mathcal{T}$ and an equivalence of $\infty$-categories 
\[ \mathcal{H}(\Gamma_\mathcal{T},\mathcal{F}_\mathcal{T})\simeq \mathcal{D}(\mathscr{G}_\mathcal{T})\,.\]
\end{theorem}

From \Cref{thm2} we will deduce the following.

\begin{corollary}\label{thm1}
Let $\mathcal{T}$ be an ideal triangulation of an oriented marked surface and let $\Gamma_\mathcal{T}^\circ$ be the ribbon graph obtained from $\Gamma_\mathcal{T}$ by removing all external edges. Consider the associated quiver $Q_\mathcal{T}^\circ$ with $3$-cyclic potential $W'_\mathcal{T}$  containing a $3$-cycle for each interior ideal triangle of $\mathcal{T}$ from \Cref{sec1.1}. There exists an equivalence of $\infty$-categories
\[ \mathcal{H}_{\Gamma_\mathcal{T}^\circ}(\Gamma_\mathcal{T},\mathcal{F}_\mathcal{T}) \simeq \mathcal{D}(\mathscr{G}(Q_\mathcal{T}^\circ,W'_\mathcal{T}))\,.\]
\end{corollary}

\begin{proof}[Proof of \Cref{thm2}.]
We proof the statement by an induction on the number of ideal triangles of $\mathcal{T}$ and a comparison of the computation of $\mathcal{H}(\Gamma_\mathcal{T},\mathcal{F}_\mathcal{T})$ with an explicit computation of homotopy colimits in $\on{dgCat}_k$ with the quasi-equivalence model structure. 

We define a dg-category $C_\mathcal{T}$ with objects the vertices of $Q_\mathcal{T}$ and morphisms freely generated by the graded arrows of $\tilde{Q}_\mathcal{T}$, see \Cref{sec1.1}, and two further endomorphisms $l_{e,1}',l_{e,2}'$ of degrees $1,2$ at each vertex given by a boundary edge $e$. To clarify notation, we denote in this proof the degree 2 loop $l_e:e\rightarrow e$ for $e\in (Q_\mathcal{T})_0$ by $l_{e,2}$. The differential $d$ of the morphism complexes is determined on the generators. It acts nontrivially on the generators $l_{e,2}'$, $l_{e,2}$ and $a^*_{p,i}$, on the latter two it acts as the differential of $\mathscr{G}_{\mathcal{T}}$ and acts on $l_{e,2}'$ as 
\[l_{e,2}'\mapsto l_{e,1}'-\sum_{a\in (Q_\mathcal{T})_1} p_e[a,a^*]p_e\,.\]
Note that the dg-category $C_\mathcal{T}$ is Morita equivalent to $\mathscr{G}_\mathcal{T}$. For each vertex $e$ corresponding to a boundary edge, we denote by $i_e^{+},i_e^{-}:k[t_1]\rightarrow C_\mathcal{T}$ the dg-functors determined by mapping the unique vertex $\ast$ of $k[t_1]$ to  $e\in C_\mathcal{T}$ and the generator $t_1$ to $l_{e,1}'$ and $-l_{e,1}'$, respectively. 

We deduce the theorem from the following statements, which we prove by an induction on the number of ideal triangles of $\mathcal{T}$.

\textit{
Let $\mathcal{T}$ be an ideal triangulation of an oriented marked surface with dual ribbon graph $\Gamma_\mathcal{T}$. There exists a $\Gamma_\mathcal{T}$-parametrized perverse schober $\mathcal{F}_\mathcal{T}$ with the following properties.
\begin{enumerate}
\item There exists an equivalence of $\infty$-categories 
\[\mathcal{H}(\Gamma_\mathcal{T},\mathcal{F}_\mathcal{T})\simeq \mathcal{D}(C_\mathcal{T})\,.\]
\item For each boundary edge $e$ of $\mathcal{T}$, we denote by $\on{ev}_e:\mathcal{H}(\Gamma_\mathcal{T},\mathcal{F}_\mathcal{T})\rightarrow \mathcal{F}_\mathcal{T}(e)$ the evaluation functor at $e\in \on{Exit}(\mathcal{T})$. The left adjoint of the functor
\begin{equation}\label{evfuneq} \mathcal{D}(C_\mathcal{T})\simeq \mathcal{H}(\Gamma_\mathcal{T},\mathcal{F}_\mathcal{T})\xrightarrow{\on{ev}_e} \mathcal{F}_\mathcal{T}(e)=\on{Fun}(S^2,\mathcal{D}(k))\simeq \mathcal{D}(k[t_1])
\end{equation} 
is equivalent to $\mathcal{D}(i_e^+)$ or $\mathcal{D}(i_e^-)$.
\end{enumerate}  }
In the following we will refer to \eqref{evfuneq} also as the evaluation functor.\medskip

\noindent \textbf{Base case of the induction}\nopagebreak

Let $\mathcal{T}_1$ be the ideal triangulation consisting of a single not self-folded ideal triangle. We choose a total order on the edges of $\mathcal{T}_1$ compatible with the cyclic order induced by the counter-clockwise orientation. This provides us with an equivalence of posets $\on{Exit}(\Gamma_{\mathcal{T}_1})\simeq C_3$. We can thus define $\mathcal{F}_{\mathcal{T}_1}$ as the parametrized perverse schober $\mathcal{G}_3(f^*)$ described in \Cref{locmod} corresponding to the spherical adjunction $f^*\dashv f_*$ (with $n=2$). The existence of the desired equivalence $\mathcal{H}(\mathcal{T}_1,\mathcal{F}_{\mathcal{T}_1})\simeq \mathcal{D}(C_{\mathcal{T}_1})$ is shown in \Cref{algmod1,algmod2}.

Let now  $\mathcal{T}_1$ be the ideal triangulation consisting of a single self-folded triangle. We define $\mathcal{F}_{\mathcal{T}_1}$ using \Cref{notation1} as follows.
\begin{equation}\label{sfs}
\begin{tikzcd}
f^* \arrow["{(\varrho_2,\varrho_3)}"', no head, loop, distance=2em, in=215, out=145] & {} \arrow[l, "\varrho_1", no head, maps to]
\end{tikzcd}
\end{equation}
The limit $\mathcal{H}(\Gamma_{\mathcal{T}_1},\mathcal{F}_{\mathcal{T}_1})$ is equivalent to the colimit of the left dual $\underset{\mathcal{P}r^L}{\on{colim}} \mathbb{D}^L \mathcal{F}_{\mathcal{T}_1}$, which is equivalent to the following coequalizer in $\mathcal{P}r^L$, where $T=T_{\on{Fun}(S^2,\mathcal{D}(k))}[2]$ denotes the suspension of the cotwist functor of the spherical adjunction $f^*\dashv f_*$, also described in \Cref{ctwprop}.
\begin{equation}\label{coeq0}
\begin{tikzcd}
{\mathcal{N}_{f^*}} \arrow[r, shift right=0.5ex, "\varsigma_3"'] \arrow[r, shift left=0.5ex, "\varsigma_1\circ T"] & \mathcal{V}^3_{f^*}
\end{tikzcd} 
\end{equation}
The colimit of \eqref{coeq0} is in turn equivalent to the pushout of the following span in $\mathcal{P}r^L$.
\begin{equation}\label{slfdiag1}
\begin{tikzcd}[column sep=large]
{\mathcal{N}_{f^*}^{\times 2}} \arrow[d, "\on{id}\times \on{id}"] \arrow[r, "{(\varsigma_3,\varsigma_1\circ T)}"] & \mathcal{V}^3_{f^*} \\
{\mathcal{N}_{f^*}}                                                                   &                    
\end{tikzcd}
\end{equation}
Let $\mathcal{T}'$ be the ideal triangulation consisting of a single not self-folded ideal triangle and denote by $e_1,e_3$ two boundary edges of $\mathcal{T}'$. Diagram \eqref{slfdiag1} is equivalent to the image under $\mathcal{D}(\mhyphen)$ of the following diagram in $\on{dgCat}_k$.
\begin{equation}\label{pudiag}
\begin{tikzcd}[column sep=large]
{k[t_1]\amalg k[t_1]} \arrow[d, "\on{id}\amalg \on{id}"] \arrow[r, "{i_{e_3}^+\amalg i_{e_1}^-}"] & C_{\mathcal{T}'} \\
{k[t_1]}                                                                                  &                  
\end{tikzcd}
\end{equation}
The morphism $i_{e_3}^+\amalg i_{e_1}^-$ in \eqref{pudiag} is a cofibration, we can thus compute the homotopy colimit of \eqref{pudiag} as the (ordinary) colimit in the $1$-category of $k$-linear dg-categories. This colimit is quasi-equivalent to $C_{\mathcal{T}_1}$, showing the desired equivalence $\mathcal{H}(\Gamma_{\mathcal{T}_1},\mathcal{F}_{\mathcal{T}_1})\simeq \mathcal{D}(C_{\mathcal{T}_1})$. The left adjoint of the evaluation functor $\mathcal{D}(C_{\mathcal{T}_1})\rightarrow \mathcal{D}(k[t_1])$ to $e$ factors as
\[ \mathcal{D}(k[t_1])\xlongrightarrow{\mathcal{D}(i_{e_2}^-)}\mathcal{D}(C_{\mathcal{T}'})\longrightarrow \mathcal{D}(C_{\mathcal{T}_1})\,,\]
where $e_2\neq e_1,e_3$ is the remaining boundary edge of $\mathcal{T}'$, and is thus equivalent to $\mathcal{D}(i_e^-)$. This completes the base case of the induction where the ideal triangulation consists of a single ideal triangle.\medskip

\noindent \textbf{Induction step}

Assume now that the statement has been shown for all ideal triangulations with at most $n$ ideal triangles. Let $\mathcal{T}_{n+1}$ be an ideal triangulation with $n+1$ ideal triangles. We choose any subtriangulation $\mathcal{T}_n$ consisting of $n$ connected ideal triangles of $\mathcal{T}_{n+1}$. The complement of $\mathcal{T}_n$ in $\mathcal{T}_{n+1}$ consists of a single ideal triangle and is denoted by $\mathcal{T}_1$. Denote by $S$ the set of edges along which $\mathcal{T}_1$ and $\mathcal{T}_n$ are glued in $\mathcal{T}_{n+1}$ and $s=|S|$. To define the $\Gamma_{\mathcal{T}_{n+1}}$-parametrized perverse schober $\mathcal{F}_{\mathcal{T}_{n+1}}$ we wish to glue $\mathcal{F}_{\mathcal{T}_{1}}$ and $\mathcal{F}_{\mathcal{T}_{n}}$ in the sense of \Cref{sgluelem}. We however first modify the diagram $\mathcal{F}_{\mathcal{T}_{n}}:\on{Exit}(\Gamma_{\mathcal{T}_n})\rightarrow \on{St}$ by postcomposing the functor $\mathcal{F}_{\mathcal{T}_{n}}(v\shortrightarrow e)$, for each edge $v\rightarrow e$ in $\on{Exit}(\Gamma_{\mathcal{T}_n})$ with $e\in S$, with an autoequivalence $T_e$ of $\on{Fun}(S^2,\mathcal{D}(k))$ described as follows. Consider the left adjoints of 
\[\mathcal{D}(\mathscr{G}_{\mathcal{T}_1})\simeq \mathcal{H}(\Gamma_{\mathcal{T}_1},\mathcal{F}_{\mathcal{T}_1})\xrightarrow{\on{ev}_e} \on{Fun}(S^2,\mathcal{D}(k))\simeq \mathcal{D}(k[t_1])\] 
and
\[ \mathcal{D}(\mathscr{G}_{\mathcal{T}_{n}})\simeq \mathcal{H}(\Gamma_{\mathcal{T}_n},\mathcal{F}_{\mathcal{T}_n}) \xrightarrow{\on{ev}_e} \on{Fun}(S^2,\mathcal{D}(k))\simeq \mathcal{D}(k[t_1])\,.\] 
The induction assumption implies that these two left adjoint are each equivalent to either $\mathcal{D}(i^+_e)$ or $\mathcal{D}(i^-_e)$ (abusing notation). If both superscripts, each $+$ or $-$, in this description are identical, we choose $T_e=T$, where $T$ is as in the base case of the induction, otherwise we choose $T_e=\on{id}_{\on{Fun}(S^2,\mathcal{D}(k))}$. We denote the $\Gamma_{\mathcal{T}_n}$-parametrized perverse schober obtained from the modification of $\mathcal{F}_{\mathcal{T}_{n}}$ by $\mathcal{F}^\ast_{\mathcal{T}_n}$. We define $\mathcal{F}_{\mathcal{T}_{n+1}}$ as the gluing of $\mathcal{F}_{\mathcal{T}_1}$ and $\mathcal{F}^\ast_{\mathcal{T}_n}$. 

We observe that $\mathcal{H}\big(\Gamma_{\mathcal{T}_n}, \mathcal{F}^\ast_{\mathcal{T}_n}\big) \simeq \mathcal{H}\big(\Gamma_{\mathcal{T}_{n}},\mathcal{F}_{\mathcal{T}_n}\big)$. Using the induction assumption we thus find equivalences of $\infty$-categories
\[ \mathcal{H}\big(\Gamma_{\mathcal{T}_n},\mathcal{F}_{\mathcal{T}_{n+1}}|_{\on{Exit}(\Gamma_{\mathcal{T}_{n}})}\big)=\mathcal{H}\big(\Gamma_{\mathcal{T}_n},\mathcal{F}^\ast_{\mathcal{T}_n}\big)\simeq \mathcal{D}(C_{\mathcal{T}_n})\,,\]
\[ \mathcal{H}\big(\Gamma_{\mathcal{T}_1},\mathcal{F}_{\mathcal{T}_{n+1}}|_{\on{Exit}(\Gamma_{\mathcal{T}_{1}})}\big)=\mathcal{H}\big(\Gamma_{\mathcal{T}_1},\mathcal{F}_{\mathcal{T}_1}\big)\simeq \mathcal{D}(C_{\mathcal{T}_1})\,.\]
A standard result on the decomposition results of colimits, see \cite[4.2.3.10]{HTT}, shows that global sections commute with gluing in the sense that there is a pullback diagram in $\mathcal{P}r^L$ as follows.
\[
\begin{tikzcd}
\mathcal{H}\big(\Gamma_{\mathcal{T}_{n+1}},\mathcal{F}_{\mathcal{T}_{n+1}}\big) \arrow[r] \arrow[d] \arrow[rd, "\lrcorner", phantom, near start] & \mathcal{H}\big(\Gamma_{\mathcal{T}_n},\mathcal{F}^\ast_{\mathcal{T}_n}\big) \arrow[d] \\
\mathcal{H}\big(\Gamma_{\mathcal{T}_1},\mathcal{F}_{\mathcal{T}_1}\big) \arrow[r]                                              & {\on{Fun}(S^2,\mathcal{D}(k))^{\times s}}
\end{tikzcd}
\]
The left adjoint diagram is equivalent to the image under $\mathcal{D}(\mhyphen)$ of the following homotopy pushout diagram in $\on{dgCat}_k$.
\begin{equation}\label{rpu}
\begin{tikzcd}
k[t_1]^{\amalg s} \arrow[r, "\alpha_n"] \arrow[d, "\alpha_1"] \arrow[rd, "\ulcorner", phantom, near end] & C_{\mathcal{T}_n} \arrow[d, "i_C"] \\
C_{\mathcal{T}_1} \arrow[r]                                     & C
\end{tikzcd}
\end{equation}
The dg-functors $\alpha_1$ and $\alpha_n$ restrict on the component of $k[t_1]^{\amalg s}$ indexed by a given edge $e\in S$ each to a dg-functor of the form $i_e^+$ or $i_e^-$ (again by abuse of notation); we denote these functors by $i_e^1$ and $i_e^{n}$, respectively. We show below that $\alpha_1$ and $\alpha_n$ are cofibrations. The diagram \eqref{rpu} is thus also pushout and we find $C$ to be given as follows. The number of objects of $C$ is $|(Q_{\mathcal{T}_n})_0|+ |(Q_{\mathcal{T}_1})_0|-s$. The morphism are freely generated by the edges of $Q_{\mathcal{T}_n}$ and the edges of $Q_{\mathcal{T}_{1}}$ and for each edge $e\in S$ two endomorphisms $(l_{e,2}')^1,(l_{e,2}')^n:e\rightarrow e$ in degree $2$ and one endomorphism $l_{e,1}':e\rightarrow e$ in degree $1$ satisfying 
\[ d((l_{e,2}')^1)=i_{e}^1(t_1)-\sum_{a\in (Q_{\mathcal{T}_1})_1} p_e[a,a^*]p_e\] 
and 
\[ d((l_{e,2}')^n)=i_{e}^n(t_1)-\sum_{a\in (Q_{\mathcal{T}_n})_1} p_e[a,a^*]p_e\,.\] 
The above choice of $T_e$ ensures that $i_{e}^1(t_1)=-i_{e}^n(t_1)=\pm l_{e,1}'$, it follows that 
\[ d(-(l_{e,2}')^1-(l_{e,2}')^n)=\sum_{a\in (Q_{\mathcal{T}_{n+1}})_1} p_e[a,a^*]p_e\] 
is a boundary. We deduce the existence of a quasi-equivalence $C_{\mathcal{T}_{n+1}}\rightarrow C$ mapping $l_{e,2}$ to $-(l_{e,2}')^1-(l_{e,2}')^n$. This shows the first part of the induction step.

Let $e$ be a boundary edge of ${\mathcal{T}_{n+1}}$. Assume that $e$ lies in $\mathcal{T}_{n}$. The left adjoint of the evaluation functor to $e$ factors as 
\[ \mathcal{D}(k[t_1])\xrightarrow{\mathcal{D}(i_{e}^{\pm})} \mathcal{D}(C_{\mathcal{T}_{n}})\xrightarrow{\mathcal{D}(i_C)} \mathcal{D}(C)\simeq \mathcal{D}(C_{\mathcal{T}_{n+1}})\]
and thus equivalent to $\mathcal{D}(i_{e}^\pm)$. If $e$ lies in $\mathcal{T}_1$, we can argue analogously. This shows the second part of the induction step, completing the induction.\medskip

\noindent {\bf $\alpha_1$ and $\alpha_n$ are cofibrations.}\\
To see that $\alpha_1$ is a cofibration, one can directly check the lifting property with respect to the fibrations in the quasi-equivalence model structure. To see that $\alpha_n$ is a cofibration, we argue by induction on $n$ and decompose $\mathcal{T}_n$ into two ideal triangulations $\mathcal{T}_i,\mathcal{T}_j$ with $i,j<n$ such that $\mathcal{T}_n$ is obtained by gluing $\mathcal{T}_i$ and $\mathcal{T}_j$ along $q$ edges. One then finds the following commutative diagram in $\on{dgCat}_k$
\begin{equation}\label{alphaneq}
\begin{tikzcd}[row sep=small, column sep=small]
0 \arrow[rr] \arrow[dd] \arrow[rd]   &                                    & {k[t_1]^{s_2}} \arrow[dd] \arrow[rd, "\alpha_i"] &                              \\
                                     & {k[t_1]^{q}} \arrow[dd] \arrow[rr] &                                      & C_{\mathcal{T}_i} \arrow[dd] \\
{k[t_1]^{s_1}} \arrow[rr] \arrow[rd, "\alpha_j"] &                                    & {k[t_1]^{s}} \arrow[rd, "\alpha_n"]  &                              \\
                                     & C_{\mathcal{T}_j} \arrow[rr]       &                                      & C_{\mathcal{T}_n}           
\end{tikzcd}
\end{equation}
whose front and back squares are pushout and where $s_1,s_2$ are the number of edges in $S$ lying in $\mathcal{T}_j$ and $\mathcal{T}_i$, respectively. Since we know all morphisms except $\alpha_n$ in \eqref{alphaneq} to be cofibrations, it follows that $\alpha_n$ is also a cofibration. This completes the argument and thus the proof.
\end{proof}

\begin{proof}[Proof of \Cref{thm1}]
The $\infty$-category $\mathcal{H}_{\Gamma^\circ_{\mathcal{T}}}(\Gamma,\mathcal{F}_\mathcal{T})$ is equivalent to the pullback of the diagram of $\infty$-categories
\[
\begin{tikzcd}[column sep=large]
                                                                                                         & 0 \arrow[d]                                              \\
{\mathcal{H}(\Gamma,\mathcal{F}_\mathcal{T})} \arrow[r, "\on{ev}^{\times |\partial\Gamma_\mathcal{T}|}"] & \mathcal{N}_{f^*}^{\times |\partial \Gamma_\mathcal{T}|}
\end{tikzcd}
\]
where $|\partial \Gamma_{\mathcal{T}}|$ denotes the number of external edges of $\Gamma_\mathcal{T}$ (or equivalently the number of boundary edges of $\mathcal{T}$) and $\on{ev}^{\times |\partial\Gamma_\mathcal{T}|}$ is the product of the evaluation functors at the external edges. The left adjoint of the above diagram lies in $\mathcal{P}r^L$ and is modeled by the diagram of cofibrant dg-categories 
\begin{equation}\label{ginzdiag1}
\begin{tikzcd}
{\underset{{|\partial \Gamma_\mathcal{T}|}}{\coprod} k[t_1]} \arrow[d, "\alpha"] \arrow[r] & 0 \\
C_{\mathcal{T}}                                                                       &  
\end{tikzcd}
\end{equation}
where $\alpha$ restricts on each component corresponding to an external edge $e$ to the dg-functor $i^+_e$ or $i^-_e$. The pushout of \eqref{ginzdiag1} is easily seen to be Morita equivalent to the Ginzburg algebra $\mathscr{G}(Q_\mathcal{T}^\circ,W_\mathcal{T}')$. We thus find the desired equivalence of $\infty$-categories $\mathcal{H}_{\Gamma^\circ_{\mathcal{T}}}(\Gamma,\mathcal{F}_\mathcal{T})\simeq \mathcal{D}(\mathscr{G}(Q_\mathcal{T}^\circ,W_\mathcal{T}'))$.
\end{proof}

We end this section by describing how the perverse schober $\mathcal{F}_\mathcal{T}$ can be modified so that its global sections describe only the perfect or only the finite modules over the relative Ginzburg algebra $\mathscr{G}_\mathcal{T}$. 

\begin{notation}
We denote
\begin{itemize}
\item by $\mathcal{D}(k)^{\on{perf}}\subset \mathcal{D}(k)$ the full subcategory spanned by perfect complexes of $k$-vector spaces.
\item by $\on{Fun}(S^2,\mathcal{D}(k))^{\on{perf}}\subset \on{Fun}(S^2,\mathcal{D}(k))$ the full subcategory spanned by compact objects. Note that $\on{Fun}(S^2,\mathcal{D}(k)^{\on{perf}})$ is contained in the full subcategory of $\on{Fun}(S^2,\mathcal{D}(k))$ spanned by compact objects, denoted $\on{Fun}(S^2,\mathcal{D}(k))^{\on{perf}}$.
\item by 
\[ (f^*)^{\on{perf}}:\mathcal{D}(k)^{\on{perf}}\longleftrightarrow \on{Fun}(S^2,\mathcal{D}(k))^{\on{perf}}:f_*^{\on{perf}}\]
and 
\[ (f^*)^{\on{fin}}:\mathcal{D}(k)^{\on{perf}}\longleftrightarrow \on{Fun}(S^2,\mathcal{D}(k)^{\on{perf}}):f_*^{\on{fin}}\]
the restrictions of the adjunction $f^*:\mathcal{D}(k)\leftrightarrow \on{Fun}(S^2,\mathcal{D}(k))$. For the well-definedness of $f_*^{\on{perf}}$, note that $f^*$ preserves filtered colimits, so that $f_*$ carries compact objects to compact objects.
\end{itemize}
\end{notation}

Adapting the construction in the proof of \Cref{thm2}, we define the $\Gamma_\mathcal{T}$-parametrized perverse schobers
\[\mathcal{F}_\mathcal{T}^{\on{fin}},\mathcal{F}_\mathcal{T}^{\on{perf}}:\on{Exit}(\Gamma_\mathcal{T})\rightarrow \on{St}\]
by replacing in $\mathcal{F}_{\mathcal{T}}$ at each vertex of $\Gamma_\mathcal{T}$ the spherical functor $f^*$ by $(f^*)^{\on{fin}}$ and $(f^*)^{\on{perf}}$, respectively. Note that $\mathcal{F}_\mathcal{T}^{\on{perf}},\mathcal{F}_\mathcal{T}^{\on{fin}}$ take values in the $\infty$-category $\on{St}^{\on{idem}}\subset \on{St}$ of idempotent complete, stable $\infty$-categories.

\begin{proposition}\label{finprop}
Consider the morphisms $\mathcal{F}_\mathcal{T}^{\on{fin}},\mathcal{F}_\mathcal{T}^{\on{perf}}\rightarrow \mathcal{F}_\mathcal{T}$ in $\mathfrak{P}(\Gamma_\mathcal{T})$, pointwise given by the fully faithful functors. Passing to global sections and colimits, respectively, yields the following commutative diagram of $\infty$-categories,
\[
\begin{tikzcd}
\underset{\on{St}^{\on{idem}}}{\on{colim}}\,\mathbb{D}^L\mathcal{F}_\mathcal{T}^{\on{perf}} \arrow[d, "\simeq"] \arrow[r] & \underset{\mathcal{P}r^L}{\on{colim}}\, \mathbb{D}^L \mathcal{F}_\mathcal{T} \arrow[d, "\simeq"] \arrow[r, "\simeq", no head] & \mathcal{H}(\Gamma_\mathcal{T},\mathcal{F}_\mathcal{T}) \arrow[d, "\simeq"] & \mathcal{H}(\Gamma_\mathcal{T},\mathcal{F}_\mathcal{T}^{\on{fin}}) \arrow[l] \arrow[d, "\simeq"] \\
{\mathcal{D}(\mathscr{G}_\mathcal{T})^{\on{perf}}} \arrow[r, hook]                                                         & {\mathcal{D}(\mathscr{G}_\mathcal{T})} \arrow[r, "=", no head]                     & {\mathcal{D}(\mathscr{G}_\mathcal{T})}     & {\mathcal{D}(\mathscr{G}_\mathcal{T})^{\on{fin}}} \arrow[l, hook]      
\end{tikzcd}
\]
where
\begin{itemize}
\item $\mathcal{D}(\mathscr{G}_\mathcal{T})^{\on{perf}}\subset \mathcal{D}(\mathscr{G}_\mathcal{T})$ denotes the full subcategory spanned by the compact objects and 
\item $\mathcal{D}(\mathscr{G}_\mathcal{T})^{\on{fin}}\subset \mathcal{D}(\mathscr{G}_\mathcal{T})$ denotes the full subcategory spanned by the modules with finite total homological dimension.
\end{itemize}
In particular, we obtain that the finite $\mathscr{G}_\mathcal{T}$-modules can be characterized as the global sections whose pointwise values on the edges of $\mathcal{T}$ in $\on{Exit}(\Gamma_\mathcal{T})$ lie in $\mathcal{N}_{(f^*)^{\on{fin}}}$.
\end{proposition}

\begin{proof}
Note that the morphism $\mathcal{F}_\mathcal{T}^{\on{perf}}\rightarrow \mathcal{F}_\mathcal{T}$ is pointwise given by $\on{Ind}$-completion. It thus follows from discussion on the computation of colimits in $\on{St}^{\on{idem}}$ in \Cref{sec2.1} that the colimit of $\mathbb{D}^L(\mathcal{F}_\mathcal{T})^{\on{perf}}$ describes the $\infty$-category of compact $\mathscr{G}_\mathcal{T}$-modules.

A $\mathscr{G}_\mathcal{T}$-module $M$ is by definition finite if and only if $\on{RHom}_{\mathscr{G}_\mathcal{T}}(\mathscr{G}_\mathcal{T},M)\in \mathcal{D}(k)$ is a finite $k$-module. Note that $\mathscr{G}_\mathcal{T}=\bigoplus_{e}p_e\mathscr{G}_\mathcal{T}$, where the sum runs over all edges of $\mathcal{T}$. Using \Cref{glgenprop}, we thus obtain that $M$ is finite if and only if $\on{RHom}_{\mathscr{G}_\mathcal{T}}(p_e\mathscr{G}_\mathcal{T},M)\simeq i^*\on{ev}_e(M)\in \mathcal{D}(k)$ is a finite $k$-module for all $e$. The latter is fulfilled if and only if the pointwise values of the coCartesian section corresponding to $M$ lies in $\mathcal{N}_{(f^*)^{\on{fin}}}$. 
\end{proof}

We illustrate how finite $\mathscr{G}_\mathcal{T}$-modules can be described by locally finite global sections in examples in \Cref{sec6.3}.

\subsection{Two examples}\label{sec6.2}

In this section we illustrate the computation of the proof of \Cref{thm2} in the case of the once-punctured torus and the unpunctured $4$-gon. 

\begin{example}\label{ex1}
Let ${\bf S}$ be the once-punctured torus and consider the ideal triangulation $\mathcal{T}$ of ${\bf S}$ consisting of two ideal triangles glued together at all three edges. Using the graphical \Cref{ribnot}, the dual ribbon graph $\Gamma_\mathcal{T}$  can be depicted as follows (the crossing has to do with the cyclic orderings of the halfedges induced by the orientation of ${\bf S}$).

\begin{center}
\begin{tikzpicture}[scale=1]
\node (a) at (-1,0)  {$\cdot$};
\node (b) at (1,0) {$\cdot$} ;
\draw (a) to[bend right=90] (b);
\draw (-0.8,-0.2) to[controls={(0,-1) and (0,1)}] (0.8,0.2);
\draw (-0.8,0.2) to[controls={(0,1) and (0,-1)}] (0.8,-0.2);
\end{tikzpicture}
\end{center}

In \Cref{sec5.2}, we introduced the dg-category $D_3$ with three objects $x,y,z$, freely generated by the following morphisms.
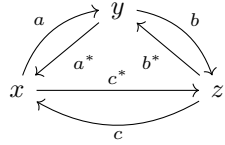
\[ 
\begin{tikzcd}
                                               & y \arrow[ld, "a^*"] \arrow[rd, "b", bend left] &                                                \\
x \arrow[ru, "a", bend left] \arrow[rr, "c^*"] &                                                & z \arrow[lu, "b^*"] \arrow[ll, "c", bend left]
\end{tikzcd}
\] 
The morphisms are in degrees $|a|=|b|=|c|=0$ and $|a^*|=|b^*|=|c^*|=1$. The differential is determined by $d(a^*)=cb,~d(b^*)=ac,~d(c^*)=ba$. The dg-category $D_3$ is Morita equivalent to the relative Ginzburg algebra of a single triangle. To describe the Ginzburg algebra of $\mathcal{T}$, we consider the following span of dg-categories, where the morphisms are defined in \Cref{ifn}.
\begin{equation}\label{exdiag1}
\begin{tikzcd}[row sep=small, column sep=small]
                  & k[t_1]^{\amalg 3} \arrow[rd, "{(i^-_1,i^+_2,i^-_3)}"] \arrow[ld, "{(i^+_1,i^-_2,i^+_3)}"'] &                   \\
D_3 &                                                                                                                         & D_3
\end{tikzcd}
\end{equation}
Informally, the above span describes a gluing of two copies of the dg-category $D_3$
at all vertices with matching orientations. To compute the homotopy colimit of \eqref{exdiag1}, we consider the cofibrant replacement of the diagram \eqref{exdiag1}, which consists in replacing $D_3$ with the following dg-category. 
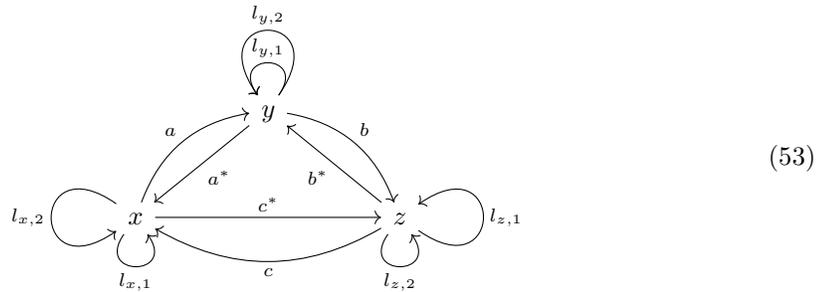
\begin{equation}\label{exdiag2}
\begin{tikzcd}[row sep=large, column sep=large]
                                                                                                                                                           & y \arrow[ld, "a^*"] \arrow[rd, "b", bend left] \arrow["l_{y,1}"', loop, distance=2em, in=125, out=55] \arrow["l_{y,2}"', loop, distance=4em, in=125, out=55] &                                                                                                                                                           \\
x \arrow[ru, "a", bend left] \arrow[rr, "c^*"] \arrow["l_{x,1}"', loop, distance=2em, in=305, out=235] \arrow["l_{x,2}"', loop, distance=4em, in=215, out=145] &                                                                                                                                                          & z \arrow[lu, "b^*"] \arrow[ll, "c", bend left] \arrow["l_{z,2}"', loop, distance=2em, in=305, out=235] \arrow["l_{z,1}"', loop, distance=4em, in=35, out=325]
\end{tikzcd}
\end{equation}
Here $|l_{x,1}|=|l_{y,1}|=|l_{z,1}|=1$ and $|l_{x,2}|=|l_{y,2}|=|l_{z,2}|=2$ and the differential acts on the additional generators as follows.
\begin{align*}
d(l_{x,1})=d(l_{y,1})=d(l_{z,1})=& 0\\
d(l_{x,2})=& l_{x,1}-(cc^*-a^*a)\\
d(l_{y,2})=& l_{y,1}-(aa^*-b^*b)\\
d(l_{z,2})=& l_{z,1}-(bb^*-c^*c)
\end{align*}
The colimit of the cofibrant replacement of \eqref{exdiag1} is up to quasi-equivalence the freely generated dg-category of the following form.
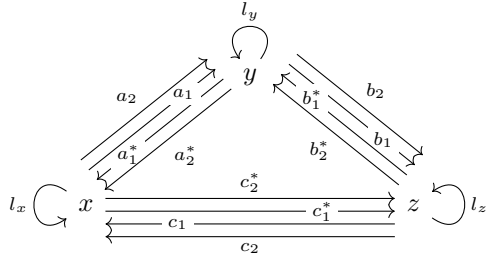
\[
\begin{tikzcd}[row sep=huge, column sep=huge]
                                                                                                                                                                                                                 & y \arrow[ld, "a_1^*" description, shift right=2, near end] \arrow[rd, "b_1" description, shift left=5, near end] \arrow[ld, "a_2^*"] \arrow[rd, "b_2", shift left=7] \arrow["l_y"', loop, distance=2em, in=125, out=55] &                                                                                                                                                                                                                    \\
x \arrow[ru, "a_1" description, shift left=4, near end] \arrow[rr, "c^*_1" description, shift right, near end] \arrow["l_x"', loop, distance=2em, in=215, out=145] \arrow[ru, "a_2", shift left=6] \arrow[rr, "c_2^*", shift left] &                                                                                                                                                                                                                   & z \arrow[lu, "b^*_1" description, shift right=3, near end] \arrow[ll, "c_1" description, shift left=3, near end] \arrow["l_z"', loop, distance=2em, in=35, out=325] \arrow[lu, "b_2^*", shift right] \arrow[ll, "c_2", shift left=5]
\end{tikzcd}
\]
A direct computation shows that the differentials match the differentials of the Ginzburg algebra $\mathscr{G}(Q_\mathcal{T},W'_\mathcal{T})$ of the quiver
\begin{equation}\label{span3}
\begin{tikzcd}[row sep=large, column sep=large]
                            & \cdot \arrow[rd, shift left=0.5ex, "b_1"] \arrow[rd, shift right=0.5ex, "b_2"'] &                             \\
\cdot \arrow[ru, shift left=0.5ex, "a_1"] \arrow[ru, shift right=0.5ex, "a_2"'] &                             & \cdot \arrow[ll, shift left=0.5ex, "c_1"] \arrow[ll, shift right=0.5ex, "c_2"']
\end{tikzcd}
\end{equation}
with potential $W'_\mathcal{T}=c_1b_1a_1+c_2b_2a_2$. Note that the once-punctured torus has no boundary, so that the relative Ginzburg algebra $\mathscr{G}_\mathcal{T}$ is identical to $\mathscr{G}(Q_\mathcal{T},W'_\mathcal{T})$.

The above gluing description of the Ginzburg algebra is caputured by the perverse schober $\mathcal{F}_\mathcal{T}$. Denote by $T$ the functor $\on{Fun}(S^2,\mathcal{D}(k))\overset{\eqref{loceq}}{\simeq} \mathcal{D}(k[t_1])\xrightarrow{\varphi^*} \mathcal{D}(k[t_1])\overset{\eqref{loceq}}{\simeq} \on{Fun}(S^2,\mathcal{D}(k))$, where $\varphi^*$ is the pullback functor along the morphism of dg-algebras $\varphi:k[t_1]\xrightarrow{t_1\mapsto -t_1} k[t_1]$ or equivalently the suspended cotwist $T_{\on{Fun}(S^{n-1},\mathcal{D}(k))}[2]$ of the adjunction $f^*\dashv f_*$, see \Cref{ctwprop}. The $\Gamma_\mathcal{T}$-parametrized perverse schober $\mathcal{F}_\mathcal{T}:\on{Exit}(\Gamma_\mathcal{T})\rightarrow \on{St}$ is given by the following diagram, see \Cref{sec3.1} for the notation.
\begin{equation}\label{exdiag0}
\begin{tikzcd}[row sep=small, column sep=small]
\mathcal{V}^3_{f^*} \arrow[rd, "T\circ \varrho_2"] \arrow[rr, "T\circ \varrho_3"] \arrow[dd, "T\circ \varrho_1"] &                   & \mathcal{N}_{f^*}                                                                           \\
                                                                                            & \mathcal{N}_{f^*} &                                                                                             \\
\mathcal{N}_{f^*}                                                                           &                   & \mathcal{V}^3_{f^*} \arrow[ll, "\varrho_1"] \arrow[lu, "\varrho_2"] \arrow[uu, "\varrho_3"]
\end{tikzcd}
\end{equation}
The limit of this diagram is equivalent to the colimit in $\mathcal{P}r^L$ of the right adjoint diagram (or the left adjoint, the adjoint diagrams are equivalent). A standard argument, for example using the decomposition of colimits \cite[4.2.3.10]{HTT}, shows that the colimit of the left adjoint diagram is equivalent to the colimit of the following span in $\mathcal{P}r^L$.
\begin{equation}\label{exdiag-1}
\begin{tikzcd}[row sep=small, column sep=small]
                    & \mathcal{N}_{f^*}^{\times 3} \arrow[rd, "{(\varsigma_1,\varsigma_2,\varsigma_3)}"] \arrow[ld, "{(\varsigma_1\circ T,\varsigma_2\circ T,\varsigma_3\circ T)}"'] &                     \\
\mathcal{V}^3_{f^*} &                                                                                                                                           & \mathcal{V}^3_{f^*}
\end{tikzcd}
\end{equation}
The above span is modeled in terms of dg-categories by the span \eqref{exdiag1}, so that the colimit of \eqref{exdiag-1} in $\mathcal{P}r^L$ is equivalent to the image under $\mathcal{D}(\mhyphen)$ of the homotopy colimit of \eqref{exdiag1} and thus equivalent to the derived category of the Ginzburg algebra $\mathscr{G}_\mathcal{T}$. Finally, we wish to emphasize that the appearance of the autoequivalence $T$ in the diagram \eqref{exdiag0} serves to fix the correct signs in the differential of the Ginzburg algebra. 
\end{example}

In next example of the unpunctured $4$-gon the surface has a non-empty boundary, so that the associated relative Ginzburg algebra $\mathscr{G}_\mathcal{T}$ contains more information than the non-relative Ginzburg algebra $\mathscr{G}(Q_\mathcal{T}^\circ,W_\mathcal{T}')$.

\begin{example}\label{ex2}
Consider the unpunctured square  ${\bf S}$ and the ideal triangulation $\mathcal{T}$ depicted as follows. 

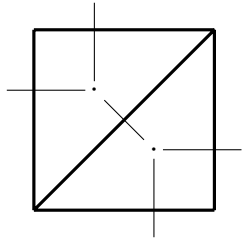
\begin{center}
\begin{tikzpicture}[scale=1.2]
  \draw[very thick]
    (0, 2) -- (0, 0)
    (2, 0) -- (0, 0)
    (2, 2) -- (2, 0)
    (0, 2) -- (2, 2)
    (0, 0) -- (2, 2);
  \node (2) at (0.67,1.33){\large $\cdot$};
  \node (3) at (1.33,0.67){\large $\cdot$};  
  \draw
   (1.22,0.78) -- (0.78, 1.22)
   (1.33,0.57) -- (1.33,-0.3)
   (1.43,0.67) -- (2.3,0.67)
   (0.57, 1.33) -- (-0.3, 1.33)
   (0.67, 1.43) -- (0.67, 2.3);
\end{tikzpicture}
\end{center}

The dual ribbon graph $\Gamma_\mathcal{T}$ is included in the above depiction. The left dual of the $\Gamma_\mathcal{T}$-parametrized perverse schober $\mathcal{F}_\mathcal{T}:\on{Exit}(\Gamma_\mathcal{T})\rightarrow \on{St}$ is modeled by the following diagram of dg-categories.
\begin{equation}\label{exdiag3}
\begin{tikzcd}
                            & {k[t_1]} \arrow[d, "i_x^+"] &                                                &                             &                             \\
{k[t_1]} \arrow[r, "i_z^-"] & D_3                         & {k[t_1]} \arrow[l, "i_y^+"] \arrow[r, "i_y^-"] & D_3                         & {k[t_1]} \arrow[l, "i_z^-"] \\
                            &                             &                                                & {k[t_1]} \arrow[u, "i_x^+"] &                            
\end{tikzcd}
\end{equation}
To compute the homotopy colimit of the above diagram, we again consider the cofibrant replacement, which consists in replacing $D_3$ with the dg-category depicted in \eqref{exdiag2}. The colimit of the cofibrant replacement of \eqref{exdiag3} is quasi-equivalent to the freely generated dg-category with five objects, which can be depicted as follows.
\begin{equation}\label{exdiag4}
\begin{tikzcd}[row sep=huge, column sep=huge]
z_2 \arrow[r, "c_2", bend left] \arrow[d, "b_2^*"]  & x \arrow["l_x"', loop, distance=2em, in=125, out=55] \arrow[rd, "c_1^*"] \arrow[ld, "a_2", bend left] \arrow[l, "c_2^*"] \arrow[r, "a_1", bend left] & y_1 \arrow[d, "b_1", bend left] \arrow[l, "a_1^*"]  \\
y_2 \arrow[ru, "a_2^*"] \arrow[u, "b_2", bend left] &                                                                                                                                                        & z_1 \arrow[lu, "c_1", bend left] \arrow[u, "b_1^*"]
\end{tikzcd}
\end{equation}
The nonzero differentials are given by 
\[ d(l_x)=c_1c_1^*+c_2c_2^*-a_1^*a_1-a_2^*a_2\]
and for $i=1,2$
\[ d(a_i^*)=c_ib_i,~d(b_i^*)=a_ic_i,~d(c_i^*)=b_ia_i\,.\]
It follows that the dg-category \eqref{exdiag4} is Morita equivalent to the relative Ginzburg algebra $\mathscr{G}_\mathcal{T}$. 

The $\infty$-category of global sections of $\mathcal{F}_\mathcal{T}$ with support on the sub-ribbon graph $\Gamma_\mathcal{T}^\circ$ of internal edges can be described as the colimit of the diagram in $\mathcal{P}r^L$,
\begin{equation}\label{exdiag5}
\begin{tikzcd}
                    & {\mathcal{N}_{f^*}} \arrow[rd, "f_!"] \arrow[ld, "f_!"'] &                     \\
\mathcal{V}^1_{f^*} &                                                            & \mathcal{V}^1_{f^*}
\end{tikzcd}
\end{equation}
where $f_!:\mathcal{N}_{f^*}=\on{Fun}(S^2,\mathcal{D}(k))\rightarrow \mathcal{D}(k)=\mathcal{V}^1_{f^*}$ is the left adjoint of the pullback functor $f^*$ along $f:S^2\rightarrow \ast$. The diagram \eqref{exdiag5} is modeled by the following diagram in $\on{dgCat}_k$ 
\begin{equation}\label{exdiag6}
\begin{tikzcd}
  & {k[t_1]} \arrow[rd, "\phi"] \arrow[ld, "\phi"'] &   \\
k &                                                 & k
\end{tikzcd}
\end{equation}
where $\phi:k[t_1]\xrightarrow{t_1\mapsto 0} k$. The homotopy colimit of \eqref{exdiag6} is given by the polynomial algebra $k[t_2]$ with $|t_2|=2$, which is the Ginzburg algebra of the $A_1$-quiver. 

The gluing construction of the relative Ginzburg algebra can be adapted to describe global sections of $\mathcal{F}_\mathcal{T}$ which vanish on any fixed subset of external edges of $\Gamma_\mathcal{T}$. The constructions of $\mathscr{G}_\mathcal{T}$ and $\mathscr{G}(Q_\mathcal{T}^\circ,W_{\mathcal{T}}')$ correspond to the case of sections which vanishing on none, respectively, all external edges. For example, the $\infty$-category $\mathcal{C}$ of sections of $\mathcal{F}_\mathcal{T}$ which vanish on all but one external edge is equivalent to the colimit of the following diagram in $\mathcal{P}r^L$.
\begin{equation}\label{exdiag7}
\begin{tikzcd}
                    & {\mathcal{N}_{f^*}} \arrow[rd, "\varsigma_1"] \arrow[ld, "f_!"'] &                     \\
\mathcal{V}^1_{f^*} &                                                            & \mathcal{V}^2_{f^*}
\end{tikzcd}
\end{equation}
The $\infty$-category $\mathcal{V}^2_{f^*}$ is modeled by the dg-category $D_2$, which is the freely generated dg-category with two objects $y,z$ and morphisms of the following form,
\[
\begin{tikzcd}
y \arrow[rr, "b", bend left] &  & z \arrow[ll, "b^*"]
\end{tikzcd}
\]
with $|b|=0,~|b^*|=1$, see \Cref{algmod1}. The $\infty$-category $\mathcal{C}$ is thus equivalent to the derived $\infty$-category of the path algebra of the graded quiver
\[
\begin{tikzcd}
y \arrow[rr, "b", bend left] &  & z \arrow[ll, "b^*"] \arrow["l_z"', loop, distance=2em, in=125, out=55]
\end{tikzcd}
\]
with differential $d(l_z)=bb^*$.
\end{example}

\subsection{Spherical and projective modules}\label{sec6.3}

Let $\mathcal{T}$ be an ideal triangulation of an oriented marked surface ${\bf S}$. The $\infty$-category $\mathcal{D}(\mathscr{G}_\mathcal{T})$ contains 
\begin{itemize}
\item for each edge $e$ of $\mathcal{T}$ a module $S_e$ which is $3$-spherical if $e$ is an internal edge and exceptional if $e$ is a boundary edge.
\item for each interior marked point $v$ of ${\bf S}$ a module $C_v$ with its $k$-linear endomorphism object quasi-isomorphic to $H^*(S^1\times S^2,k)$.
\item for each edge $e$ of $\mathcal{T}$ the projective module $P_e=p_e\mathscr{G}_\mathcal{T}$, where $p_e$ is the lazy path.
\end{itemize}
The goal of this section is to identify the above listed modules in $\mathcal{H}(\Gamma_\mathcal{T},\mathcal{F}_\mathcal{T})\simeq \mathcal{D}(\mathscr{G}_\mathcal{T})$. In terms Smith's description of the $\infty$-category $\mathcal{D}^{\on{fin}}(\mathscr{G}(Q_\mathcal{T}^\circ,W_\mathcal{T}'))\subset \mathcal{D}(\mathscr{G}_\mathcal{T})$ as a full subcategory of an (untwisted) Fukaya-category of a Calabi-Yau threefold $Y$, see \cite[Section 5.4]{Smi15}, the $S_e$, where $e$ is an internal edge, correspond to Lagrangian matching spheres and the $C_v$ correspond to Lagrangian embeddings of $S^1\times S^2$ in $Y$.

We can use the limit descriptions of the $\infty$-categories of global sections to geometrically describe the objects $S_e$ and $C_v$, as coCartesian sections of the Grothendieck construction $\Gamma(\mathcal{F}_{\mathcal{T}})\rightarrow \on{Exit}(\mathcal{T})$, see also \Cref{sec2.1} at i). We begin with the case where $e$ is an internal edge. Locally at $e$, we find $\mathcal{F}_\mathcal{T}$ to be up to natural equivalence of the following form:
\[ \mathcal{V}^{3}_{f^*} \xlongrightarrow{\varrho_3} \on{Fun}(S^{2},\mathcal{D}(k)) \xlongleftarrow{T\circ \varrho_3}\mathcal{V}^{3}_{f^*}\]
Let $\iota_1:\mathcal{D}(k)\rightarrow \mathcal{V}^3_{f^*}$ denote the inclusion of the first component of the semiorthogonal decomposition. We find a coCartesian section of $\Gamma(\mathcal{F}_\mathcal{T})$, which is locally at $e$ of the form
\[\iota_1(k)\xlongrightarrow{!} f^*(k) \xlongleftarrow{!} \iota_1(k)\,, \]
and vanishes otherwise, see also \Cref{cocartnot1}. We define $S_e$ as this coCartesian section. Suppose now that $e$ is a boundary edge. Locally at $e$, $\mathcal{F}_\mathcal{T}$ is given up to natural equivalence as follows:
\[\on{Fun}(S^{2},\mathcal{D}(k)) \xlongleftarrow{\varrho_3}\mathcal{V}^{3}_{f^*}\]
We define $S_e$ to be the coCartesian section of $\Gamma(\mathcal{F}_\mathcal{T})$ which is at $e$ of the form
\[ f^*(k) \xlongleftarrow{!} \iota_1(k)\,, \] and vanishes otherwise.

For the definition of $C_v$, we consider an interior marked point of ${\bf S}$. Locally around $v$, $\mathcal{F}_\mathcal{T}$ is up to natural equivalence of the following from.
\begin{equation}\label{loccirceq}
\begin{tikzcd}[column sep=small]
                               &                                                                                          & {\on{Fun}(S^2,\mathcal{D}(k))}                                                           &                                                                                          &                                \\
                               & {\on{Fun}(S^2,\mathcal{D}(k))}                                                           & \mathcal{V}^3_{f^*} \arrow[l, "\varrho_1"] \arrow[r, "\varrho_2"] \arrow[u, "\varrho_3"] & {\on{Fun}(S^2,\mathcal{D}(k))}                                                           &                                \\
{\on{Fun}(S^2,\mathcal{D}(k))} & \mathcal{V}^3_{f^*} \arrow[d, "\varrho_1"] \arrow[u, "\varrho_2"] \arrow[l, "\varrho_3"] &                                                                                          & \mathcal{V}^3_{f^*} \arrow[u, "\varrho_1"] \arrow[d, "\varrho_2"] \arrow[r, "\varrho_3"] & {\on{Fun}(S^2,\mathcal{D}(k))} \\
                               & {\on{Fun}(S^2,\mathcal{D}(k))}                                                           & \dots \arrow[l]                                                                          & \dots                                                                                    &                               
\end{tikzcd}
\end{equation}
The corresponding object $C_v$ is given by the coCartesian section which locally at $v$ is given as follows and vanishes everywhere else. Below, $\iota_3$ is the inclusion of the third component of the semiorthogonal decomposition of $\mathcal{V}^3_{f^*}$.
\[
\begin{tikzcd}
  &                                                                                      & 0                                                                                    &                                                                                      &   \\
  & f^*(k)                                                                               & \iota_3(f^*(k)) \arrow[l, "\varrho_1"] \arrow[r, "\varrho_2"] \arrow[u, "\varrho_3"] & f^*(k)                                                                               &   \\
0 & \iota_3(f^*(k)) \arrow[d, "\varrho_1"] \arrow[u, "\varrho_2"] \arrow[l, "\varrho_3"] &                                                                                      & \iota_3(f^*(k)) \arrow[u, "\varrho_1"] \arrow[d, "\varrho_2"] \arrow[r, "\varrho_3"] & 0 \\
  & f^*(k)                                                                               & \dots \arrow[l]                                                                      & \dots                                                                                &  
\end{tikzcd}
\]

We now describe the projective modules $P_e$. To simplify notation, we do not distinguish in notation between $P_e\in \mathcal{D}(\mathscr{G}_\mathcal{T})$ and the corresponding element in the equivalent $\infty$-category $\mathcal{H}(\Gamma_\mathcal{T},\mathcal{F}_\mathcal{T})$. We denote by $\on{ev}_e:\mathcal{H}(\Gamma_\mathcal{T},\mathcal{F}_\mathcal{T})\rightarrow \on{Fun}(S^{2},\mathcal{D}(k))$ the evaluation functor at the vertex $e\in\on{Exit}(\Gamma_\mathcal{T})$. Note that limits and colimits in $\mathcal{H}(\Gamma_\mathcal{T},\mathcal{F}_\mathcal{T})$ are computed pointwise, so that the functor $\on{ev}_e$ preserves all limits and colimits and thus admits a left adjoint, denoted $\on{ev}_e^*$. Consider further the adjoint functors $i_!:\mathcal{D}(k)\leftrightarrow \on{Fun}(S^2,\mathcal{D}(k)):i^*$ defined in the discussion preceding \Cref{pbpulem}.

\begin{proposition}\label{glgenprop}
The left adjoint $\on{ev}_e^*i_!$ of the functor  
\[i^* \on{ev}_e:\mathcal{H}(\Gamma_\mathcal{T},\mathcal{F}_\mathcal{T})\rightarrow \mathcal{D}(k)\]
satisfies $\on{ev}_e^*i_!(k)\simeq P_e$.
\end{proposition}

\begin{proof}
We prove the statement via an induction on the number of ideal triangles of $\mathcal{T}$. Let $\mathcal{T}_1$ be the ideal triangulation consisting of a single not self-folded triangle. \Cref{algmod2} shows that $\on{ev}_e^*i_!$ is modeled by the dg-functor $\alpha:k\rightarrow D_{3}$ determined by mapping the unique object of $k$ to any one of the objects of $D_{3}$. This shows that $\on{ev}_e^*i_!(k)\simeq P_e$, as desired. The case of the triangulation consisting of a single self-folded triangle is treated similarly and left to the reader. 

Suppose the statement has been shown for all ideal triangulations $\mathcal{T}$ with at most $n$ ideal triangles. The setup is as in the induction step in the proof of \Cref{thm2}: we consider an ideal triangulation $\mathcal{T}_{n+1}$ with $n+1$ ideal triangles obtained via the gluing of ideal triangulations $\mathcal{T}_{n}$ and $\mathcal{T}_1$ with $n$, respectively, $1$ ideal triangles along $s$ boundary edges. If $e$ is an edge of $\mathcal{T}_{n}$, the functor $\mathcal{D}(k)\xrightarrow{\on{ev}_e^* i_!} \mathcal{H}(\Gamma_{\mathcal{T}_{n+1}},\mathcal{F}_{\mathcal{T}_{n+1}})$ factors as 
\[ \mathcal{D}(k)\xrightarrow{\on{ev}_e^*i_!}\mathcal{H}(\Gamma_{\mathcal{T}_n},\mathcal{F}_{\mathcal{T}_n})\xlongrightarrow{\alpha} \mathcal{H}(\Gamma_{\mathcal{T}_{n+1}},\mathcal{F}_{\mathcal{T}_{n+1}})\,,\]
where $\alpha$ is modeled by a dg-functor $C_{\mathcal{T}_n}\rightarrow C_{\mathcal{T}_{n+1}}$ and thus satisfies $\alpha(p_e\mathscr{G}_{\mathcal{T}_{n}})\simeq p_e\mathscr{G}_{\mathcal{T}_{n+1}}=P_e$. Using the induction assumption, it thus follows that $\on{ev}_e^*i_!(k)\simeq P_e$. If $e$ is an edge of $\mathcal{T}_1$, we can argue analogously. This completes the induction.
\end{proof}

\subsection{Derived equivalences arising from flips of the triangulation}\label{sec6.4}

We consider two ideal triangulations $\mathcal{T}$ and $\mathcal{T}'$ of a surface related by the flip of internal edge $e$ of $\mathcal{T}$ which is not self-folded. Locally at $e$, the change in the triangulations can be depicted as follows.

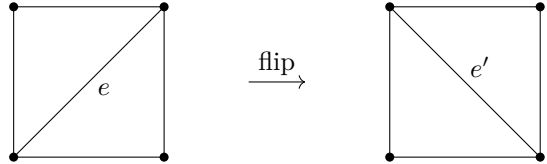
\begin{center}
\begin{tikzpicture}[scale=1]
  \draw
    (0, 0)
    (2, 2)
    (0, 2)
    (2, 0) 
    (0, 2) -- (0, 0)
    (2, 0) -- (0, 0)
    (2, 2) -- (2, 0)
    (0, 2) -- (2, 2)
    (0, 0) -- (2, 2); 
  \node (0) at (0, 0){};
  \node (1) at (0, 2){};
  \node (2) at (2, 0){};
  \node (3) at (2, 2){};   
  \node (4) at (1.2,0.9){$e$};
  \fill (0) circle (0.06);
  \fill (1) circle (0.06);
  \fill (2) circle (0.06);
  \fill (3) circle (0.06); 

  \node (A) at (3, 1){};
  \node (B) at (4, 1){};
  \node at (3.5, 1.25){flip};
  \draw [->] (A) edge (B);  
  
  \draw
    (5, 0)
    (7, 2)
    (5, 2)
    (7, 0) 
    (5, 2) -- (5, 0)
    (7, 0) -- (5, 0)
    (7, 2) -- (7, 0)
    (5, 2) -- (7, 2)
    (7, 0) -- (5, 2); 
  \node (5) at (5, 0){};
  \node (6) at (5, 2){};
  \node (7) at (7, 0){};
  \node (8) at (7, 2){};    
  \node (9) at (6.2,1.2){$e'$}; 
  \fill (5) circle (0.06);
  \fill (6) circle (0.06);
  \fill (7) circle (0.06);
  \fill (8) circle (0.06);   
\end{tikzpicture}
\end{center}

As shown in \cite{Lab09} the flip at the edge $e$ corresponds on the associated quivers with potential to the quiver mutation at the vertex of the quiver corresponding to $e$. It is shown in \cite[Section 7.6]{Kel11} that there exists an associated equivalence between the derived categories of the $3$-CY Ginzburg algebras $\mathscr{G}(Q_\mathcal{T}^\circ,W_\mathcal{T})$ and $\mathscr{G}(Q^\circ_{\mathcal{T}'},W_{\mathcal{T}'})$. The goal of this section is to associate to the flip an equivalence between the derived $\infty$-categories of the relative Ginzburg algebras $\mathscr{G}_\mathcal{T}$ and $\mathscr{G}_{\mathcal{T}'}$, thus extending the combined result of \cite{Lab09},\cite{Kel11} to relative Ginzburg algebras. 

From the perspective of dual ribbon graphs, the flip at $e$ relates the two ribbon graphs $\Gamma_{\mathcal{T}}$ and $\Gamma_{\mathcal{T}'}$, which locally differ as follows:
\begin{equation}\label{rgflip}
\begin{tikzcd}[row sep=small, column sep=small]
                      & {} \arrow[d, no head]     &                                             &    \\
{} \arrow[r, no head] & \cdot \arrow[rd, no head] &                                             &    \\
                      &                           & \cdot \arrow[r, no head] \arrow[d, no head] & {} \\
                      &                           & {}                                          &   
\end{tikzcd}
\overset{flip}{\leftrightsquigarrow} 
\begin{tikzcd}[row sep=small, column sep=small]
   &                                             & {} \arrow[d, no head]     &                       \\
   &                                             & \cdot \arrow[ld, no head] & {} \arrow[l, no head] \\
{} & \cdot \arrow[d, no head] \arrow[l, no head] &                           &                       \\
   & {}                                          &                           &                      
\end{tikzcd}
\end{equation}
The flip \eqref{rgflip} can be described by the pair of spans of contractions of ribbon graphs \eqref{span1} and \eqref{span2} below. 

\begin{equation}\label{span1}
\begin{tikzcd}[row sep=small, column sep=small]
                               & {} \arrow[d, no head, maps to] &                                &                                \\
{} \arrow[r, no head, maps to] & \cdot \arrow[rd, no head]     &                                &                                \\
                               &                                & \cdot                         & {} \arrow[l, no head, maps to] \\
                               &                                & {} \arrow[u, no head, maps to] &                               
\end{tikzcd} 
\xlongleftarrow{c_1} 
\begin{tikzcd}[row sep=small, column sep=small]
                               &                           & {} \arrow[d, no head, maps to] &                                &                           &                                \\
{} \arrow[r, no head, maps to] & \cdot \arrow[r, no head] & \cdot \arrow[rd, no head]      &                                &                           &                                \\
                               &                           &                                & \cdot                          & \cdot \arrow[l, no head] & {} \arrow[l, no head, maps to] \\
                               &                           &                                & {} \arrow[u, no head, maps to] &                           &                               
\end{tikzcd} 
\xlongrightarrow{c_2} 
\begin{tikzcd}[row sep=small, column sep=small]
                               &                           & {} \arrow[d, no head, maps to] &        &                                \\
{} \arrow[r, no head, maps to] & \cdot \arrow[r, no head] & \cdot \arrow[r, no head]       & \cdot & {} \arrow[l, no head, maps to] \\
                               &                           & {} \arrow[u, no head, maps to] &        &                               
\end{tikzcd}
\end{equation}
\begin{equation}\label{span2}
\begin{tikzcd}[row sep=small, column sep=small]
                               &                           & {} \arrow[d, no head, maps to] &        &                                \\
{} \arrow[r, no head, maps to] & \cdot \arrow[r, no head] & \cdot \arrow[r, no head]       & \cdot & {} \arrow[l, no head, maps to] \\
                               &                           & {} \arrow[u, no head, maps to] &        &                               
\end{tikzcd}
\xlongleftarrow{c_3} 
\begin{tikzcd}[row sep=small, column sep=small]
                               &                           &                                & {} \arrow[d, no head, maps to]               &        &                                \\
                               &                           &                                & \cdot \arrow[r, no head] \arrow[ld, no head] & \cdot & {} \arrow[l, no head, maps to] \\
{} \arrow[r, no head, maps to] & \cdot \arrow[r, no head] & \cdot                          &                                              &        &                                \\
                               &                           & {} \arrow[u, no head, maps to] &                                              &        &                               
\end{tikzcd} 
\xlongrightarrow{c_4}
\begin{tikzcd}[row sep=small, column sep=small]
                               &                                & {} \arrow[d, no head, maps to] &                                \\
                               &                                & \cdot \arrow[ld, no head]     & {} \arrow[l, no head, maps to] \\
{} \arrow[r, no head, maps to] & \cdot                         &                                &                                \\
                               & {} \arrow[u, no head, maps to] &                                &                               
\end{tikzcd}
\end{equation}

Denote by 
\[ T:\mathcal{N}_{f^*}\overset{\eqref{model1}}{\simeq} \mathcal{D}(k[t_1])\xrightarrow{\varphi^*} \mathcal{D}(k[t_1])\overset{\eqref{model1}}{\simeq} \mathcal{N}_{f^*}\] 
the autoequivalence, with $\varphi:k[t_{1}]\xrightarrow{t_1\mapsto -t_1} k[t_{1}]$. The contractions \eqref{span1} and \eqref{span2} give rise to the following arrangement of perverse schobers below, related via equivalences and pushforwards along the contractions. Note that no edges joining two singularities of the perverse schobers are contracted. Each step below preserves the $\infty$-category of global sections up to a canonical equivalence of $\infty$-categories.

\begin{adjustwidth}{-0.5in}{-0.5in}

\begin{equation}\label{schobers0}
\begin{tikzcd}[row sep=12, column sep=small]
                                            & {} \arrow[d, "\varrho_2", no head, maps to]              &                                              &                                              \\
{} \arrow[r, "\varrho_3", no head, maps to] & f^* \arrow[rd, "{(\varrho_1,T\circ\varrho_1)}", no head] &                                              &                                              \\
                                            &                                                          & f^*                                          & {} \arrow[l, "\varrho_3"', no head, maps to] \\
                                            &                                                          & {} \arrow[u, "\varrho_2"', no head, maps to] &                                             
\end{tikzcd} \simeq 
\end{equation}

\begin{equation}\label{schobers1}
\simeq
\begin{tikzcd}[row sep=12, column sep=small]
                                      & {} \arrow[d, "\varrho_2", no head, maps to]  &                                        &                                        \\
{} \arrow[r, "\varrho_3", no head, maps to] & f^* \arrow[rd, "{(\varrho_1,\varrho_1)}", no head] &                                        &                                        \\
                                      &                                        & f^*                                    & {} \arrow[l, "T\circ\varrho_3"', no head, maps to] \\
                                      &                                        & {} \arrow[u, "T\circ \varrho_2"', no head, maps to] &                                       
\end{tikzcd}
\xlongleftarrow{~(c_1)_*~}
\begin{tikzcd}[row sep=12, column sep=small]
                                            &                                                   & [10pt] {} \arrow[d, "\varrho_2", no head, maps to]                          &                                                    &                                                    &                                                    \\
{} \arrow[r, "\varrho_2", no head, maps to] & f^* \arrow[r, "{(\varrho_1,\varrho_3)}", no head] & 0_{\mathcal{N}_{f^*}} \arrow[rd, "{(\varrho_1,\varrho_1)}", no head] &                                                    &                                                    &                                                    \\
                                            &                                                   &                                                                      & 0_{\mathcal{N}_{f^*}}                              & f^* \arrow[l, "{(\varrho_3,\varrho_1)}"', no head] & {} \arrow[l, "T\circ\varrho_2"', no head, maps to] \\
                                            &                                                   &                                                                      & {} \arrow[u, "T\circ\varrho_2"', no head, maps to] &                                                    &                                                   
\end{tikzcd}
\simeq
\end{equation}

\begin{equation}\label{schobers2} 
\simeq
\begin{tikzcd}[row sep=12, column sep=small]
                                      &      [-5pt]                               & [15pt]{} \arrow[d, "{\varrho_1[1]}", no head, maps to]                    &                                        &                                        &                                       \\
{} \arrow[r, "\varrho_2", no head, maps to] & f^* \arrow[r, "{(\varrho_1,\varrho_2[1])}"', no head] & 0_{\mathcal{N}_{f^*}} \arrow[rd, "{(\varrho_3,\varrho_1)}", no head] &                                        &                                        &                                       \\
                                      &                                       &                                                          & 0_{\mathcal{N}_{f^*}}                  & f^* \arrow[l, "{(\varrho_3,\varrho_1)}"', no head] & {} \arrow[l, "T\circ\varrho_2"', no head, maps to] \\
                                      &                                       &                                                          & {} \arrow[u, "T\circ\varrho_2"', no head, maps to] &                                        &                                      
\end{tikzcd}
\xlongrightarrow{~(c_2)_*~}
\begin{tikzcd}[row sep=12, column sep=small]
                                      &                                      [-5pt]     & [15pt]{} \arrow[d, "{\varrho_1[1]}", no head, maps to]             &  [5pt]   &                                        \\
{} \arrow[r, "\varrho_2", no head, maps to] & f^* \arrow[r, "{(\varrho_1,\varrho_2[1])}"', no head] & 0_{\mathcal{N}_{f^*}} \arrow[r, "{(\varrho_4,\varrho_1)}", no head] & f^* & {} \arrow[l, "T\circ\varrho_2"', no head, maps to] \\
                                      &                                           & {} \arrow[u, "T\circ\varrho_3"', no head, maps to]                  &     &                                       
\end{tikzcd}
\simeq 
\end{equation}

\begin{equation}\label{schobers3}
\simeq
\begin{tikzcd}[row sep=12, column sep=small]
                                    &                                         [-5pt]  &  [15pt]{} \arrow[d, "{\varrho_4[1]}", no head, maps to]              &  [20pt]   &                                       [-5pt] \\
{} \arrow[r, "\varrho_2", no head, maps to] & f^* \arrow[r, "{(\varrho_1,\varrho_1[3])}", no head] & 0_{\mathcal{N}_{f^*}} \arrow[r, "{(\varrho_3[2],\varrho_1)}", no head] & f^* & {} \arrow[l, "T\circ \varrho_2"', no head, maps to] \\
                                      &                                          & {} \arrow[u, "{T\circ \varrho_2[2]}"', no head, maps to]             &     &                                       
\end{tikzcd}
\xlongleftarrow{~(c_3)_*~}
\begin{tikzcd}[row sep=12, column sep=small]
                                      &   [-5pt]                                       &                                                         [15pt]  & [-5pt]{} \arrow[d, "{\varrho_3[1]}", no head, maps to]              &   [10pt]  &                                        \\
                                      &                                          &                                                          & 0_{\mathcal{N}_{f^*}} \arrow[r, "{(\varrho_2[2],\varrho_1)}", no head] & f^* & {} \arrow[l, "T\circ \varrho_2"', no head, maps to] \\
{} \arrow[r, "\varrho_2", no head, maps to] & f^* \arrow[r, "{(\varrho_1,\varrho_1[3])}", no head] & 0_{\mathcal{N}_{f^*}} \arrow[ru, "{(\varrho_3,\varrho_1)}", no head] &                                                         &     &                                        \\
                                      &                                          & {} \arrow[u, "{T\circ \varrho_2[2]}", no head, maps to]               &                                                         &     &                                       
\end{tikzcd}
\simeq
\end{equation}

\begin{equation}\label{schobers4}
\simeq
\begin{tikzcd}[row sep=12, column sep=small]
                                      &                                       &  [20pt]                                                        & {} \arrow[d, "{\varrho_1[2]}", no head, maps to]                   &  [20pt]   &                                        \\
                                      &                                       &                                                          & 0_{\mathcal{N}_{f^*}} \arrow[r, "{(\varrho_3[2],\varrho_1)}", no head] & f^* & {} \arrow[l, "T\circ \varrho_2"', no head, maps to] \\
{} \arrow[r, "\varrho_2", no head, maps to] & f^* \arrow[r, "{(\varrho_1,\varrho_3[3])}", no head] & 0_{\mathcal{N}_{f^*}} \arrow[ru, "{(\varrho_2[1],\varrho_2)}", no head] &                                                         &     &                                        \\
                                      &                                       & {} \arrow[u, "{T\circ\varrho_1[3]}", no head, maps to]                    &                                                         &     &                                       
\end{tikzcd}
\xlongrightarrow{~(c_4)_*~}
\begin{tikzcd}[row sep=12, column sep=small]
                                      &                                        & {} \arrow[d, "{\varrho_1[2]}", no head, maps to] &                                       \\
                                      &                                        & f^*                                   & {} \arrow[l, "{T\circ\varrho_3[2]}", no head, maps to] \\
{} \arrow[r, "{\varrho_3[3]}", no head, maps to] & f^* \arrow[ru, "{(\varrho_2[1],\varrho_2)}", no head] &                                       &                                       \\
                                      & {} \arrow[u, "{T\circ\varrho_1[3]}", no head, maps to]  &                                       &                                      
\end{tikzcd}
\simeq 
\end{equation}

\begin{equation}\label{schobers5}
\simeq 
\begin{tikzcd}[row sep=12, column sep=small]
                                      &                                        & {} \arrow[d, "{T\circ\varrho_1}", no head, maps to] &                                       \\
                                      &                                        & f^*                                   & {} \arrow[l, "{\varrho_3}", no head, maps to] \\
{} \arrow[r, "{\varrho_3}", no head, maps to] & f^* \arrow[ru, "{(\varrho_2,T\circ \varrho_2)}", no head] &                                       &                                       \\
                                      & {} \arrow[u, "{T\circ \varrho_1}", no head, maps to]  &                                       &                                      
\end{tikzcd}
\end{equation}
\end{adjustwidth}

Each of the above equivalences of parametrized perverse schober is nontrivial only at one or two vertices with label $0_{\mathcal{N}_{f^*}}$, where it is given by the paracyclic twist functor $T_{\mathcal{V}^n_{0_{\mathcal{N}_{f^*}}}}$ of \Cref{sec3.2}, see also \Cref{paralem4}, except for the equivalence between the parametrized perverse schober in \eqref{schobers0} and the left parametrized perverse schober in \eqref{schobers1} and the equivalence between the right parametrized perverse schober of \eqref{schobers4} and the parametrized perverse schober of \eqref{schobers5}. The former is nontrivial only at the lower vertex labeled $f^*$, where it is given by the autoequivalence $\epsilon$ of $\mathcal{V}^3_{f^*}$ which restricts on both the components $\on{Fun}(S^2,\mathcal{D}(k))$ of the semiorthogonal decomposition to $T$ and on the component $\mathcal{D}(k)$ of the semiorthogonal decomposition to the identity functor. The latter equivalence of parametrized perverse schobers is nontrivial at three objects of the exit path category corresponding to the two vertices and the edge connecting them. At the lower vertex, the equivalence is given by $[3]$, at the upper vertex by $\epsilon \circ [2]$ and at the object of $\on{Exit}(\mathcal{T})$ corresponding to the diagonal edge by $[-2]$. 

We assume for the moment that neither $e$ nor its flip $e'$ are the outer edge of a self-folded triangle. The perverse schober $\mathcal{F}_{\mathcal{T}}$ from \Cref{thm2} can be chosen to restrict locally at $e$ to the perverse schober \eqref{schobers0}. We can thus describe $\mathcal{F}_{\mathcal{T}}$ as the gluing (in the sense of \Cref{sgluelem}) of the parametrized perverse schober of \eqref{schobers0} and its complement in $\mathcal{F}_{\mathcal{T}}$.  Similarly, $\mathcal{F}_{\mathcal{T}'}$ can be chosen to be the gluing of the complement of the parametrized perverse schober \eqref{schobers0} in $\mathcal{F}_{\mathcal{T}}$ and the parametrized perverse schober \eqref{schobers5}. 

We can thus glue the complement of the parametrized perverse schober \eqref{schobers0} in $\mathcal{F}_\mathcal{T}$ with the parametrized perverse schobers \eqref{schobers0}-\eqref{schobers5} and use that the global sections are in each step preserved up to equivalence of $\infty$-categories, see \Cref{conprop}, to obtain an equivalence of $\infty$-categories
\begin{equation}\label{muteqeq}
\mu_e:\mathcal{H}(\Gamma_{\mathcal{T}},\mathcal{F}_{\mathcal{T}})\rightarrow \mathcal{H}(\Gamma_{\mathcal{T}'},\mathcal{F}_{\mathcal{T}'})
\end{equation}
which we call the mutation equivalence at $e$. 

If $e$ or $e'$, say $e$, is the outer edge of a self-folded triangle, we cannot take the naive  complement of \eqref{schobers0} in $\mathcal{F}_\mathcal{T}$ because the underlying ribbon graph would have an edge without any endpoints. We change $\Gamma_\mathcal{T}$ by replacing the edge $e$ by the ribbon graph 
\[
\begin{tikzcd}
{} \arrow[r, no head, maps to] & \cdot & {} \arrow[l, "f", no head, maps to]
\end{tikzcd}
\]
and obtain a parametrized perverse schober $\tilde{\mathcal{F}}_\mathcal{T}$ by changing $\mathcal{F}_\mathcal{T}$ at $e$ to 
\[\begin{tikzcd}
{} \arrow[r, no head, maps to] & 0_{\mathcal{N}_{f^*}} & {} \arrow[l, no head, maps to]
\end{tikzcd}\,.\]
We also change $\mathcal{F}_{\mathcal{T}'}$ accordingly to $\tilde{\mathcal{F}}_{\mathcal{T}'}$. The contraction of the edge $f$ then induces equivalences between global sections of $\mathcal{F}_\mathcal{T}$ and $\tilde{\mathcal{F}}_\mathcal{T}$ as well as $\mathcal{F}_{\mathcal{T}'}$ and $\tilde{\mathcal{F}}_{\mathcal{T}'}$. Composing with the equivalence on global sections as in \eqref{muteqeq} applied to $\tilde{\mathcal{F}}_\mathcal{T}$, we obtain the equivalence of $\infty$-categories \eqref{muteqeq}.

\begin{remark}
The mutation equivalence $\mu_e$ maps global sections of $\mathcal{F}_{\mathcal{T}}$ with support on $\Gamma_{\mathcal{T}}^\circ$ to global sections of $\mathcal{F}_{\mathcal{T}'}$ which support on $\Gamma_{\mathcal{T}'}^\circ$. We thus obtain an equivalence of $\infty$-categories
\[\mathcal{D}(Q_\mathcal{T}^\circ,W_{\mathcal{T}}')\simeq \mathcal{H}_{\Gamma_\mathcal{T}^\circ}(\Gamma_{\mathcal{T}},\mathcal{F}_{\mathcal{T}})\xrightarrow{\mu_e} \mathcal{H}_{\Gamma_{\mathcal{T}'}^\circ}(\Gamma_{\mathcal{T}'},\mathcal{F}_{\mathcal{T}'})\simeq \mathcal{D}(Q_{\mathcal{T}'}^\circ,W_{\mathcal{T}'}')\,.\]
\end{remark}

\begin{proposition}\label{mutprop1}
Let $\mathcal{T},\mathcal{T}'$ be two ideal triangulations of a surface, differing by a flip of an internal edge $e\in \mathcal{T}$ to $e'\in \mathcal{T}'$. Given an edge $f\neq e$ of $\mathcal{T}$, we denote by $f'$ the corresponding edge of $\mathcal{T}'$.
\begin{enumerate} 
\item There exist an equivalence in $\mathcal{H}(\Gamma_{\mathcal{T}'},\mathcal{F}_{\mathcal{T}'})$
\begin{equation}\label{peq1} 
\mu_e(P_e)\simeq \on{cof}\big(P_{e'}\rightarrow \bigoplus_{\alpha\in (Q_{\mathcal{T}})_1,\,s(\alpha)=e} P_{t(\alpha)'}\big)\,.
\end{equation}
\item Let $f\neq e$ be an edge of $\mathcal{T}$. There exists an equivalence in $\mathcal{H}(\Gamma_{\mathcal{T}'},\mathcal{F}_{\mathcal{T}'})$
\[ \mu_e(P_f)\simeq P_{f'}\,.\]
\end{enumerate}
\end{proposition}

\begin{proof}
We begin by showing statement $1$. The idea of the proof is to trace through the equivalences on global sections induced by \eqref{schobers0}-\eqref{schobers5} and describe the composition of the inverses of these equivalences with the evaluation at a point functor $\mathcal{H}(\Gamma_{\mathcal{T}},\mathcal{F}_{\mathcal{T}})\xrightarrow{\on{ev}_e}\on{Fun}(S^2,\mathcal{D}(k))\xrightarrow{i^*}\mathcal{D}(k)$. Passing to left adjoints and evaluating at $k\in \mathcal{D}(k)$ yields the image of $P_e$ under $\mu_e$.

We denote the complement of the parametrized perverse schober of \eqref{schobers0} in $\mathcal{F}_{\mathcal{T}}$ by $\mathcal{F}^c$. 
 
We denote the gluing of $\mathcal{F}^c$ with the parametrized perverse schober
\begin{itemize}
\item on the right of \eqref{schobers2} by $\mathcal{G}_1$.
\item on the left of \eqref{schobers3} by $\mathcal{G}_2$.
\item on the right of \eqref{schobers3} by $\mathcal{G}_3$. 
\item on the left of \eqref{schobers4} by $\mathcal{G}_4$.
\end{itemize}
The evaluation of $\mathcal{G}_1$ at the central vertex yields the $\infty$-category 
\[ \mathcal{V}^4_{0_{\mathcal{N}_{f^*}}}\simeq \{\on{Fun}(S^2,\mathcal{D}(k)),\on{Fun}(S^2,\mathcal{D}(k)),\on{Fun}(S^2,\mathcal{D}(k))\}\,.\]
A direct computation shows that the composite of the equivalence $\lim \mathcal{G}_1\simeq \mathcal{H}(\Gamma_{\mathcal{T}},\mathcal{F}_{\mathcal{T}})$ with $\on{ev}_{e}$ is given by the composite functor $R_1$ of the evaluation at the central vertex (labeled $\mathcal{V}^4_{0_{\mathcal{N}_{f^*}}}$) with the restriction functor to the second component of the semiorthogonal decomposition of $\mathcal{V}^4_{0_{\mathcal{N}_{f^*}}}$. Precomposing the functor $R_1:\lim \mathcal{G}_1\rightarrow \on{Fun}(S^2,\mathcal{D}(k))$ with the equivalence $\lim \mathcal{G}_2\simeq \lim \mathcal{G}_1$ yields the functor $R_2$ given by the composite of the evaluation functor to the central vertex (labeled $\mathcal{V}^4_{0_{\mathcal{N}_{f^*}}}$) with the functor $\on{cof}_{1,3}[1]$, which is the composite of the suspension of the cofiber functor with the restriction functor to the first and third component of the semiorthogonal decomposition. We denote the diagonal edge of the central ribbon graph of \eqref{span2} by $e''$, adjacent to the upper vertex denoted $v_1$ and the lower vertex denoted $v_2$. Note that $\on{ev}_{e''}\simeq \varrho_1\circ \on{ev}_{v_1}\simeq \varrho_3\circ \on{ev}_{v_2}$. Further, there exist canonical natural transformations 
\[a:\varrho_3[-1]=\pi_1\longrightarrow \pi_2=\varrho_1\,,\] 
given at a vertex $x\xrightarrow{\alpha}y\in \mathcal{V}^3_{0_{\mathcal{N}_{f^*}}}=\{\on{Fun}(S^2,\mathcal{D}(k)),\on{Fun}(S^2,\mathcal{D}(k))\}$ by the edge $\alpha$ and a canonical natural transformation 
\[ b:\varrho_2=\on{cof}_{1,2}\longrightarrow\pi_1[1]=\varrho_3\,,\]
given by $b=a\circ T_{\mathcal{V}^3_{0_{\mathcal{N}_{f^*}}}}[-1]$. Again, it can be directly checked that the composition of $R_2$ with the equivalence $\lim \mathcal{G}_3\simeq \lim \mathcal{G}_2$ induced by the contraction $c_3$ yields the functor $R_3:\lim \mathcal{G}_3\rightarrow \on{Fun}(S^2,\mathcal{D}(k))$ given by the suspension of the cofiber 
\[ \on{cof}\big(\varrho_3[-1]\circ \on{ev}_{v_1}\oplus \varrho_2\circ \on{ev}_{v_2} \xrightarrow{(a\circ \on{ev}_{v_1},b\circ \on{ev}_{v_2})} \on{ev}_{e''}\big)[1]\]
in the stable $\infty$-category $\on{Fun}(\on{lim}\mathcal{G}_3,\on{Fun}(S^2,\mathcal{D}(k))$. The composition $R_4$ of $R_3$ with $\lim \mathcal{G}_4\simeq \lim \mathcal{G}_3$ yields the functor given by the suspension of the cofiber
\[ \on{cof}\big(\varrho_1\circ \on{ev}_{v_1}\oplus \varrho_1[1]\circ \on{ev}_{v_2} \rightarrow\on{ev}_{e''}\big)[1]\simeq \on{cof}\big(\on{ev}_{f_1}[-2]\oplus \on{ev}_{f_2}[-2] \rightarrow\on{ev}_{e''}\big)[1]\,.\]
Here $f_1$ and $f_2$ denote the two edges of $\Gamma_\mathcal{T}$ which precede $e$ in the cyclic order of the edges at the two vertices incident to $e$ (we consider these edges as edges of the ribbon graph underlying \eqref{schobers4}). Note also that the two edges $f_1,f_2$ exactly describe the targets of the arrows $\alpha\in (Q_\mathcal{T})_1$ satisfying $s(\alpha)=e$. Continuing as before, we see that the composite of the equivalence $\mathcal{H}(\Gamma_{\mathcal{T}'},\mathcal{F}_{\mathcal{T}'})\simeq \mathcal{H}(\Gamma_{\mathcal{T}},\mathcal{F}_{\mathcal{T}})$ with $\on{ev}_{e}$ yields the functor 
\[ R=\on{cof}\big(\on{ev}_{f_1}[-2]\oplus \on{ev}_{f_2}[-2]\rightarrow \on{ev}_{e'}[-2]\big)[1] \simeq \on{fib}\big(\on{ev}_{f_1}\oplus \on{ev}_{f_2}\rightarrow \on{ev}_{e'}\big)\,.\] 
It follows from \Cref{glgenprop} that $\on{ev}_g^*i_!\simeq \mhyphen \otimes P_g$ for $g=e',f_1,f_2$. The functor $i^*\circ \on{ev}_g$ is thus equivalent to the morphism object functor $\on{Mor}_{\mathcal{H}(\Gamma_{\mathcal{T}'},\mathcal{F}_{\mathcal{T}'})}(P_e,\mhyphen)$ with respect to the $k$-linear structure of $\mathcal{H}(\Gamma_{\mathcal{T}'},\mathcal{F}_{\mathcal{T}'})\simeq \mathcal{D}(\mathscr{G}_{\mathcal{T}'})$, see \cite[4.2.1.28]{HA}. It follows that the functor $i^*\circ R$ is equivalent to 
\[ \on{Mor}_{\mathcal{H}(\Gamma_{\mathcal{T}'},\mathcal{F}_{\mathcal{T}'})}\big(\on{cof}\big(P_{e'}\rightarrow P_{f_1}\oplus P_{f_2}\big),\mhyphen\big)\] 
for some morphism $P_{e'}\rightarrow P_{f_1}\oplus P_{f_2}$ and the left adjoint thus maps $k\in \mathcal{D}(k)$ to $\on{cof}\big(P_{e'}\rightarrow P_{f_1}\oplus P_{f_2}\big)$, showing statement 1.

Statement 2 can be approached like statement 1., but is more immediate, because the respective edges of the ribbon graphs corresponding to the the projective objects are not affected by \eqref{schobers0}-\eqref{schobers5}.
\end{proof}

\begin{remark}
The formulas in \Cref{mutprop1} for the images of the the projective modules under $\mu_e$ recover the formulas given in the context of completed non-relative Ginzburg algebras in \cite{KY11}. In the context of completed Ginzburg algebras, it is noted in \cite[Theorem 7.4]{Kel12} that there exist two mutation equivalences for each vertex $i$ of the quiver, whose values on objects differ by the spherical twist of the spherical object associated to $i$. In terms of our construction, we can also produce a second mutation equivalence $\mu_e'$ which differs on objects by the spherical twist around $S_e$ by replacing the spans of ribbons graphs \eqref{span1} and \eqref{span2} by the spans 
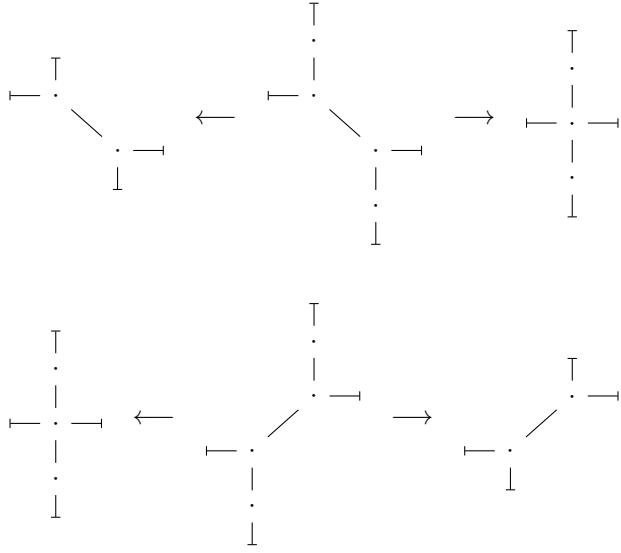
\begin{equation*}
\begin{tikzcd}[row sep=small, column sep=small]
                               & {} \arrow[d, no head, maps to] &                                &                                \\
{} \arrow[r, no head, maps to] & \cdot \arrow[rd, no head]     &                                &                                \\
                               &                                & \cdot                         & {} \arrow[l, no head, maps to] \\
                               &                                & {} \arrow[u, no head, maps to] &                               
\end{tikzcd} 
\longleftarrow
\begin{tikzcd}[row sep=small, column sep=small]
                               & {} \arrow[d, no head, maps to] &                                &                                \\
                               & \cdot \arrow[d, no head]      &                                &                                \\
{} \arrow[r, no head, maps to] & \cdot \arrow[rd, no head]      &                                &                                \\
                               &                                & \cdot                          & {} \arrow[l, no head, maps to] \\
                               &                                & \cdot \arrow[u, no head]      &                                \\
                               &                                & {} \arrow[u, no head, maps to] &                               
\end{tikzcd}
\longrightarrow
\begin{tikzcd}[row sep=small, column sep=small]
                               & {} \arrow[d, no head, maps to] &                                \\
                               & \cdot \arrow[d, no head]      &                                \\
{} \arrow[r, no head, maps to] & \cdot                          & {} \arrow[l, no head, maps to] \\
                               & \cdot \arrow[u, no head]      &                                \\
                               & {} \arrow[u, no head, maps to] &                               
\end{tikzcd}
\end{equation*}
\begin{equation*}
\begin{tikzcd}[row sep=small, column sep=small]
                               & {} \arrow[d, no head, maps to] &                                \\
                               & \cdot \arrow[d, no head]      &                                \\
{} \arrow[r, no head, maps to] & \cdot                          & {} \arrow[l, no head, maps to] \\
                               & \cdot \arrow[u, no head]      &                                \\
                               & {} \arrow[u, no head, maps to] &                               
\end{tikzcd}
\longleftarrow
\begin{tikzcd}[row sep=small, column sep=small]
                               &                                & {} \arrow[d, no head, maps to] &                                \\
                               &                                & \cdot \arrow[d, no head]      &                                \\
                               &                                & \cdot \arrow[ld, no head]      & {} \arrow[l, no head, maps to] \\
{} \arrow[r, no head, maps to] & \cdot                          &                                &                                \\
                               & \cdot \arrow[u, no head]      &                                &                                \\
                               & {} \arrow[u, no head, maps to] &                                &                               
\end{tikzcd}
\longrightarrow
\begin{tikzcd}[row sep=small, column sep=small]
                               &                                & {} \arrow[d, no head, maps to] &                                \\
                               &                                & \cdot \arrow[ld, no head]     & {} \arrow[l, no head, maps to] \\
{} \arrow[r, no head, maps to] & \cdot                         &                                &                                \\
                               & {} \arrow[u, no head, maps to] &                                &                               
\end{tikzcd}
\end{equation*}
and adapting the construction of the mutation equivalence $\mu_e$ accordingly.

It remains an interesting problem to construct a natural equivalence $\mu_e\simeq T_{S_e}\circ \mu_e'$, where $T_{S_e}$ denotes the twist functor of the spherical adjunction induced by the spherical object $S_e$.  
\end{remark}

\section{Further directions}\label{sec7}

\subsection{Invariants of triangulated spin surfaces}\label{spinsec}

In this section we show that the equivalence class of the parametrized perverse schober $\mathcal{F}_\mathcal{T}$ constructed in \Cref{thm2} does not depend on any of the choices made in its construction. Before we can state a precise result we need to briefly discuss combinatorial models for spin surfaces, following \cite{DK15}.

\begin{definition}\label{spgrdef}
Let ${\bf S}$ be an oriented marked surface. Let $\Gamma$ be a graph with an embedding $f:|\on{Exit}(\Gamma)|\rightarrow {\bf S}\backslash M$ and consider the induced ribbon graph $\Gamma$, see \Cref{orrem}. Let $B=\partial|\on{Exit}(\Gamma)|$ be the boundary of $|\on{Exit}(\Gamma)|$ in ${\bf S}\backslash M$. We call $f$ (or by abuse of notation $\Gamma$) a \textit{spanning graph} for ${\bf S}$ if
\begin{enumerate}
\item the embedding $f$ is a homotopy equivalence,
\item $f$ induces a homotopy equivalence $B\rightarrow \partial {\bf S}\backslash M$.
\end{enumerate}
\end{definition}

\begin{example}\label{ribglem}
Let ${\bf S}$ be a marked surface with an ideal triangulation $\mathcal{T}$. The dual ribbon graph $\Gamma_\mathcal{T}$ is a spanning graph for ${\bf S}$. 
\end{example}

\begin{definition}
Let $\Gamma$ be a ribbon graph. We define the incidence diagram $I:\on{Exit}(\Gamma)\rightarrow \on{Set}$ to be the functor determined by
\begin{itemize}
\item $I(v)=\on{H}(v)$ for $v\in \Gamma_0\subset \on{Exit}(\Gamma)_0$,
\item $I(e)=\{e_1,e_2\}$ for an edge $e\in \Gamma_1$ consisting of halfedges $e_1,e_2$ (counted twice for external edges),
\item and assigning to a morphism $v\rightarrow e$ from a vertex $v\in \Gamma_0$ to an incident edge $e\in \Gamma_1$ consisting of $\{e_1,e_2\}$ with $\sigma(e_1)=v$ the morphism of sets $\on{H}(v)\rightarrow \{e_1,e_2\}$, mapping $e_1$ to $e_2$ and $\on{H}(v_1)\backslash \{e_1\}$ to $e_1$. 
\end{itemize}
\end{definition}

We proceed with the definition of the $2$-cyclic category $\Lambda_2$. The definition is similar to the definition of the paracyclic category $\Lambda_\infty$, see \Cref{linftydef}, with the difference that the cyclic automorphisms $\tau^n$ satisfy the additional relation $(\tau^{n})^{2(n+1)}=\on{id}_{[n]}$.

\begin{definition}
For $n\geq 0$, let $[n]$ denote the set $\{0,\dots,n\}$. The $2$-cyclic category $\Lambda_2$ has as objects the \mbox{sets $[n]$}. The morphism in $\Lambda_2$ are generated by morphisms
\begin{itemize}
\item $\delta^{0},\dots,\delta^{n}:[n-1]\rightarrow [n]$,
\item $\sigma^{0},\dots,\sigma^{n-1}:[n]\rightarrow [n-1]$,
\item $\tau^n:[n]\rightarrow [n]$,
\end{itemize}
subject to the simplicial relations and the further relations
\begin{align*} 
(\tau^n)^{2(n+1)}=\on{id}_{[n]},\quad\quad \\
\tau^{n}\delta^{i}=\delta^{i-1}\tau^{n-1}\text{ for }i>0,\quad\quad & \tau^{n}\delta^{0}=\delta^n,\\
\tau^{n}\sigma^{i}=\tau^{n+1}\sigma^{i-1}\text{ for }i>0,\quad\quad & \tau^{n}\sigma^0=\sigma^{n}\,.
\end{align*}
\end{definition}

\begin{definition}\label{fsdef}
Let ${\bf S}$ be an oriented marked surface and $\Gamma$ a spanning ribbon graph. Denote by $\mathbb{N}\subset \on{Set}$ the full subcategory spanned by objects of the form $[n]$ with $n\geq 0$. Choose any diagram 
\[ \tilde{I}:\on{Exit}(\Gamma)\rightarrow \mathbb{N}\,,\] 
equipped with a natural equivalence $\nu:I\xRightarrow{\simeq} \tilde{I}$, which respects the cyclic orders on $I(x)$ for $x\in \on{Exit}(\Gamma)$ and on $\tilde{I}(x)=[|I(x)|-1]$ (obtained from the apparent linear order). A 2-spin structure, or simply spin structure, on ${\bf S}\backslash (M\cap{\bf S}^\circ)$ (or on $\Gamma$) is a lift $\tilde{I}^2$
\[
\begin{tikzcd}
                                             & \Lambda_2 \arrow[d] \\
{\on{Exit}(\Gamma)} \arrow[r, "{\tilde{I}}"] \arrow[ru, "{\tilde{I}^2}" ] & \mathbb{N}                
\end{tikzcd}
\]
of the diagram $\tilde{I}$. 
\end{definition}

\begin{remark}
In the definition of spin structure on a surface we deviate from \cite{DK15} to keep the exposition more direct and better applicable. While a spin structure in \cite{DK15} consists of a $\mathbb{Z}_6$-torsor at every trivalent vertex of the ribbon graph and a $\mathbb{Z}_4$-torsor at every edge of the ribbon graph, we include a identity element for each torsor, encoded in $\nu$. Every spin structure in the sense of \Cref{fsdef} defines a spin structure in the sense of \cite{DK15} and vice versa. Note that in Lemma IV.26 in \cite{DK15} it is shown that the datum of a spin structure coincides with the more standard notion of a spin structure on ${\bf S}\backslash (M\cap {\bf S}^\circ)$ in the sense of a reduction of the structure group of the tangent bundle to the connected two-fold covering of $\on{GL}^+(2,\mathbb{R})\subset \on{GL}(2,\mathbb{R})$. 
\end{remark}

\begin{definition}
An equivalence of two ribbon graphs $\Gamma\simeq \Gamma'$ is defined as an equivalence of posets $\phi:\on{Exit}(\Gamma)\xrightarrow{\simeq} \on{Exit}(\Gamma')$.
\end{definition}

Note that an equivalence $\phi:\on{Exit}(\Gamma)\xrightarrow{\simeq} \on{Exit}(\Gamma')$ extends to a natural equivalence
\[ \eta_\phi:I_{\Gamma'}\circ \phi\xRightarrow{~\simeq~} I_{\Gamma}\] 
between the incidence diagrams.

\begin{definition}
An equivalence between two ribbon graphs with spin structure $(\Gamma,\tilde{I}^2_{\Gamma}),(\Gamma',{\tilde{I}^2}_{\Gamma'})$ consists of an equivalence $\phi:\Gamma\xrightarrow{\simeq} \Gamma'$ of ribbon graphs together with a lift 
\[ \eta^2_{\phi}:\tilde{I}^{2}_{\Gamma'}\circ \phi\xRightarrow{~\simeq~} \tilde{I}^2_{\Gamma}\] 
of 
\[\tilde{I}_{\Gamma'}\circ \phi\xRightarrow{\simeq} I_{\Gamma'}\circ \phi\xRightarrow{\eta_{\phi}} I_{\Gamma}\xRightarrow{\simeq} \tilde{I}_{\Gamma}\,.\]
\end{definition}

The main result of this section is the following.

\begin{proposition}\label{thm3}
Let ${\bf S}$ be an oriented marked surface equipped with an ideal triangulation $\mathcal{T}$ and let $\Sigma={\bf S}\backslash (M\cap {\bf S}^\circ)$ be the surface without the interior marked points.
\begin{enumerate}
\item For every spin structure $U$ on $\Sigma$ there exists a $\Gamma_\mathcal{T}$-parametrized perverse schober $\mathcal{F}_{\mathcal{T}}^U$. If two spin structures $U,U'$ on $\Sigma$ are equivalent then there exists an equivalence of parametrized perverse schobers 
\begin{equation}\label{ftueq}
\mathcal{F}_\mathcal{T}^U\simeq \mathcal{F}_\mathcal{T}^{U'}\,.
\end{equation}
If $\on{char}(k)\neq 2$, the converse is also true, i.e.~if there exists an equivalence as in \eqref{ftueq}, then $U\simeq U'$. 
\item Given a parametrized perverse schober $\mathcal{F}_\mathcal{T}$ as constructed in the proof of \Cref{thm2}, there exists a spin structure $U$ on $\Sigma$ and an equivalence of parametrized perverse schobers 
\[ \mathcal{F}_\mathcal{T}\simeq \mathcal{F}^U_\mathcal{T}\,.\]
\end{enumerate}
\end{proposition}

\begin{definition}
A ribbon graph  $\Gamma$ is called trivalent if all vertices have valency three.
\end{definition}

\begin{definition}
We denote by $M_2$ the subcategory of $\Lambda_2$ spanned by the two objects $[1],[2]$ and morphisms generated under composition by $\tau^1:[1]\rightarrow [1]$, $\tau^2:[2]\rightarrow [2]$ and $\delta^l:[1] \rightarrow [2]$ for $l=1,2,3$.
\end{definition}

\begin{remark}
Let $\Gamma$ be a trivalent ribbon graph and $\tilde{I}^2:\on{Exit}(\Gamma)\rightarrow \Lambda_2$ a spin structure. Then $\tilde{I}^2$ factors through the inclusion $M_2\rightarrow \Lambda_2$.
\end{remark}

\begin{proof}[Proof of \Cref{thm3}]
Consider the morphism of dg-algebras $\varphi:k[t_1]\xrightarrow{t_1\rightarrow -t_1} k[t_1]$, the pullback functor $\varphi^*:\on{dgMod}(k[t_1])\rightarrow \on{dgMod}(k[t_1])$ and \Cref{ifn}. We further denote by $T_{D_3}:D_3\rightarrow D_3$ the dg-functor determined by
\begin{align*}
x,y,z\mapsto       & ~ y,z,x\,,\\
a,b,c\mapsto       & ~ b,c,a\,,\\
a^*,b^*,c^*\mapsto       & ~ -b^*,-c^*,-a^*\,.
\end{align*}
We define a functor
\[ \mathcal{Q}:M_2^{\on{op}}\rightarrow \on{dgCat}_k\] 
by
\begin{align*}
\mathcal{Q}([2])&=D_3\,, \\
\mathcal{Q}([1])&=k[t_1]\,, \\
\mathcal{Q}(\tau^2)&=T_{D_3}\,, \\
\mathcal{Q}(\tau^1)&=\varphi^*\,, \\
\mathcal{Q}(\delta_1)&=i_1^+\,, \\
\mathcal{Q}(\delta_2)&=i_2^-\,, \\
\mathcal{Q}(\delta_3)&=i_3^+\,. 
\end{align*}
Let $\mathcal{T}$ be an ideal triangulation and $\Gamma_\mathcal{T}$ the dual ribbon graph. Given a spin structure $U$ on $\Gamma_\mathcal{T}$ described in terms of $\tilde{I}^2:\on{Exit}(\Gamma_\mathcal{T})\rightarrow M_2$ we obtain a parametrized perverse schober $\mathcal{F}_\mathcal{T}^U$ defined via its left dual
\[ \mathbb{D}^L\mathcal{F}^U_\mathcal{T}:\on{Entry}(\Gamma_\mathcal{T})\xrightarrow{(\tilde{I}_2)^{\on{op}}} M_2^{\on{op}}\xrightarrow{\mathcal{Q}}\on{dgCat}_k\xrightarrow{L} \on{dgCat}_k[W^{-1}]\xrightarrow{\mathcal{D}(\mhyphen)} \on{St}\,.\]
Given an equivalence of spin structures $U\simeq U'$, expressed in terms of an equivalence of ribbon graph with spin structure $(\Gamma_\mathcal{T},\tilde{I}^2)\simeq(\Gamma_\mathcal{T},(\tilde{I}^2)')$, it is immediate from the construction that there is an equivalence between $\mathcal{F}^U_\mathcal{T}$ and $\mathcal{F}^{U'}_{\mathcal{T}}$. For the converse, we note that by a standard result spin structures on $\Sigma$ are classified by $\on{H}^1(\Sigma,\mathbb{Z}_2)$. Thus, if two spin structures structures differ, then there exists an embedded circle in $\Gamma_\mathcal{T}$ such that the restrictions of the spin structures to the full ribbon subgraph spanned by the vertices on the circle differ. Up to equivalence of spin structures, these two spin structures can be assumed to be identical, except that they assign to the incidence of some fixed edge lying on the circle and some fixed vertex two morphisms $[1]\rightarrow [2]$ in $\Lambda_2$ which differ exactly by precomposition with $\tau^1$. The corresponding restrictions of the parametrized perverse schobers thus differ at this subgraph in terms of their monodromy by $\mathcal{D}(\phi^*)$. Note that the assumption $\on{char}(k)\neq 2$ implies $\mathcal{D}(\phi^*)\neq \on{id}_{\mathcal{D}(k[t_1])}$. The two parametrized perverse schobers can thus not be equivalent, showing the converse implication.

Let $\mathcal{T}$ be an ideal triangulation and consider a parametrized perverse schober $\mathcal{F}_\mathcal{T}$ as constructed in \Cref{thm2}. In the following we describe a spin structure $U$ on $\Gamma_\mathcal{T}$ such that $\mathcal{F}_\mathcal{T}^U\simeq \mathcal{F}_\mathcal{T}$. We follow the iterative procedure and notation used in the proof of \Cref{thm2}. Starting with any ideal triangle $\mathcal{T}_1$ of $\mathcal{T}$, one can directly find a spin structure $U$ on $\Gamma_{\mathcal{T}_1}$ such that $\mathcal{F}_{\mathcal{T}_1}\simeq \mathcal{F}_{\mathcal{T}_1}^U$. We continue by extending the spin structure in each induction step by gluing. Consider an ideal triangulation $\mathcal{T}_{n+1}$, obtained from gluing two ideal triangulations $\mathcal{T}_n$ and $\mathcal{T}_1$ equipped with spin structures. There is a spin structure $\tilde{I}_{n+1}^2:\on{Exit}(\Gamma_{\mathcal{T}_{n+1}})\rightarrow M_2$ on $\Gamma_{\mathcal{T}_{n+1}}$ with the property that the restriction to $\on{Exit}(\Gamma_{\mathcal{T}_{n}})$ and $\on{Exit}(\Gamma_{\mathcal{T}_1})$ recovers the respective spin structures of $\Gamma_{\mathcal{T}_n}$ and $\Gamma_{\mathcal{T}_1}$. We now modify $\tilde{I}_{n+1}^2$ to obtain the desired spin structure $U$  on $\Gamma_{\mathcal{T}_{n+1}}$. At each edge $e$ used in the gluing of $\mathcal{T}_{n}$ and $\mathcal{T}_1$ and connecting two vertices $v_n$ and $v_1$ in $\Gamma_{\mathcal{T}_n}$ and $\Gamma_{\mathcal{T}_1}$, respectively, the spin structure $\tilde{I}^2_{n+1}$ is locally of the form 
\[ [2]=\tilde{I}^{2}_{n+1}(v_1)\xleftarrow{\delta^i(\tau^1)^{l_1}}[1]=\tilde{I}^{2}_{n+1}(e)\xrightarrow{\delta^j(\tau^1)^{l_2}} [2]=\tilde{I}^{2}_{n+1}(v_n)\,.\]
We obtain $U$ by changing the spin structure $I^2_{n+1}$ at each edge $e$ used in the gluing of $\Gamma_{\mathcal{T}_n}$ and $\Gamma_{\mathcal{T}_1}$ in the local picture to 
\[ [2]\xleftarrow{\delta^i(\tau^1)^{l_1+1}}[1]\xrightarrow{\delta^j(\tau^1)^{l_2}} [2]\]
if the superscripts of the associated functors $i_e^1=i_e^{\pm}$ and $i_e^n=i_e^{\pm}$, appearing in the proof of \Cref{thm2}, match. The resulting parametrized perverse schober $\mathcal{F}_{\mathcal{T}_{n+1}}^U$ is then equivalent to $\mathcal{F}_{\mathcal{T}_{n+1}}$. This shows statement 2, thus completing the proof.
\end{proof}

\subsection{Coefficients in spectra}\label{specsec}

Let $\mathcal{T}$ be an ideal triangulation of an oriented marked surface and consider the construction of the parametrized perverse schober $\mathcal{F}_\mathcal{T}$ of \Cref{thm2}. Given any stable $\infty$-category $\mathcal{D}$, we can construct a parametrized perverse schober $\mathcal{F}_\mathcal{T}(\mathcal{D})$ by replacing the spherical adjunction at each vertex with the spherical adjunction 
\[ f^*:\mathcal{D}\longleftrightarrow \on{Fun}(S^2,\mathcal{D}):f_*\]
and the autoequivalence $T=T_{\on{Fun}(S^2,\mathcal{D}(k))}[2]$ (used there for fixing the correct signs) by the two-fold suspension of the cotwist functor $T_{\on{Fun}(S^2,\mathcal{D})}[2]$ of the same adjunction $f^*\dashv f_*$ but with values in $\mathcal{D}$ instead of $\mathcal{D}(k)$. The goal of this sections is to discuss which results of this paper translate to this more general setting.

We begin by noting that the main result which works irrespective of the choice of $\mathcal{D}$ is the construction of the derived equivalence
\[ \mu_e:\mathcal{H}(\Gamma_\mathcal{T},\mathcal{F}_\mathcal{T}(\mathcal{D}))\simeq \mathcal{H}(\Gamma_{\mathcal{T}'},\mathcal{F}_{\mathcal{T}'}(\mathcal{D}))\]  
of \Cref{sec6.4} associated to the flip of the edge $e$ of the triangulation $\mathcal{T}$.  

More results extend in the case where $\mathcal{D}=\on{RMod}_R$ is the $\infty$-category of right modules of an $\mathbb{E}_\infty$-ring spectrum $R$. We can for example choose $R$ to be the sphere spectrum, so that $\mathcal{D}\simeq \on{Sp}$ is the $\infty$-category of spectra. The construction of the spherical/exceptional objects $S_e$, the objects $C_p$ and the projective objects $P_e$ given in \Cref{sec6.3} works in the same way for $\mathcal{H}(\Gamma_\mathcal{T},\mathcal{F}_\mathcal{T}(\on{RMod}_R))$. We denote by $R[t_n]$ the free algebra object in $\on{RMod}_R$ generated by $R[n]$. Note that if $R=k$ is a commutative ring, there exists an equivalence $R[t_n]\simeq k[t_n]$. We show below in \Cref{2genlem4} that there exists an equivalence of $\infty$-categories $\on{Fun}(S^n,\on{RMod}_R)\simeq \on{RMod}_{R[t_{n-1}]}$. We thus find a compact generator of the $\infty$-category $\mathcal{H}(\Gamma_\mathcal{T},\mathcal{F}_\mathcal{T}(\on{RMod}_R))$, as in \Cref{glgenprop}, given by the sum of the images of $R[t_1]\in \on{RMod}_{R[t_1]}\simeq \on{Fun}(S^2,\on{RMod}_R)$ under the left adjoints of the evaluation functors to the edges of $\Gamma_\mathcal{T}$. Furthermore, \Cref{mutprop1} describing the images of the $P_e$ under the derived equivalences also translates. Contrary to the dg-setting, it is however not clear if the endomorphism algebra of the compact generator admits an explicit description. 

Crucial for a relation between the construction of $\mathcal{F}_\mathcal{T}(\mathcal{D}(k))$ and spin structures on the surface without the interior marked points $\Sigma$ in \Cref{spinsec} is the observation that the suspended cotwist functor $T_{\on{Fun}(S^2,\mathcal{D}(k))}[2]$ is an involution, i.e.
\[ \big( T_{\on{Fun}(S^2,\mathcal{D}(k))}[2]\big)^2\simeq \on{id}_{\on{Fun}(S^2,\mathcal{D}(k))}\,.\]
It seems likely that if the cotwist functor $T_{\on{Fun}(S^2,\mathcal{D})}[2]$ is also an involution \Cref{thm3} can be generalized to $\mathcal{F}_\mathcal{T}(\mathcal{D})$. The remainder of this section consists of a proof of \Cref{2genlem4} and a conjecture for a description of $T_{\on{Fun}(S^2,\mathcal{D})}$ and an algebraic description of $T_{\on{Fun}(S^2,\on{RMod}_R)}$. 

The following lemma is a generalization of \Cref{l1lem}.

\begin{lemma}\label{l1lem2}
Consider the morphism of simplicial sets $g:L\rightarrow \ast$ and the associated pullback functor $g^*:\on{RMod}_R\rightarrow \on{Fun}(L,\on{RMod}_R)$. There exists an equivalence of $R$-linear $\infty$-categories 
\[ \on{Fun}(L,\on{RMod}_R)\simeq \on{RMod}_{R[t_0]}\] 
such that the following diagram commutes.
\begin{equation}\label{2tri1}
\begin{tikzcd}
                                            & \on{RMod}_R \arrow[ld, "g^*"'] \arrow[rd, "\phi^*"] &                      \\
{\on{Fun}(L,\on{RMod}_R)} \arrow[rr, "\simeq"] &                                                     & {\on{RMod}_{R[t_0]}}
\end{tikzcd}
\end{equation}
Here $\phi^*$ denotes the pullback functor along the morphism of $R$-algebras $R[t_0]\rightarrow R$, determined on the generator by the morphism $R\xrightarrow{\on{id}}R$ in $\on{RMod}_R$.
\end{lemma}

\begin{proof}
Consider the object $X\in \on{Fun}(L,\on{RMod}_R)$ given by the diagram $R[t_0]\xrightarrow{\cdot  t_0}R[t_0]$ in $\on{RMod}_R$. Let $h:\ast \rightarrow L$ be the morphism of simplicial sets given by inclusion of the unique vertex and consider the associated pullback functor $h^*:\on{RMod}_R\rightarrow \on{Fun}(L,\on{RMod}_R)$ with right adjoint $h_*$ given by evaluation at $\ast \in L$. We prove that $X\simeq h^*(R)$ by showing that $\on{Mor}_{\on{Fun}(L,\on{RMod}_R)}(X,\mhyphen)\simeq h_*$, where $\on{Mor}_{\on{Fun}(L,\on{RMod}_R)}(\mhyphen,\mhyphen)$ is the $R$-linear morphism object functor of \cite[4.2.1.28]{HA}.

Let $Y\in\on{Fun}(L,\on{RMod}_R)$. The morphism object $\on{Mor}_{\on{Fun}(L,\on{RMod}_R)}(X,Y)$ is equivalent to the equalizer 
\[ 
\begin{tikzcd}
\prod_{i\in \mathbb{N}}h_*(Y) \arrow[r, "\simeq"] & {\on{Mor}_{R}(R[t_0],Y)} \arrow[r, "(\mhyphen)\circ t_0", shift left] \arrow[r, "Y(l)\circ (\mhyphen)"', shift right] & {\on{Mor}_{R}(R[t_0],Y)} \arrow[r, "\simeq"] & \prod_{i\in \mathbb{N}}h_*(Y)
\end{tikzcd}\,,
\]
where $l$ is the unique non-degenerate $1$-simplex of $L$. This can be seen as follows. Consider the simplicial set $L'$ consisting of four vertices $x_1,x_2,x_3,x_4$ and four non-degenerate $1$-simplicies $l_1,l_2,l_3,l_4$ arranged as follows.
\[
\begin{tikzcd}
x_1 \arrow[r, "l_1"'] \arrow[d, "l_4"] & x_3                                    \\
x_4                                    & x_2 \arrow[u, "l_2"] \arrow[l, "l_3"']
\end{tikzcd}
\] 
The morphism of simplicial sets $p:L'\rightarrow L$, mapping all vertices to $\ast \in L$, $l_1$ to $l$ and $l_2,l_3,l_4$ to the degenerate $1$-simplex, induces a fully faithful $R$-linear functor $p^*:\on{Fun}(L,\on{RMod}_R)\rightarrow \on{Fun}(L',\on{RMod}_R)$. The description of $\on{Mor}_{\on{Fun}(L,\on{RMod}_R)}(X,Y)\simeq \on{Mor}_{\on{Fun}(L',\on{RMod}_R)}(p^*(X),p^*(Y))$ as an equalizer can now be obtained by using a pushout description of $p^*(X)\simeq X_1\amalg_{X_3}X_2$, with 
\[ 
X_1=\begin{tikzcd}[row sep=small, column sep=small]
{R[t_0]} \arrow[r, "\simeq"'] \arrow[d, "\simeq"] & {R[t_0]}              \\
{R[t_0]}                                          & 0 \arrow[u] \arrow[l]
\end{tikzcd}\,,\quad 
X_2=\begin{tikzcd}[row sep=small, column sep=small]
0 \arrow[r] \arrow[d] & {R[t_0]}                                          \\
{R[t_0]}              & {R[t_0]} \arrow[u, "\simeq"] \arrow[l, "\simeq"']
\end{tikzcd}\,,\quad
X_3=\begin{tikzcd}[row sep=small, column sep=small]
0 \arrow[r] \arrow[d] & {R[t_0]}              \\
{R[t_0]}              & 0 \arrow[u] \arrow[l]
\end{tikzcd}\,
\] 
and that $\on{Mor}_{\on{Fun}(L',\on{RMod}_R}(\mhyphen,p^*(Y))$ is an exact functor. The equalizer is given by $h_*(Y)$, the morphism of $R$-modules $h_*(Y)\rightarrow \prod_{i\in \mathbb{N}} h_*(Y)$ is informally given by mapping $z\in h_*(Y)$ to $(Y(l)^i(z))_{i\in \mathbb{N}}\in \prod_{i\in \mathbb{N}} h_*(Y)$. We note that the equivalence $\on{Mor}_{\on{Fun}(L,\on{RMod}_R)}(X,Y)\simeq h_*(Y)$ is functorial in $Y$, so that indeed $\on{Mor}_{\on{Fun}(L,\on{RMod}_R)}(X,\mhyphen)\simeq h_*$. 

It follows that $X$ is a compact generator of $\on{Fun}(L,\on{RMod}_R)$. Applying \cite[Proposition 4.1.1.18]{HA}, we further obtain an equivalence of $R$-linear ring spectra $\on{End}_R(X)\simeq R[t_0]$, showing the existence of an equivalence of $R$-linear $\infty$-categories $\on{Fun}(L,\on{RMod}_R)\simeq \on{RMod}_{R[t_0]}$.

The commutativity of the diagram \eqref{2tri1} can be checked using the fact that the $R$-linear functors $\phi^*,g^*:\on{RMod}_R\rightarrow \on{RMod}_{R[t_0]}$ are fully determined by $\phi^*(R)$, respectively $g^*(R)$, see \cite[Section 4.8.4]{HA}.
\end{proof}

\begin{lemma}\label{pbpulem2}~
\begin{enumerate}
\item There exists a pushout diagram in $\on{LinCat}_R$ as follows.
\begin{equation}\label{2pupbeq1}
\begin{tikzcd}
\on{Fun}(L,\on{RMod}_R)        \arrow[r, "g_!"] \arrow[d, "g_!"'] \arrow[rd, "\ulcorner", phantom, near end] & \on{RMod}_R \arrow[d, "i_!"] \\
\on{RMod}_R \arrow[r, "i_!"']                                               & \on{Fun}(S^2,\on{RMod}_R)              
\end{tikzcd}
\end{equation}
\item Let $n\geq 2$. There exists a pushout diagram in $\on{LinCat}_R$ as follows.
\begin{equation}\label{2pupbeq2}
\begin{tikzcd}
\on{Fun}(S^{n-1},\on{RMod}_R) \arrow[r, "f_!"] \arrow[d, "f_!"'] \arrow[rd, "\ulcorner", phantom, near end] & \on{RMod}_R\arrow[d, "i_!"] \\
\on{RMod}_R \arrow[r, "i_!"']                                               & \on{Fun}(S^n,\on{RMod}_R)              
\end{tikzcd}
\end{equation}
\end{enumerate}
\end{lemma}

\begin{proof}
The proof of \Cref{pbpulem} directly generalizes.
\end{proof}

We now prove the analogue of \Cref{genlem}. 

\begin{proposition}\label{2genlem4}
Let $n\geq 2$. There exists an equivalence of $R$-linear $\infty$-categories 
\begin{equation}\label{eq77}
\on{Fun}(S^n,\on{RMod}_R)\simeq \on{RMod}_{R[t_{n-1}]}\,,
\end{equation}
such that the following diagram in $\on{LinCat}_R$ commutes. 
\begin{equation}\label{eqdiag14}
\begin{tikzcd}
{\on{Fun}(S^n,\on{RMod}_R)} \arrow[rr, "\simeq"] \arrow[rd, "i^*"] &                                                     & {\on{RMod}_{R[t_{n-1}]}} \arrow[ld, "G"] \\
                                                                           & \on{RMod}_R \arrow[ld, "f^*"'] \arrow[rd, "\phi^*"] &                                          \\
{\on{Fun}(S^n,\on{RMod}_R)} \arrow[rr, "\simeq"]                   &                                                     & {\on{RMod}_{R[t_{n-1}]}}                
\end{tikzcd}
\end{equation}
Here $G$ denoted the monadic functor and $\phi^*$ the pullback functor along the morphism of $R$-algebras $\phi:R[t_{n-1}]\rightarrow R$ determined on the generator by the morphism $R[n-1]\xrightarrow{0}R$ in $\on{RMod}_R$. 
\end{proposition}

\begin{proof}
Consider the following biCartesian square in $\on{RMod}_R$. 
\[
\begin{tikzcd}
R[n-1] \arrow[r] \arrow[d] \arrow[rd, "\square", phantom] & 0 \arrow[d] \\
0 \arrow[r]                                 & {R[n]}     
\end{tikzcd}
\] 
Applying the colimit preserving free $R$-algebra functor $\on{RMod}_R\rightarrow \on{Alg}(\on{RMod}_R)$ yields the following pushout diagram of $R$-algebras. 
\begin{equation}\label{Rpu1}
\begin{tikzcd}
{R[t_{n-1}]} \arrow[r, "t_{n-1}\mapsto 0"] \arrow[d, "t_{n-1}\mapsto 0"'] \arrow[rd, "\ulcorner", phantom, near end] & R \arrow[d] \\
R \arrow[r]                                              & {R[t_n]}   
\end{tikzcd}  
\end{equation}
Consider the morphism of ring spectra $R[t_0]\xrightarrow{t_0\mapsto t_0+1} R[t_0]$, determined from the universal property by the morphism $R\xrightarrow{1\mapsto 1+t_0} R[t_0]$ in $\on{RMod}_R$. Using the commutativity of the diagram
\[ 
\begin{tikzcd}
{R[t_0]} \arrow[r, "t_0\mapsto 1"] \arrow[d, "t_0\mapsto t_0+1"'] & R \\
{R[t_0]} \arrow[ru, "t_0\mapsto 0"']                              &  
\end{tikzcd}
\]
it follows that for $n=1$ the image of the diagram \eqref{Rpu1} under $\theta:\on{Alg}(\on{RMod})\rightarrow \on{LinCat}_R$ is equivalent to the pushout diagram in \eqref{2pupbeq1}. It follows that there exists an equivalence of $R$-linear $\infty$-categories $\on{Fun}(S^2,\on{RMod}_R)\simeq \on{RMod}_{R[t_1]}$. Using that the monadic functor $G$ is equivalent to the pullback along $R\rightarrow R[t_1]$, we obtain that the upper half of the diagram \eqref{eqdiag1} commutes. 

Using that $f_!\circ i_!\simeq \on{id}_{\on{RMod}_R}$, see \Cref{sphrem}, we obtain the following commutative diagram.  

\begin{equation}\label{2pudiag1}
\begin{tikzcd}[column sep=small]
{\on{Fun}(L,\on{RMod}_R)} \arrow[dr, phantom, near end, "\ulcorner"]\arrow[r, "g_!"] \arrow[d, "g_!"] & \on{RMod}_R \arrow[d, "i_!"] \arrow[rdd, "\on{id}", bend left] &             \\
\on{RMod}_R \arrow[r, "i_!"] \arrow[rrd, "\on{id}", bend right=18]  & {\on{Fun}(S^2,\on{RMod}_R)} \arrow[rd, "f_!"]             &             \\
                                                            &                                                           & \on{RMod}_R
\end{tikzcd}
\end{equation}
The diagram \eqref{2pudiag1} is equivalent to the image under $\theta$ of the following diagram in $\on{Alg}(\on{RMod}_R)$.
\begin{equation}\label{2pudiag2}
\begin{tikzcd}
R[t_0] \arrow[r] \arrow[d]  \arrow[dr, phantom, near end, "\ulcorner"]             & R \arrow[d] \arrow[rdd, bend left , "\on{id}"] &   \\
R \arrow[r] \arrow[rrd, bend right, "\on{id}"] & {R[t_1]} \arrow[rd]                  &   \\
                                    &                                    & R
\end{tikzcd}
\end{equation}
By the universal property of the pushout in $\on{Alg}(\on{RMod}_R)$ there exists a unique morphism of ring spectra $R[t_1]\rightarrow R$ such that \eqref{2pudiag2} commutes. Such a map is given by $\phi$. It follows that the functor $f_!$ is equivalent to $\theta(\phi)$ and using \cite[4.6.2.17]{HA} also that the functor $f^*$ is equivalent to the pullback functor along $\phi$.

For $n\geq 2$, we can continue by induction and as before. The image of \eqref{Rpu1} under the functor $\theta$ is the pushout diagram in \eqref{2pupbeq2}. We thus find the desired equivalence $\on{Fun}(S^n,\on{RMod}_R)\simeq \on{RMod}_{R[t_{n-1}]}$ so that the upper half of diagram \eqref{eqdiag14} commutes. Analogous to the case $n=1$, it can be checked that the lower half of the diagram \eqref{eqdiag14} commutes.
\end{proof}

Our proof of \Cref{ctwprop} characterizing the cotwist functor of the spherical adjunction $f^*\dashv f_*$ does not directly generalize to the spectral setting. We conjecture the following. 

\begin{conjecture}\label{conj2}
For any stable $\infty$-category $\mathcal{D}$ and $n\geq 2$, consider the cotwist functor $T_{\on{Fun}(S^{n},\mathcal{D})}$ of the spherical adjunction $f^*:\mathcal{D}\leftrightarrow \on{Fun}(S^n,\mathcal{D}):f_*$, see \Cref{locprop1}.
\begin{enumerate}
\item Let $R$ be an $\mathbb{E}_\infty$-ring spectrum and $\mathcal{D}=\on{RMod}_R$. Let $\varphi:R[t_{n-1}]\rightarrow R[t_{n-1}]$ be the equivalence of ring spectra determined by $\varphi(t_{n-1})=(-1)^{n-1}t_{n-1}$ (via the involved universal properties). There exists a commutative diagram in $\on{LinCat}_R$ as follows.
\[
\begin{tikzcd}[column sep=60]
{\on{Fun}(S^n,\on{RMod}_R)} \arrow[r, "{T_{\on{Fun}(S^n,\on{RMod}_R)}}"] \arrow[d, "{\eqref{eq77}}"']
\arrow[d, "\simeq"] & {\on{Fun}(S^n,\on{RMod}_R)} \arrow[d, "{\eqref{eq77}}"']\arrow[d, "\simeq"]\\
{\on{RMod}_{R[t_{n-1}]}} \arrow[r, "{\varphi^*[-n]}"]                    & {\on{RMod}_{R[t_{n-1}]}}                      
\end{tikzcd}
\]
\item Denote by $S^n_{\on{top}}$ the topological $n$-sphere embedded as the unit-sphere into $\mathbb{R}^{n+1}$, so that its singular set is given by $\on{Sing}(S^n_{\on{top}})=S^n$. Let $r:S^n_{\on{top}}\rightarrow S^n_{\on{top}}$ be the antipodal map, mapping $x\in S^n_{\on{top}}$ to $-x\in S^n_{\on{top}}$ and $r^*$ its pullback functor. There exists a natural equivalence 
\[T_{\on{Fun}(S^n,\mathcal{D})}\simeq r^*[-n]\,.\] 
\end{enumerate}
\end{conjecture}

\bibliography{biblio} 
\bibliographystyle{alpha}

\textsc{Fachbereich Mathematik, Universität Hamburg, Bundesstraße 55, 20146 Hamburg, Germany}

\textit{Email address:} \texttt{merlin.christ@uni-hamburg.de}
\end{document}